\documentclass[a4paper,oneside,reqno]{amsart}
\usepackage[utf8]{inputenc}

\usepackage{amssymb,amsthm,amsmath}
\usepackage{bbm}
\usepackage{enumitem}
\usepackage{hyperref}
\usepackage[capitalize]{cleveref}
\usepackage{mathtools}
\usepackage{stmaryrd} 
\usepackage{thmtools}
\usepackage{thm-restate}
\usepackage{enumitem}
\usepackage{algorithm}
\usepackage{algpseudocode}
\usepackage{comment}
\usepackage[toc,page]{appendix}

\usepackage{geometry}
\declaretheorem[numberwithin=section]{theorem}
\declaretheorem[sibling=theorem]{lemma}
\declaretheorem[sibling=theorem]{claim}
\declaretheorem[sibling=theorem]{fact}

\declaretheorem[sibling=theorem]{corollary}

\declaretheorem[sibling=theorem]{observation}
\declaretheorem[style=definition,sibling=theorem]{definition}
\declaretheorem[style=remark,sibling=theorem]{remark}

\renewcommand{\le}{\leqslant}
\renewcommand{\ge}{\geqslant}
\renewcommand{\leq}{\leqslant}
\renewcommand{\geq}{\geqslant}

\renewcommand{\Pr}{\mathbb{P}}
\renewcommand{\emptyset}{\varnothing}

\newcommand{\cA}{\mathcal{A}}

\newcommand{\cD}{\mathcal{D}}
\newcommand{\cE}{\mathcal{E}}
\newcommand{\cF}{\mathcal{F}}

\newcommand{\cH}{\mathcal{H}}

\newcommand{\cL}{\mathcal{L}}
\newcommand{\cP}{\mathcal{P}}

\newcommand{\cR}{\mathcal{R}}
\newcommand{\cS}{\mathcal{S}}
\newcommand{\cT}{\mathcal{T}}

\newcommand*\diff{\mathop{}\!\mathrm{d}}
\newcommand{\eps}{\varepsilon}

\DeclareMathOperator{\Aut}{Aut}

\DeclareMathOperator{\UB}{UB}
\DeclareMathOperator{\LB}{LB}

\renewcommand{\Pr}{\mathbb{P}}
\newcommand{\E}{\mathbb{E}}

\usepackage{xcolor}

\author{Asaf Cohen Antonir}
\address{School of Mathematical Sciences, Tel Aviv University, Tel Aviv 6997801, Israel}
\email{asafc1@tauex.tau.ac.il}

\author{Ilay Hoshen}
\address{School of Mathematical Sciences, Tel Aviv University, Tel Aviv 6997801, Israel}
\email{ilayhoshen@gmail.com}

\author{Maksim Zhukovskii}
\address{Department of Computer Science, University of Sheffield, Sheffield S1 4DP, UK}
\email{m.zhukovskii@sheffield.ac.uk}

\title{A Local Central Limit Theorem for Clique Counts in Sparse Random Graphs}

\begin{document}

\maketitle
\begin{abstract}
    Let $X_H$ denote the number of copies of a fixed graph $H$ in $G_{n, p}$. Gilmer and Kopparty conjectured that $X_H$ satisfies a local central limit theorem (LCLT) provided that $H$ is connected, $p \gg n^{-1/m(H)}$, and $n^2 (1-p) \gg 1$, where $m(H)$ is the maximum density. 
    
    Following the work of Berkowitz, Sah and Sawhney confirmed this conjecture for every constant $p$, leaving the regime where $p=o(1)$ open. In this regime, the only case addressed in the literature is when $H=K_3$, where, in a recent paper, Araújo and Mattos confirmed the conjecture for $p \in (4n^{-1/2}, 1/2)$. This, together with a general result of Röllin and Ross, essentially settles the conjecture for the triangle. We generalise these results by showing that an LCLT holds for $H = K_r$ (for any fixed $r \ge 3$) in the regime $n^{-1/m(H)}\ll p\leq 1/2$, essentially settling the conjecture for cliques.
\end{abstract}

\section{Introduction}\label{section: introduction}

Throughout this paper, we use standard graph theory notation: For a graph $G$, its vertex set is denoted by $V(G)$, its edge set is denoted by $E(G)$, and their cardinalities are denoted by $v(G)$ and $e(G)$, respectively.
The binomial random graph $G_{n,p}$ is obtained by independently retaining each edge of $K_n$, the complete graph on the vertex set $[n]\coloneqq \{1,\ldots ,n\}$, with probability $p$.  
This paper is concerned with the random variable $X_H$ counting the number of labelled copies of $H$ in $G_{n,p}$.
The study of this random variable dates back to the seminal works of Erd\H{o}s and R\'enyi~\cite{ErdRen1960} and of Bollob\'{a}s~\cite{Bol1981}, which established the following threshold phenomenon:
For every fixed graph $H$, upon setting $m(H)\coloneqq \max\{e(J)/v(J):\emptyset \neq J\subseteq H\}$ we have
\begin{equation*}
    \lim_{n\to \infty} \mathbb{P}(X_H\neq 0) = \begin{cases}0& \text{if}\ \  p\ll n^{-{1}/{m(H)}},\\
    1& \text{if}\ \  p\gg n^{-1/{m(H)}}.\end{cases}
\end{equation*}
When $H$ is strictly balanced \footnote{A graph $H$ is called \emph{strictly balanced} if $H$ is the unique $J\subseteq H$ attaining the maximum in the definition of $m(H)$.} and $p=\lambda n^{-1/m(H)}$ where $\lambda>0$ is some constant, Bollob\'{a}s \cite{Bol1981} proved that $X_H$ converges in distribution to a Poisson random variable with mean $\lambda ^{e(H)}/\lvert \Aut(H) \rvert$. 

For $p\gg n^{-1/m(H)}$ and $n^2(1-p)\gg 1$, in a sequence of works \cite{Kar1984,KarRuc1983,NowWie1986} culminating in the celebrated result of Ruci\'nski~\cite{Ruc1988}, it was shown that $X_H$ satisfies a central limit theorem. Formally, letting $\mu_H\coloneqq \E[X_H]$ and $\sigma_H^{2}\coloneqq \mathbb{V}ar[X_H]$ Ruci\'nski showed that for every fixed real $x$ we have
\[
    \lim_{n\to \infty }\Pr\left(\frac{X_H-\mu_H}{\sigma_H}\leq x\right) = \Pr\left(Z\leq x\right),
\]
where $Z$ is the standard normal random variable.
In a subsequent paper, Barbour, Karo\'nski, and Ruci\'nski \cite{BarKarRuc1989} derived a quantitative version of the central limit theorem for `decomposable' random variables, which include $X_H$. Another influential result in the study of small subgraph counts due to Janson \cite{Jan1990} used the $U$-statistics technique, from the asymptotic theory of statistics. 
This result establishes a central limit theorem for the joint distribution of $(X_{H_1},\ldots, X_{H_k})$, where $H_1,\ldots ,H_k$ are fixed graphs.

The next natural step is to refine the `integral' version of the central limit theorem of Ruci\'nski~\cite{Ruc1988} by proving a \emph{local central limit theorem} (LCLT). More precisely, is it true that
\[
    \Pr(X_H=x) = \frac{1}{\sqrt{2\pi}\cdot \sigma_H}\exp\left({-\frac{(x-\mu_H)^2}{2\sigma_H^2}}\right) + o (\sigma_H^{-1})?
\]
It was conjectured by Gilmer and Kopparty \cite{GilKop2016} that for any connected graph $H$, as long as $X_H$ admits a central limit theorem, it should admit an LCLT.  
We remark that for disconnected graphs, there are examples where a central limit theorem holds, whereas an LCLT fails; see \cite{SahSaw2022} or \cite{BerSahSaw2021} for a failure of a local central limit in a different but related model. 

In this paper, we resolve the Gilmer-Kopparty conjecture when $H$ is a clique in an essentially optimal range of edge probabilities $p$.
To state our main result, let us introduce some notations. 
For a graph $H$ its $2$-density is $m_2(H)\coloneqq \max \left\{\frac{e(J)-1}{v(J)-2}:J\subseteq H,v(J)>2\right\}\cup \left\{\frac{1}{2}\right\}$ and in particular, letting $r\geq 3$ be an integer we have $m_2(K_r)=\frac{r+1}{2}$ whereas $m(K_r)=\frac{r-1}{2}$.
In addition, we abbreviate $X_{K_r},\mu_{K_r},$ and $\sigma_{K_r}$ by writing $X_r,\mu_r,$ and $\sigma_r$, respectively. Lastly, set $X^*_r\coloneqq \frac{1}{\sigma_r}(X_r-\mu_r)$ and set $\cL_r$ to be the support of $X_r^*$. Our main theorem reads as follows.

\begin{theorem}\label{theorem:main theorem}
    Suppose that $r\geq 3$ is an integer. For any $n^{-1/m(K_r)}\ll p \leq \frac{1}{2}$ we have
    \begin{equation*}
        \Phi _r(n,p) \coloneqq \sup_{x\in \mathcal{L}_r}\left\lvert\sigma_r \cdot \mathbb{P}(X^*_r= x) - \frac{1}{\sqrt{2\pi}}e^{-x^2/2}\right\rvert = o(1).
    \end{equation*}
    Quantitatively, for any $c,\epsilon>0$ with $c<\frac{3}{2r}$, we have
    \begin{equation}\label{eq:main}
        \Phi_r (n,p) =\begin{cases}
	    O\left(\frac{1}{n^{1-\epsilon}\sqrt{p}}\right)&\text{if } p\geq n^{-c},\\
        O\left(n^{\epsilon}p^{3/2}\right)&\text{if } n^{-\frac{1}{m_2(K_r)}}(\log n)^{\frac{4r}{e(K_r)-1} }\leq p< n^{-c}	,\\
        O\left(\frac{(\log n)^{46r}}{\sqrt{n^{r}p^{\binom{r}{2}}}}\right)&\text{if } n^{-\frac{1}{m(K_r)}}(\log n)^{128}\leq p< n^{-\frac{1}{m_2(K_r)}}(\log n)^{\frac{4r}{e(K_r)-1} }, \\
        O\left( \frac{\log \left(n^{r}p^{\binom{r}{2}}\right)}{\sqrt{n^{r}p^{\binom{r}{2}}}}\right)&\text{if } n^{-\frac{1}{m(K_r)}}\ll  p<  n^{-\frac{1}{m(K_r)}}(\log n)^{128}.
        \end{cases}
    \end{equation}
\end{theorem}

\begin{remark}
    We did not try to optimise the quantitative part of Theorem~\ref{theorem:main theorem}. In particular, our proof likely yields that the polynomials $n^{\epsilon}$ can be replaced by poly-logarithmic factors, and moreover, the powers in the poly-logarithmic factors can certainly be reduced. However, in the regime where $p=\Theta(1)$, the (unoptimised) error term in \eqref{eq:main} improves upon previous works \cite{Ber2018,SahSaw2022}.
\end{remark}

This theorem continues a line of work establishing local central limits for subgraph counts in various setups. To the best of our knowledge, the first result of this type was obtained by Gilmer and Kopparty \cite{GilKop2016}, who considered triangle counts for $p$ bounded away from zero and one and, in the same work, formulated the above-mentioned conjecture. Then, R\"{o}llin and Ross \cite{RolRos2015} proved a local limit theorem in a general framework, and as a consequence, obtained an LCLT for triangle counts above the appearance threshold $\Omega(n^{-1})$ and below $O\left(n^{-1/2}\right)$. Recently, Ara\'{u}jo and Mattos \cite{AraMat2023} closed the gap between the sparse regime and the dense regime where $p=\Theta(1)$. This work confirmed the Gilmer--Kopparty conjecture for the triangle by establishing an LCLT for every $n^{-1}\ll p\le 1-\Theta(1)$.

In the dense regime $p=\Theta(1)$, Berkowitz \cite{Ber2016,Ber2018} improved the result of Gilmer and Kopparty by first establishing an LCLT for triangle counts and then for clique counts with a quantitative control over the error terms. These works introduced the decoupling method, which was central to all subsequent works \cite{AraMat2023, SahSaw2022, SahSawZhu2024}, including this one. This method is discussed both in the overview and in Section \ref{section:decoupling}. 

Sah and Sawhney \cite{SahSaw2022} extended Berkowitz's theorem and proved an LCLT for every connected graph $H$ and every constant $p$, again, with a quantitative control over the error terms. This result was further refined by Sah, Sawhney, and Zhu \cite{SahSawZhu2024}, who recently established a local central limit for $(X_{H_1},\ldots,X_{H_k})$ where $H_1,\ldots ,H_k$ are fixed connected graphs and $p=\Theta(1)$.

Before giving some details on the proof of Theorem \ref{theorem:main theorem}, we briefly discuss another major topic in combinatorics and probability theory that has connections to local (central) limit theorems, namely, \emph{anti-concentration}.
For a discrete random variable $X$, the anti-concentration question asks to determine the quantity $\sup_{x\in \textup{Supp}X}\Pr(X=x)$. Clearly, an LCLT for $X$ answers this question asymptotically. In the case of subgraph counts, Fox, Kwan and Sauermann~\cite{FoxKwanSau2021b, FoxKwanSau2021a} obtained an asymptotic logarithmic optimal anti-concentration for $X_H$, when $H$ is connected, and $p$ is a constant. As was mentioned, Sah and Sawheny~\cite{SahSaw2022} obtained an LCLT in this case, thereby improving the results in \cite{FoxKwanSau2021b, FoxKwanSau2021a} and settling the problem for connected graphs in the dense regime. 
However, in the sparse regime, where $p=o(1)$, apart from the case $H=K_3$, the anti-concentration behaviour of $X_H$ remains unknown. As this work establishes a local central limit for clique counts in the sparse regime, it also provides an asymptotic answer to the anti-concentration question for these random variables in the sparse regime.

In a more general framework of polynomials of i.i.d.\ Rademacher random variables (i.e.\ uniform random variables on $\{\pm1\}$), Kwan and Sauermann \cite{KwaSau2023} made important progress towards a general anti-concentration theorem by resolving the problem for all polynomials of degree $2$. Interestingly, they used a decoupling technique similar to that of Berkowitz.

\subsection* {Proof strategy for Theorem \ref{theorem:main theorem}.}
Our proof of Theorem \ref{theorem:main theorem} begins with the classical Fourier approach used in essentially all earlier works on the matter. This approach reduces the problem to showing that the difference between the characteristic function of $X_r^*$ and that of the standard normal distribution is small. More specifically, the problem is reduced to showing that as $n\to \infty$ we have
\begin{equation}\label{eq: integral bound intro}
    \int^{\pi \sigma_r}_{-\pi \sigma_r}\left\lvert{\E\left[e^{it X_r^*}\right] -e^{-t^2/2}}\right\rvert \diff t \to 0.
\end{equation}
Moreover, the rate of decay above dictates the quality of the bound in \eqref{eq:main}. 

Restricting the integral to $[-K,K]$, where $K=K(H,n,p)$ is some `small enough' function, it is well-known that the quantitative central limit theorem in \cite{BarKarRuc1989} together with Stein's method (see e.g.\ \cite[Section 6]{Ros2011}) yield the required bound. For the sake of completeness, we reproduce these bounds in Section \ref{sec:quantitative-central-limit}. Therefore, the main challenge in proving Theorem \ref{theorem:main theorem} lies in bounding the integrand in \eqref{eq: integral bound intro} for frequencies $t\in [-\pi \sigma_r ,\pi\sigma_r ]\setminus [-K,K]$. The following lemmas establish this required bound and are the main contributions of this paper.

\begin{lemma}\label{lemma:characteristic function integrable}
    Suppose that $r\geq 3$ is an integer and that $K>0$. The following holds for every $\epsilon>0$ and every sufficiently large integer $ n$. 
    For every $n^{-\frac{1}{m_2(K_r)}}(\log n)^{\frac{4r}{e(K_r)-1}}\leq p\leq \frac{1}{2}$ and every $n^{1/2+\epsilon}p\le t \le \pi \sigma_r$, we have
    \begin{align}\label{align:characteristic function integrable}
        \left\lvert\E\left[e^{it \cdot X_r/\sigma_r}\right]\right\rvert \le t^{-K}.
    \end{align}
\end{lemma}

\begin{lemma}\label{lemma:characteristic function integrable - low freq and dense p}
    Suppose that $r\geq 3$ is an integer and that $c,K$ are positive reals with $c<\frac{3}{2r}$. The following holds for every small enough $\epsilon>0$ and every sufficiently large integer $n$. For every $n^{-c}\leq p\leq \frac{1}{2}$ and every $n^{\epsilon}\le t \le n^{1/2+\epsilon}p$, we have
    \begin{align}\label{align:characteristic function integrable - low freq and dense p}
        \left\lvert\E\left[e^{it \cdot X_r/\sigma_r}\right]\right\rvert \le t^{-K}.
    \end{align}
\end{lemma}

\begin{lemma}\label{lemma:characteristic function integrable - low freq and low p}
    For every integer $r\geq 3$ there is $\epsilon=\epsilon(r)>0$ such that the following holds. Let $K > 0$ be a constant, assume that $n$ is a sufficiently large integer, and that $p$ satisfies $n^{-1/m(K_r)} \ll p\leq n^{-\frac{1}{m_2(K_r)}}(\log n)^{\frac{4r}{e(K_r)-1}}$. Letting $\tilde{n} \le n/r$ be the largest integer satisfying $\tilde{n}^{r-2} p^{\binom{r}{2}-1} < (\log n)^{-10}$, for every $t \in [0, \pi\sigma_r]$, we have
    \begin{align}\label{align:characteristic function integrable - low freq and low p}
        \left\lvert\E\left[e^{it \cdot X_r/\sigma_r}\right]\right\rvert \le n^{-K} + \exp\left(-\frac{\epsilon\tilde{n}^r p^{\binom{r}{2}} t^2}{\sigma_r^2}\right).
    \end{align}
\end{lemma}

As we shall see, the lower bound assumptions on $t$ in Lemmas \ref{lemma:characteristic function integrable} and \ref{lemma:characteristic function integrable - low freq and dense p} match the bounds we derive from the above-mentioned Stein's method and quantitative central limit, and are sufficient for our use. For example, when $p< n^{-c}$ satisfies the assumptions in Lemma \ref{lemma:characteristic function integrable}, the integral in \eqref{eq: integral bound intro} restricted to $t\in I \coloneqq  [- n^{1/2+\epsilon}p, n^{1/2+\epsilon}p]$, tends to zero with $n$ by the Stein's method and the quantitative central limit theorem. This, combined with Lemma \ref{lemma:characteristic function integrable}, which shows that the integral in \eqref{eq: integral bound intro} restricted to $[-\sigma_r \pi,\sigma_r \pi]\setminus I$ converges to zero, establishes \eqref{eq: integral bound intro} as required. A similar situation arises also when $p\le n^{-\frac{1}{m_2(K_r)}}(\log n)^{\frac{4r}{e(K_r)-1}}$ or when $p\geq n^{-c}$, where the frequencies left uncovered by Lemmas \ref{lemma:characteristic function integrable}, \ref{lemma:characteristic function integrable - low freq and dense p}, and \ref{lemma:characteristic function integrable - low freq and low p} are dealt as above.

We wish to emphasise a great difference between the first two Lemmas \ref{lemma:characteristic function integrable}, \ref{lemma:characteristic function integrable - low freq and dense p} and the third Lemma~\ref{lemma:characteristic function integrable - low freq and low p}: \emph{the densities at which they are applicable}.
A key point highlighting the need to distinguish these cases is that 
$\sigma_r = \Theta\left(\max\left\{n^{v(K_r)}p^{e(K_r)},n^{2v(K_r)-2}p^{2e(K_r)-1}\right\}\right)$, which is maximised on a different term depending on the relationship between $p$ and the \emph{$2$-density} $n^{-\frac{1}{m_2(K_r)}}$. 
Below the $2$-density, typically an edge is not contained in any copy of $K_r$, whereas poly-logarithmically above the $2$-density the number of $r$-cliques an edge is contained in is well concentrated around $n^{r-2}p^{\binom{r}{2}-1}=\omega(1)$.
In particular, in the sparse regime, typically copies of $K_r$ are edge-disjoint, and thus $X_r$ behaves like a sum of independent random variables. 
This can be carried out formally by the use of the decoupling method while carefully conditioning on typical events. 

The real difficulty lies above the $2$-density, as the dependencies between $r$-cliques become an essential obstruction in the decoupling method.
At a very high-level, by using the decoupling method, we can bound the characteristic function of $X_r$ by the characteristic function of a sum of weighted Bernoulli random variables and an `error' term. This error term resembles most of the dependencies. Dealing with the weighted sum of Bernoulli random variables is standard. We deal with the error term via a high-moment computation, which requires novel ideas and constitutes the most conceptual and technical part of the paper.
Due to the technical depth of this argument, we defer its exposition to the proof overview appearing in Section \ref{sec:proof-overview}.

\subsection*{Organisation of the paper:} The remainder of this paper is organised as follows. In Section~\ref{sec:proof-overview}, we provide a high-level overview of the proof of Theorem \ref{theorem:main theorem}. We first outline how our main result, Theorem \ref{theorem:main theorem}, can be inferred from Lemmas \ref{lemma:characteristic function integrable}, \ref{lemma:characteristic function integrable - low freq and dense p}, and \ref{lemma:characteristic function integrable - low freq and low p}, deferring the formal derivation to Section \ref{sec:quantitative-central-limit}. Since these three key lemmas share a common underlying strategy, the overview in Section \ref{sec:proof-overview} focuses primarily on the proof structure of Lemma \ref{lemma:characteristic function integrable}. The subsequent three sections provide the background and the technical ingredients needed for our proof:
\begin{itemize}
    \item Section \ref{sec:preliminaries} collects the main concentration inequalities utilised throughout the paper.
    \item Section \ref{subsection: decoupling new} contains a standalone discussion of the decoupling argument, which serves as a central tool in our proofs.
    \item Section \ref{sec:typical-properties} establishes several typical properties of the random graph $G_{n, p}$ that are leveraged in the later stages of the argument.
\end{itemize}
The core of the proof of Lemma \ref{lemma:characteristic function integrable} is presented in Section \ref{section:main lemma for medium frequencies}. As will become apparent, this proof relies on two main ingredients. The first is bounding the characteristic function of the clique counts for a class of `special', well-behaving cliques, which is carried out in Section \ref{section:decoupling}. The second ingredient involves bounding a high moment of a certain signed sum involving the remaining, problematic cliques, which is handled in Section \ref{section:high moments}. Together, these two components cover almost the entire range of relevant frequencies, while the small remaining boundary case is handled in Section \ref{sec:high-freq}. Finally, Sections \ref{section:low frequencies} and \ref{sec:low p} are dedicated to the proofs of Lemmas \ref{lemma:characteristic function integrable - low freq and dense p} and \ref{lemma:characteristic function integrable - low freq and low p}, respectively, building heavily upon the framework and technical estimates established in the preceding sections.

\section{Proof overview}\label{sec:proof-overview}
Fix $r\geq 3$ and assume that $ n^{-\frac{1}{m(K_r)}}\ll p \leq \frac{1}{2}$.
The proof of Theorem \ref{theorem:main theorem} takes two steps. 
The first is a reduction to our key Lemmas \ref{lemma:characteristic function integrable}, \ref{lemma:characteristic function integrable - low freq and dense p}, and \ref{lemma:characteristic function integrable - low freq and low p}.
This is fairly standard and appears in previous works, see for example \cite{AraMat2023}. 
Clearly, the second step is then to prove Lemmas \ref{lemma:characteristic function integrable}, \ref{lemma:characteristic function integrable - low freq and dense p}, and \ref{lemma:characteristic function integrable - low freq and low p}, which is the main content of this paper.
In the next subsections, we briefly explain the reduction step, and then focus on a high-level overview of the proofs of Lemmas \ref{lemma:characteristic function integrable},  \ref{lemma:characteristic function integrable - low freq and dense p}, and \ref{lemma:characteristic function integrable - low freq and low p}, with emphasis on Lemma \ref{lemma:characteristic function integrable} where most of the work is carried.

\subsection{The reduction}
By the classical Fourier inversion formula (see e.g.\ Chapter~3 in \cite{Dur2010}) one can derive that for every $x\in\frac{1}{\sigma_r}(\mathbb{Z}-\mu_r)$ and every $K > 0$, 
\begin{equation}\label{eq: integral bound overview}
\begin{aligned}
    \Bigg\lvert\sigma_r \cdot \mathbb{P}(X^*_r= x)& - \frac{1}{\sqrt{2\pi}}e^{-x^2/2}\Bigg\rvert \\ &\leq\int^{K}_{-K}\left\lvert{\E\left[e^{it X_r^*}\right] -e^{-t^2/2}}\right\rvert \diff t +2\int^{\pi\sigma_r}_{K}\left\lvert\E\left[e^{itX^*_r}\right]\right\rvert \diff 
        t+2\int^{\infty}_{K}e^{-t^2/2}\diff t.
\end{aligned}
\end{equation}
To prove \cref{theorem:main theorem} it is enough to show that for any $c,\epsilon>0$ such that $c<\frac{3}{2r}$ there is a choice of $K$ so that each of the integrals above can be bounded as in the quantitative part of Theorem \ref{theorem:main theorem}.

Since $K$ is always significantly large in our analysis, the third summand in \eqref{eq: integral bound overview}  satisfies the required bound, as a tail of the standard normal distribution.
To bound the first summand in~\eqref{eq: integral bound overview}, we combine the quantitative central limit theorem for subgraph counts \cite{BarKarRuc1989} with Stein's method for normal approximation~\cite{Ros2011} (see Lemma \ref{lem: low frequencies}) to obtain the following: 
\begin{align}\label{eq:bottleneck}
     \int_{-K}^{K}\left|\E\left[e^{itX_r^*}\right]-e^{-t^2/2}\right|\diff t = O\left(\frac{K^2}{\min \left\{np^{1/2},n^{r}p^{\binom{r}{2}}\right\}}\right).
\end{align}
It is left to provide suitable bounds for the second summand in \eqref{eq: integral bound overview}; this is where we use Lemmas \ref{lemma:characteristic function integrable}, \ref{lemma:characteristic function integrable - low freq and dense p}, and \ref{lemma:characteristic function integrable - low freq and low p}. The value of $K$ will be chosen differently depending on the value of $p$ in order to get the required bound. The reduction from bounding the left-hand side in \eqref{eq: integral bound overview} to proving Lemmas \ref{lemma:characteristic function integrable}, \ref{lemma:characteristic function integrable - low freq and dense p}, and \ref{lemma:characteristic function integrable - low freq and low p} is carried out in Section \ref{sec:quantitative-central-limit}.

\subsection{Proofs of Lemmas \ref{lemma:characteristic function integrable}, \ref{lemma:characteristic function integrable - low freq and dense p}, and \ref{lemma:characteristic function integrable - low freq and low p}} 
The proofs of Lemmas \ref{lemma:characteristic function integrable}, \ref{lemma:characteristic function integrable - low freq and dense p}, and \ref{lemma:characteristic function integrable - low freq and low p} share many similarities. To simplify our discussion, we will focus mainly on the proof of Lemma \ref{lemma:characteristic function integrable}. 
Nevertheless, we briefly comment on the modifications needed in the proofs of Lemmas \ref{lemma:characteristic function integrable - low freq and dense p} and \ref{lemma:characteristic function integrable - low freq and low p} at the end of the section. We stress that these proofs are fairly simple.

As a preliminary step, following Berkowitz \cite{Ber2018}, we represent $X_r$ as a sum of two random variables, $\Tilde{X}$ and $\Tilde{Y}$. The variable $\Tilde{X}$, in a sense, counts cliques which are `compatible' with the decoupling trick, and $\Tilde{Y}$ is the leftover.
From here, our proof strategy essentially has two main components. 

First, we use the decoupling trick (described below) to bound the characteristic function of $X_r/\sigma_r$ by the characteristic function of $\Tilde{X}/\sigma_r$ together with an error term which has the form of a high moment of a random variable associated with $\Tilde{Y}$. The compatibility of $\Tilde{X}$ with the decoupling trick allows us to approximate the characteristic function of $\Tilde{X}/\sigma_r$ by that of a weighted sum of independent Bernoulli random variables. This reduction is highly advantageous, as the characteristic functions of such random variables are well understood. In particular, the desired bounds follow from Claim~\ref{claim:char-bounds}, which provides a simple estimate for the characteristic function of a weighted Bernoulli random variable (and thereby a sum of such independent variables).

The challenge left is therefore to control the error term, which is the second component of the proof. In the dense regime, where $p=\Theta(1)$, as was addressed in \cite{Ber2018,SahSaw2022}, this task is fairly easy as the problem is somewhat degenerate. However, when $p=o(1)$, unlike standard subgraph counts, this random variable has a significantly more intricate combinatorial interpretation, and obtaining sharp bounds on the order of magnitude of its moments for deriving \eqref{eq: integral bound overview}, constitutes a significant challenge, this is one of our main contribution in the paper.

Before we dive into the details, let us illustrate how the values of the weights in the approximation to a weighted sum of Bernoulli random variables effect the quality of our bounds.
Suppose that $\textbf{x}_1,\ldots,\textbf{x}_s$ are independent Bernoulli random variables with mean $p$ and that $w<w_1,\ldots,w_s <W$ are positive reals. Then, by Claim \ref{claim:char-bounds} for every $t\in \mathbb{R}$, we have
\[
    \left|\E\left[e^{it\sum_{k=1}^{s}w_k\cdot \textbf{x}_k}\right]\right| = \prod_{k=1}^{s}\left|\E\left[e^{it\cdot w_k\cdot \textbf{x}_k}\right]\right| \leq \exp\left(-8p(1-p)\cdot\sum_{k=1}^{s} \left\|\frac{w_k\cdot t}{2\pi}\right\|^{2}\right), 
\]
where $\|x\|$ is the distance between $x$ and $\mathbb{Z}$. To be able to control the sum above, we make sure that $\frac{w_kt}{2\pi} \le \frac{Wt}{2\pi}<\frac{1}{2}$. If this is the case, 
\[
    \left|\E\left[e^{it\sum_{k=1}^{s}w_k\cdot \textbf{x}_k}\right]\right|\leq \exp\left(-8p(1-p)\cdot s\cdot \left(\frac{w\cdot t}{2\pi}\right)^{2}\right).    
\]
Now it is clear that the absolute value of the characteristic function of $\sum_{k=1}^{s}w_k\cdot \mathbf{x}_k$ gets smaller as $s$ and $w$ get larger, as long as $Wt< \pi$.

With the above intuition in mind, we move to the basic methodology of relating the characteristic function of $X_r/\sigma_r$ to that of a weighted sum of independent Bernoulli random variables. Let us abuse notation and for disjoint $A,B\subseteq [n]$ write $A\times B$ for the collection of (undirected) edges with endpoints in $A$ and $B$.
Fix disjoint sets $A_1,A_2\subseteq [n]$, and write $X_r = \sum_{f\in A_1\times A_2}w_f\cdot \mathbf{x}_f+Y$, where (i) $\mathbf{x}_f$ is the indicator of $\{f\in G_{n,p}\}$; (ii) $w_f$ is the number of extensions of $f$ to an $r$-clique with the endpoints of $f$ being the only vertices of this copy in $A_1\cup A_2$; (iii) $Y$ is the remainder defined as $X_r-\sum_{f\in A_1\times A_2}w_f\cdot \mathbf{x}_f$.
For most $f\in A_1\times A_2$ with high probability we have $w_f\approx n^{r-2}p^{\binom{r}{2}-1}$. Therefore, if we were able to approximate the characteristic function of $X_r/\sigma_r$ by that of $\frac{1}{\sigma_r}\sum_{f\in A_1\times A_2} w_f\cdot \mathbf{x}_f$, neglecting technical details, we would have had for every $t$
\[
    \left|\E\left[e^{itX_r/\sigma_r}\right]\right|\lesssim \exp\left(-8p(1-p)\cdot |A_1\times A_2| \cdot \left\|\frac{n^{r-2}p^{\binom{r}{2}-1}\cdot t}{2\pi \cdot \sigma_r}\right\|^{2}\right).
\]
Note that this bound is ineffective for large frequencies: if $t\gg np^{1/2}$ then $n^{r-2}p^{\binom{r}{2}-1} t\gg \sigma_r$ and we cannot evaluate $\left\|\frac{n^{r-2}p^{\binom{r}{2}-1}\cdot t}{2\pi \cdot \sigma_r}\right\|$. 
To overcome this issue, we wish to find a way to decrease the weights according to the frequency $t$.
This raises an immediate concern: it is clear that if $\left\|\frac{w_f\cdot t}{2\pi \cdot \sigma_r}\right\|$ is reduced too much, our upper bound on the characteristic function will become insufficient --- it may not even tend to $0$. Therefore, the weights $w_f$ must be reduced with some care.

\subsection*{The decoupling trick} To remedy the issue of possibly `large' weights, we use a separation-of-variables technique, also used in previous works \cite{AraMat2023, Ber2018, SahSaw2022, SahSawZhu2024}, though our application requires some subtle changes. The formal details are given in Section \ref{subsection: decoupling new}, and here we only explain the main aspects of the method and how it is used. That said, we have to be slightly technical and to introduce several key notations.

Let $k\geq 2$ be a positive integer, let $\cP=(P_1,\ldots,P_k)$ be a partition of $[n]$, and let $\cE\subseteq \binom{[k]}{2}$ be a {\it colouring scheme}. In addition, let $G_0\sim G_{n,p}$ and let us sample independently the random graph $G_1$ by including each edge in $\bigcup_{\{i,j\}\in \cE}P_{i}\times P_{j}$ independently with probability $p$. Then, for every $r$-clique $F\subseteq K_n$, define its sign $\cS_{G_0,G_1}^{\cP,\cE}(F)\in \{0,\pm1\}$ by setting
\[
        S_{G_0, G_1}^{\cP, \cE}(F) \coloneqq \mathbf{1}\left[F \subseteq G_0 \cup G_1\right] \cdot \left(\sum_{\mathbf{v} \in \{0,1\}^{\cE}} \left(-1\right)^{\lvert\mathbf{v}\rvert} \prod_{\{i,j\}\in \cE} \mathbf{1}\left[{F \cap (P_i\times P_j) \subseteq G_{\mathbf{v}_{ij}}}\right]\right);
\] 
for more details see Definition \ref{definition:sign} and Claim \ref{claim: sign zero new}. Let us remark that for our discussion here, it suffices to think of the sign as a mapping from the collection of $K_r$-copies to $\{0,\pm1\}$ that depends on the colouring scheme and satisfies several properties which will be mentioned shortly. Now by applying the Cauchy-Schwarz inequality multiple times --- the decoupling trick --- we can bound the characteristic function of $X_r/\sigma_r$ by the $\frac{1}{2^{|\cE|}}$-power of the characteristic function of $S\coloneqq \frac{1}{\sigma_r}\sum_{F\cong K_r}\cS^{\cP,\cE}_{G_0,G_1}(F)$, see Corollary~\ref{cor: the decoupling lemma new}.

We continue by bounding the characteristic function of $S$. For this, we use the following crucial properties of the sign function. Call an $r$-clique $F\subseteq K_n$ \emph{rainbow} if $F\cap (P_i\times P_j)\neq \emptyset$ for every $\{i,j\}\in \cE$, and denote by $\cR$ the collection of rainbow copies of $K_r$. It is easy to see that if $F$ is not rainbow, its sign is \emph{deterministically} zero, see Claim \ref{claim: sign zero new} and the remark after it. Therefore, the sum in the definition of $S$ can be restricted to the rainbow copies. 
Another important property of the sign function is that, in expectation, the sign function always vanishes, even if the copy is rainbow, see Claim \ref{claim:sufficient condition for sign zero in expectation new}.

For the sake of convenience, rewrite $\cP=(A_1,A_2,B_1\ldots,B_{k-3},Z)$ with $k\geq 4$ and assume that for every $(i,j)\in [2]\times [k-3]$ we have $|A_i|\eqqcolon a, |B_{j}|\eqqcolon b,$ and $|Z|=\Theta(n)$.
The colouring scheme that we use is $\cE\coloneqq [2]\times \{3,\ldots ,k-1\}$. In other words, a rainbow copy is a copy of $K_r$ that intersects each of $A_i$ and $B_j$ (possibly more than once).
Let $X$ be the sum of signs of the rainbow copies taking a \emph{single} vertex in each of $A_i$ and $B_j$, and set $Y\coloneqq \sigma_r S- X$ so that $S=\frac{1}{\sigma_r}(X+Y)$. 
For every $f\in A_1\times A_2$, define $w_f$, the weight of $f$, as the sum of $S_{G_0\cup f, G_1}^{\cP, \cE}(F)$ over all rainbow copies $F\ni f$ that intersect each one of $A_i$ and $B_j$ in a single vertex, and define $\mathbf{x}_f$ to be the indicator of $\{f\in G_{0}\}$. Then, we have $X=\sum_{f\in A_{1}\times A_2} w_f\cdot \mathbf{x}_f$.

For a suitable choice of $a$ and $b$, typically, we have $w_f\ll n^{r-2}p^{\binom{r}{2}-1}$ for a $(\log n)^{-O(1)}$ proportion of the edges $f\in A_1\times A_2$, which suffices for our purposes.
To be more specific, as the expected sign of every copy of $K_r$ is zero, for $(\log n)^{-O(1)}$ proportion of the edges $f\in A_1\times A_2$, the weight $w_f$ is of the order of the standard deviation of the number of copies of $K_{k-1}$ containing $f$ times~$\Theta\left(n^{r-k+1}p^{\binom{r}{2}-\binom{k-1}{2}}\right)$, accounting for the expected number of extensions of a fixed $K_{k-1}$ to a copy of $K_r$ with vertices in $Z$. 
This description is valid provided that $n^{r-k+1}p^{\binom{r}{2}-\binom{k-1}{2}}$ tends to infinity sufficiently fast; we will therefore avoid using this decoupling scheme in the remaining regimes. This achieves our goal in using the decoupling trick, i.e.\ it provides moderate weights for sufficiently many edges in $A_1\times A_2$.
The full details are given in Section \ref{section:decoupling}.

Before moving on, we stress that there are key differences between the cases $r=3$ and $r>3$. 
First, if $r=3$ then for every pair of edges $f,g\in A_1\times A_2$ not sharing a vertex, the weights $w_f$ and $w_g$ are independent. Exploiting this fact, Araújo and Mattos \cite{AraMat2023} used essentially the same partition as above, decomposed $A_1\times A_2$ into perfect matchings, and established the required control on the weights by proving concentration within each matching, where the weights are independent.
On the other hand, when $r>3$, the weights $w_f$ and $w_g$ are not independent even if $f$ and $g$ do not share vertices.
To overcome this, we first reveal $(G_0\cup G_1)\setminus (A_1\times A_2)$ without specifying for each revealed edge whether it belongs to $G_0$ or $G_1$. Then, we show that the number of extensions of an edge $f\in A_1\times A_2$ to a $(k-1)$-clique intersecting $A_1,A_2$, and every $B_i$, is well concentrated. By revealing whether edges belong to $G_0$ or $G_1$, we can show that, despite the dependencies among the weights, with high probability every matching contains sufficiently many edges whose weights have the required size. Finally, we use a trick similar to the matching decomposition in~\cite{AraMat2023}.

The second difference is in the necessity of the random variable $Y$, which we discuss shortly. Indeed, if $r=3$, the partition we consider is $\cP=(A_1,A_2,B_1,Z)$ and every rainbow copy intersects $A_1,A_2$ and $B_1$ in a single vertex. This makes the definition of $Y$ redundant in that case. However, for every $r>3$ when dealing with partitions with $k<r+1$ parts --- which we will have to consider if $p$ is large --- it is clear that $Y$ is non-trivial.

\subsection*{Dealing with dependencies} So far, in our discussion, we ignored the effect of the random variable $Y$. That said, since $X$ and $Y$ are not independent, we cannot simply separate their characteristic functions and focus only on that of $X$. To solve this issue, we use a reduction similar to that of Sah and Sawhney \cite{SahSaw2022}, showing that essentially it suffices to bound the characteristic function of $X$ and show that $\E[(tY/\sigma_r)^{2L}]=n^{-\Omega(1)}$ where $L$ is a large constant, see \eqref{align:splitting to main and error} and Section \ref{section:decoupling}.

To bound $\E\left[Y^{2L}\right]$ we first identify $T$, a collection of sequences $(F_1,\ldots,F_{2L})$ of rainbow copies of $K_r$, satisfying $\E\left[\prod_{i=1}^{2L}\cS_{G_{0},G_{1}}^{\cP,\cE}(F_i)\right]\neq 0$; see Claim \ref{claim:sufficient condition for sign zero in expectation new} and Definition \ref{def: count of good rainbow copies}. Then, we have
\[
    \E\left[Y^{2L}\right]=\sum_{(F_1,\ldots ,F_{2L})\in T}\E\left[\prod_{i=1}^{2L}\cS_{G_{0},G_{1}}^{\cP,\cE}(F_i)\right]. 
\] 
By letting $T(\mathbf{v},e)$ be the collections of tuples $(F_1,\ldots ,F_{2L})\in T$ with $\mathbf{v}$ being the vector encoding the distribution of vertices of $F_1,\ldots,F_{2L}$ among the partition parts and $e=e\left(\bigcup_{i=1}^{2L}F_i\right)$, we further have 
\[
\E\left[Y^{2L}\right]= \sum_{\mathbf{v}}\sum_{e}\sum_{(F_1,\ldots ,F_{2L})\in T(\mathbf{v},e)}\E\left[\prod_{i=1}^{2L}\cS_{G_{0},G_{1}}^{\cP,\cE}(F_i)\right].
\]
It is not difficult to show that for $(F_1,\ldots ,F_{2L})\in T(\mathbf{v},e)$ we have $\E\left[\prod_{i=1}^{2L}\cS_{G_{0},G_{1}}^{\cP,\cE}(F_i)\right] =O(p^{e})$ which implies the bound
\[
\E\left[Y^{2L}\right]= O\left(\sum_{\mathbf{v}}\sum_{e} |T(\mathbf{v},e)|\cdot p^{e}\right).
\]
Then, the challenging part is to find the maximum of the right-hand side up to a multiplicative constant. This is done by first analysing the possible values of $\mathbf{v}$ and of $e$ (as a function of $\mathbf{v}$). Bounding the possible values of $\mathbf{v}$ is easy, while obtaining tight bounds on the possible values of $e$ is non-trivial and requires a fine understanding of the possible interactions between the $2L$-tuples of rainbow copies belonging to $T$.
Then, by some computations, we determine the maximum of the right-hand side above as a function of $a,b$ and $p$. For more details, see Subsection \ref{subsec:structure-of-maximum}.

Let us conclude our discussion in this subsection by mentioning that the upper bound we get is increasing in both $a$ and $b$.
On the other hand, $a$ and $b$ need to be sufficiently large for the approximation of the characteristic function by that of a sum of weighted Bernoulli random variables to be efficient. This constrains $t$ to be sufficiently small in order to have the desired $\E\left[(tY/\sigma_r)^{2L}\right]=n^{-\Omega(1)}$. Therefore, a careful analysis is required to show that for every $t$ one can find $a$ and $b$ so that both bounds are achievable.
The formal details regarding this high moment argument are given in Section \ref{section:high moments}.

\subsection*{Concluding the proof} 
Given a partition $\cP=(A_1,A_2,B_1,\ldots , B_{k-3},Z)$ of $[n]$ with cardinalates $a$ and $b$ as above, we conclude that there are $L=L(a,b,k)$ and $U=U(a,b,k)$ such that $\left|\E\left[e^{itX_r/\sigma_r}\right]\right|\leq n^{-\Omega(1)}$ for every frequency $t\in I(a,b,k)=[L,U]$. 
The proof of Lemma \ref{lemma:characteristic function integrable} follows once we show that for every $t\in [n^{1/2+\epsilon}p,\pi \sigma_r]$ there are permissible values of $a,b,$ and $k$ such that $t\in I(a,b,k)$; in this case we say that $t$ is \emph{covered}. 

The way we show this is quite simple, although it requires several tedious computations. Indeed, for every $k$, by varying over all values of $a,b$ for which the above two steps are carried out, we show that there is a non-empty interval $I_k=[L_k,U_k]$ of covered frequencies.
Then, we find some $k^*$ depending on $p$, for which we show the following:
\begin{itemize}
    \item $U_{k^*}\geq \pi\sigma_r$ and $L_{r-2}\leq n^{1/2+\epsilon}p$.
    \item For every $k^*\leq \ell \le  r-3$ we have $L_{\ell}<U_{\ell+1}$.
\end{itemize}
This clearly shows that $\bigcup_{k=k^*}^{r-2}I_k\supseteq [n^{1/2+\epsilon}p,\pi\sigma]$, as required. See Section~\ref{section:main lemma for medium frequencies} for the exact details.

\subsection*{Lemma \ref{lemma:characteristic function integrable} vs Lemmas \ref{lemma:characteristic function integrable - low freq and dense p} and \ref{lemma:characteristic function integrable - low freq and low p}}
As was already mentioned, the proofs of Lemmas \ref{lemma:characteristic function integrable - low freq and dense p} and \ref{lemma:characteristic function integrable - low freq and low p} are very similar to that of Lemma \ref{lemma:characteristic function integrable} and follow the same strategy. 
The difference in Lemma \ref{lemma:characteristic function integrable - low freq and dense p} comes from taking a simpler partition, $\cP=(A_1,A_2,Z)$, and using the decoupling scheme $\cE=\{\{1,2\}\}$. The properties of the rainbow copies are preserved, and we still make the distinction between the contributions of rainbow copies intersecting both $A_1$ and $A_2$ once, and those that do not.

The advantage of this decoupling scheme is larger weights, which is key to cover lower frequency ranges. Another nice feature of this setup is that our argument of neglecting the contribution of the rainbow copies that intersect $A_1$ or $A_2$ in at least two vertices is simpler. 
However, this argument imposes strong constraints on the frequency range covered as a function of $a=|A_1|=|A_2|$. To resolve this issue, we instead implement the proof with $a$ varying over many distinct values, and do not take $a$ to be fixed of order $n^{1-o(1)}$ as before. The exact details are given in Section \ref{section:low frequencies}.

Lastly, to prove Lemma \ref{lemma:characteristic function integrable - low freq and low p}, we consider a related partition $\cP=(A_1,A_2,B_1,\ldots,B_{r-2},Z)$, where $|A_i|=|B_i|=\tilde{n}$, and also consider the colouring scheme $\cE=[2]\times \{3,\ldots , r-2\}$. The key difference here is that when $p$ is sufficiently small, with high probability, most edges of $A_1\times A_2$ are not contained in an $r$-clique. We then need to ask for a quantitative control on the number of edges of $A_1\times A_2$ contained in a unique $r$-rainbow-clique. While this seems to complicate the proof, as one needs stronger control over the typical number of extensions of an edge to an $r$-clique, the proof in fact becomes simpler than that of Lemma~\ref{lemma:characteristic function integrable}. This is mostly because we do not need to deal with dependencies (i.e.\ the random variable $Y$ in the above discussion). For the full detailed proof, see Section \ref{sec:low p}.

\section{Preliminaries}\label{sec:preliminaries}
Throughout this paper, we use the following notation.
For a graph $G$, we write $V(G)$ to denote its vertex set, and $E(G)$ to denote its edge set, moreover we denote the cardinality of these sets by $v(G)$ and $e(G)$ respectively. Sometimes, for notational convenience we will not distinguish between the graph and its edge set. The degree of a vertex $v\in V(G)$ is denoted by $d_G(v)$, and when $G$ is clear from the context, we will simply write $d(v)$.
For two sets of vertices $A$ and $B$, we abuse notation and write $A\times B$ for the set of (undirected) edges with one endpoint in $A$ and the other endpoint in $B$. When there is ambiguity, this will be clarified.

We denote by  $\left\|x\right\|$ the distance between a real $x$ and $\mathbb{Z}$. 
As usual, for a function $f$ we write $\textup{Im}(f)$ to denote the image of $f$. 
For some $p\in(0,1)$, we often write $G_p$ to denote the percolated subgraph of $G$. This subgraph is obtained by retaining each edge of $G$ with probability $p$ independently of the others. We note that $G_{n,p}$ is then simply $(K_n)_p$. Lastly, for an event $\mathcal{A}$ in some probability space, we write $\mathbf{1}[\cA]$ to denote the indicator random variable of $\cA$.

\subsection{Characteristic function of Bernoulli random variables} 
Throughout the paper, we use the following well-known bound, see e.g.\ \cite[Lemma 1]{GilKop2016}.

\begin{claim}\label{claim:char-bounds}
    Let $X$ be a Bernoulli random variable with mean $p \in [0, 1]$. Then, for every real $t$,
        \[
            \left\lvert\E \left [e^{itX} \right]\right\rvert\leq 1-8p(1-p)\cdot \left\|\frac{t}{2\pi}\right\|^{2}.
        \]
\end{claim}

\subsection{Concentration inequalities}
Given a set $V$ and $p \in [0,1]$, denote by $V_p$ the random subset of $V$ obtained by keeping each element of $V$ with probability $p$, independently. Suppose that $\cH$ is a hypergraph with vertex set $V$. Define
\[
    \mu_p(\cH) \coloneqq \sum_{S \in \cH} p^{|S|} \quad \text{and} \quad \Delta_p(\cH) \coloneqq \sum_{\substack{S_1, S_2 \in \cH \\ S_1 \neq S_2,\ S_1 \cap S_2 \neq \emptyset}} p^{|S_1 \cup S_2|},
\]
where the second sum is over unordered pairs of edges $S_1,S_2 \in \cH$. In other words, $\mu_p(\cH)$ is the expected number of edges in the subhypergraph $\mathcal{H}[V_p]$ of $\mathcal{H}$ induced by the set $V_p$, and $\Delta_p(\cH)$ is the expected number of distinct unordered pairs of edges of $\cH[V_p]$ with a non-empty intersection.
The following theorems are well-known concentration results for the random variable $\left|\cH[V_p]\right|$.
\begin{theorem}[Janson's inequality \cite{Jan1990}]\label{theorem: Janson}
    For all $p\in [0,1]$ and a constant $\gamma \in(0,1)$, we have
    \[
        \Pr\left(\left|\cH[V_p]\right| \le(1- \gamma)\cdot   \mu_p(\cH)\right) \le \exp\left(-\frac{\gamma^{2}}{4}\cdot\min\left\{\mu_p(\cH), \frac{\mu_p(\cH)^2}{\Delta_p(\cH)}\right\}\right).
    \]
\end{theorem}

\begin{theorem}[Kim--Vu Polynomial Concentration Inequality \cite{MR1774845}]\label{theorem: Kim Vu}
    For every $A \subseteq V$, let $\cH_A$ be the hypergraph consisting of the edges of $\cH$ that contain $A$.
    Let $p \in [0, 1]$ and let $k \coloneqq \max_{e \in \cH} |e|$. For every integer $i\in [k]$, define $E_i \coloneqq \max_{\substack{A \subseteq V \\ |A|=i}} \mu_p(\cH_A)$. Set
    \[
        E' \coloneqq \max_{1 \le i \le k} E_i \quad \text{and} \quad E \coloneqq  \max\{\mu_p(\cH), E'\}.
    \]    
    Then, there exists a constant $C = C(k) > 0$ such that 
    \[
        \Pr\left(\left|\cH[V_p] - \mu_p(\cH)\right| > C (E E')^{1/2} \lambda^k\right) < 18 e^{-\lambda} n^{k-1}.
    \]
\end{theorem}

\begin{theorem}[Harris inequality \cite{Har1960}]\label{theorem: Harris}
   Let $\Omega$ be a finite set, and let $X$ and $Y$ be random variables defined on a product probability space over $\{0,1\}^{\Omega}$.
   If $X$ and $Y$ are both non-decreasing (or non-increasing), then
   \[
   \E[XY]\geq \E[X]\E[Y].
   \]
\end{theorem}

\begin{lemma}[Paley-Zygmund inequality]\label{claim: Paley-Zygmund}
    Let $X \ge 0$ be a random variable with finite variance and let $\theta \in [0, 1]$. Then,
    \[
        \Pr(X > \theta \E[X]) \ge (1-\theta)^2 \cdot \frac{\E[X]^2}{\E\left[X^2\right]}.
    \]
\end{lemma}

We will also need the following tail bound for binomial random variables, known as the Chernoff bound.
\begin{lemma}[Chernoff Bound]\label{lemma:chernoff}    
    Suppose that $X \sim \textup{Bin}(N, p)$ and that $\delta \in (0,1)$. Then,
    \[
        \Pr\left(\lvert X - Np \rvert \ge \delta Np\right) \le 2e^{-\frac{\delta^2 Np}{3}}.
    \]
\end{lemma}

\section{The decoupling}\label{subsection: decoupling new}
In this subsection, we recall a separation-of-variables technique due to Sah and Sawheny \cite{SahSaw2022}, that borrowed these ideas from Berkowitz \cite{Ber2016,Ber2018}. We begin with a lemma concerning general product spaces. Afterwards, we will apply this lemma to the number of $K_r$ in $G_{n, p}$ that we denote by $X_r$. We first define an operator $\alpha$ that plays a crucial role.

\begin{definition}
    Suppose that $\Omega_0,\ldots, \Omega_m$ are sets and that $f\colon \prod_{i=0}^{m} \Omega_i \to \mathbb{R}$.
    For a vector $\mathbf{v}\in \{0,1\}^{m}$, letting $\left\lvert\mathbf{v}\right\rvert=\sum_{i=1}^{m}\mathbf{v}_i$, we define $\alpha_m(f) : \Omega_0\times \prod_{i=1}^{m} \Omega_i^2 \to \mathbb{R}$ to be the following function:
    \[
        \alpha_m(f)(x,y_1^0,y_1^1,\ldots,y_m^0,y_m^1) \coloneqq \sum_{\mathbf{v} \in \{0,1\}^m} (-1)^{\left\lvert\mathbf{v}\right\rvert} f(x,y_{1}^{\mathbf{v}_1},\ldots,y_{m}^{\mathbf{v}_m}).
    \]
\end{definition}

The main result concerning this operator $\alpha_m$ is a separation-of-variables technique, which essentially relies on a repeated use of the Cauchy-Schwarz inequality. The next lemma was proved in \cite{Ber2016, Ber2018, SahSaw2022}; for the sake of completeness, we repeat its proof here.

\begin{lemma}\label{lemma:the decopuling lemma new}
    Suppose that $(X,\mathbf{Y})\coloneqq(X,Y_{1},\ldots,Y_{m})$ are mutually independent random variables taking values in $\Omega_X$ and $\Omega_{Y_i}$ respectively. Let $f\colon \Omega_X\times \prod_{i=1}^{m} \Omega_{Y_i} \to \mathbb{R}$ be a measurable function and denote by $\varphi (t) \coloneqq \E_{(X,\mathbf{Y})}\left[e^{itf(X,\mathbf{Y})}\right]$ the characteristic function of $f(X,\mathbf{Y})$. Then, 
    \[
        \left\lvert\varphi(t)\right\rvert ^{2^{m}} \le \E_{X, \mathbf{Y^{*}}}\left[e^{it\alpha_m(f)(X,\mathbf{Y^*})}\right],
    \]
    where $\mathbf{Y^*}\coloneqq (Y_1^{0},Y_1^{1},\ldots,Y_m^{0},Y_m^{1})$ is a vector of mutually independent random variables and for every $i$, the random variables $Y_i^0,Y_i^1,Y_i$ are identically distributed.
\end{lemma}

\begin{proof}
    The proof is done by induction on $m$. The base case $m=0$ trivially holds as $\alpha_0(f)=f \colon \Omega_X\to\mathbb{R}$.
    For the induction step, fix some integer $m>0$, and assume the assertion of the lemma for this $m$.
    Let $(X,\mathbf{Y})\coloneqq(X,Y_{1},\ldots,Y_{m+1})$ be a sequence of mutually independent random variables taking values in $\Omega_X$ and $\Omega_{Y_i}$, respectively. Fix some function $f\colon \Omega_X\times \prod_{i=1}^{m+1} \Omega_{Y_i} \to \mathbb{R}$. Set $\mathbf{\tilde{Y}}\coloneqq (Y_{1},\ldots,Y_{m})$ and let $\tilde{f} \colon \Omega_{X}\times \Omega_{Y_{m+1}}\times \prod_{i=1}^{m} \Omega_{Y_i} \to \mathbb{R}$ be defined as $\tilde{f}((x,y_{m+1}),y_1,\ldots,y_{m})=f(x,y_1,\ldots,y_{m+1})$.
    By the induction hypothesis and the triangle inequality, letting $\varphi$ be the characteristic function of $f(X,Y_1,\ldots,Y_{m+1})$, we have    
    \begin{align*}
        \left\lvert\varphi(t)\right\rvert^{2^{m}} &\le \E_{\mathbf{\tilde{Y}^{*}}}\left\lvert\E_{X,Y_{m+1}}\left[e^{it\alpha_m(\tilde{f})((X,Y_{m+1}),\mathbf{\tilde{Y}^*})}\right]\right\rvert\\
        &\le \E_{\mathbf{\tilde{Y}^{*}},X}\left\lvert\E_{Y_{m+1}}\left[e^{it\alpha_m(\tilde{f})((X,Y_{m+1}),\mathbf{\tilde{Y}^*})}\right]\right\rvert,
    \end{align*}  
    where $\mathbf{\tilde{Y}^*}\coloneqq (Y_1^{0},Y_1^{1},\ldots,Y_m^{0},Y_m^{1})$ are mutually independent and for every $i$, the random variables $Y_i^0,Y_i^1,Y_i$ are identically distributed.
    Applying the Cauchy-Schwarz inequality, we obtain     
    \begin{align*}
        \left\lvert\varphi(t)\right\rvert^{2^{m+1}}
        &\le \E_{\mathbf{Y^{*}},X}\left\lvert\E_{Y_{m+1}}\left[e^{it\alpha_m(\tilde{f})((X,Y_{m+1}),\mathbf{\tilde{Y}^*})}\right]\right\rvert^2\\
        &= \E_{X,\mathbf{Y^{*}},Y^0_{m+1},Y^1_{m+1}}\left[e^{it\alpha_m(\tilde{f})((X,Y^{0}_{m+1}),\mathbf{\tilde{Y}^*})-it\alpha_m(\tilde{f})((X,Y^{1}_{m+1}),\mathbf{\tilde{Y}^*})}\right]\\
        &= \E_{X,\mathbf{Y^{*}},Y^0_{m+1},Y^1_{m+1}}\left[ e^{it\alpha_{m+1}(f)(X,\mathbf{\tilde{Y}^*},Y^{0}_{m+1},Y^{1}_{m+1})}\right],
    \end{align*}
    completing the proof.
\end{proof}

Lemma \ref{lemma:the decopuling lemma new} has played a central role in several previous works on LCLTs for subgraph counts and related random variables; see, for example, \cite{AraMat2023,Ber2016,Ber2018,BerSahSaw2021,GilKop2016,SahSaw2022,SahSawZhu2024}. In these works, one specifies collections of edges $\cF_i$ and applies Lemma \ref{lemma:the decopuling lemma new} with $Y_i=(\mathbf{1}[e\in E(G_{n,p})])_{e\in \cF_i}$ and $f=X_r$ (or, more generally, $f=X_H$).

The key idea is that, viewing $f$ as a polynomial in the $\binom{n}{2}$ edge-indicator variables of degree $\binom{r}{2}$, Lemma \ref{lemma:the decopuling lemma new} reduces the problem of bounding the characteristic function of $f(X,\mathbf Y)$ to bounding the characteristic function of $\alpha_m(f)(X,\mathbf {Y^*})$, conditioned on $\mathbf {Y^*}$. Crucially, the latter is a polynomial of degree at most $\binom{r}{2}-m$. Since linear polynomials of independent random variables are well understood, this reduction becomes particularly effective when the resulting polynomial has degree $1$ (conditionally on $\mathbf{Y^*}$).

However, reducing all the way to a linear polynomial can be too costly and, in general, yields only suboptimal bounds for the characteristic function of $X_r$. Our approach is instead to reduce the problem to the analysis of lower degree polynomials which are not necessarily linear. The challenge is that these nonlinear polynomials contain highly dependent monomials. While such polynomials can be analysed when $p$ is bounded away from $0$ and $1$, as was done in earlier works~\cite{Ber2018,SahSaw2022}, new ideas are required in the sparse regime when $n^{-1/m(K_r)} \ll p =o(1)$.

For the rest of this section, fix non-negative integers $k, n$ and $r$. In addition, fix a partition $\cP = (P_1, \dots, P_k)$ of $[n]$ and $\cE \subseteq \binom{[k]}{2}$. The partition $\cP$ and the collection $\cE$ will dictate the sets $\cF_i$ mentioned above, and the following definition will encapsulate the reduction of the degree of $X_r$ (conditionally on $Y_i$) via application of Lemma \ref{lemma:the decopuling lemma new}.
Combinatorially, $\alpha_m(X_r)$ applied with a pair of the edge indicators for graphs $G_0 \subseteq K_n$ and $G_1 \subseteq \bigcup_{\{i, j\} \in \cE} \left(P_i \times P_j\right)$ is a signed sum of indicators of $r$-cliques formed by combination of edges from $G_0$ and $G_1$.
The following definition groups together all signed contributions of $r$-cliques on a vertex set of size $r$ to $\alpha_m(X_r)$ applied with such pair $(G_0,G_1)$. For clarity of presentation, throughout this section unless stated otherwise fix 
\[ 
G_0\subseteq K_n \quad \text{and} \quad G_1 \subseteq \bigcup_{\{i, j\} \in \cE} \left(P_i \times P_j\right)\eqqcolon \cP_{\cE}.
\]

\begin{definition}\label{definition:sign}
   For every copy $F \subseteq K_n$ of $K_r$, define the \textit{sign} of $F$ in $(G_0,G_1)$ with respect to $\cP$ and $\cE$ to be\footnote{We note that in the definition of the sign function, we abuse notation and use $\mathbf{v}$ to denote a vector.}
    \[
        S_{G_0, G_1}^{\cP, \cE}(F) \coloneqq \mathbf{1}\left[F \subseteq G_0 \cup G_1\right] \cdot \left(\sum_{\mathbf{v} \in \{0,1\}^{\cE}} \left(-1\right)^{\lvert\mathbf{v}\rvert} \prod_{\{i,j\}\in \cE} \mathbf{1}\left[{F \cap (P_i\times P_j) \subseteq G_{\mathbf{v}_{ij}}}\right]\right).
    \]
    When clear from the context we will simply write $S_{G_0,G_1}(F)$ to denote $S_{G_0,G_1}^{\cP,\cE}(F)$.
\end{definition}

A fact that we prove next is that when the sign function $S_{G_0,G_1}$ does not vanish it takes values only in $\{\pm 1\}$. To prove this we define rainbow $r$-cliques that have a non-empty intersection with every bipartite graph $P_i\times P_j$. We further show that, for all non-rainbow cliques, their sign function vanishes.

\begin{definition}
   We say that a copy $F \subseteq K_n$ of $K_r$, is $\left(\cP, \cE\right)$-\textit{rainbow} if the following holds:
\[\text{For every } \{i, j\} \in \cE \text{ we have }  F \cap ( P_i\times P_j) \neq \emptyset.\]
    Denote by $\cR(\cP, \cE)$ the collection of all $\left(\cP, \cE\right)$-rainbow $r$-cliques. When the dependency on $\cP$ and $\cE$ is clear from the context, we will omit it from the notation.
\end{definition}

\begin{claim}\label{claim: sign zero new}
    The sign function satisfies $
        \textup{Im}(S_{G_{0},G_{1}}) \subseteq \{0,\pm 1\}$.
    Moreover, if $F$ is a copy of $K_r$ satisfying $F\cap (P_{i}\times P_{j})\subseteq G_0\cap G_1$ for some $\{i,j\}\in \cE$, then the sign $S_{G_0,G_1}(F)$ vanishes.
\end{claim}

Before we prove this claim, let us remark that the `moreover' part of the claim implies that $S_{G_0,G_1}(F)=0$ for every $F\not\in \cR$. To see this note that for every $F\not\in \cR$ there is $\{i,j\}\in \cE$ such that $F\cap (P_i\times P_j)=\emptyset\subseteq G_0\cap G_1$.

\begin{proof}
    Let $F$ be a copy of $K_r$. We may assume that $F \subseteq G_0 \cup G_1$, as otherwise we have $S_{G_0, G_1}(F) = 0$ immediately from the definition of the sign function. Let us begin by proving the second assertion of the claim. Let $\{i^*,j^*\}\in\mathcal{E}$ satisfy $F\cap(P_{i^*}\times P_{j^*})\subseteq G_0\cap G_1$. We have
    \begin{align*}
        \sum_{\substack{\mathbf{v} \in \{0,1\}^{\cE}\\ \mathbf{v}_{i^* j^*}=1}} \left(-1\right)^{\lvert\mathbf{v}\rvert} \prod_{\{i, j\}\in \cE} \mathbf{1}\left[{F \cap (P_i\times P_j) \subseteq G_{\mathbf{v}_{i j}}}\right] =-\sum_{\substack{\mathbf{v} \in \{0,1\}^{\cE}\\ \mathbf{v}_{i^* j^*}=0}} \left(-1\right)^{\lvert\mathbf{v}\rvert} \prod_{\{i, j\}\in \cE} \mathbf{1}\left[{F \cap (P_i\times P_j) \subseteq G_{\mathbf{v}_{i j}}}\right].
    \end{align*}
    This implies the required equality
    \[
         S_{G_0, G_1}(F) =\sum_{\mathbf{v} \in \{0,1\}^{\cE}} \left(-1\right)^{\lvert\mathbf{v}\rvert} \prod_{\{i,j\} \in \cE} \mathbf{1}\left[{F \cap (P_i\times P_j) \subseteq G_{\mathbf{v}_{i j}}}\right] = 0.
    \]

    Let us now turn to the proof of the first assertion of the claim. To this end, fix $F$ such that $S_{G_0,G_1}(F)\neq 0$.
    Due to the second assertion that we have just proved, we may assume that for every $\{i,j\}\in\mathcal{E}$, the graph $F\cap(P_i\times P_j)$ is a subgraph of exactly one of the two graphs $G_0$ and $G_1$. In other words, there exists a vector $\mathbf{u}\in\{0,1\}^{\cE}$ (which we fix in what follows) such that for every $\{i,j\}\in \cE$, 
    $$
    F\cap(P_i\times P_j)\subseteq G_{\mathbf{u}_{ij}}\quad\text{ while }\quad
    F\cap(P_i\times P_j)\not\subseteq G_{1-\mathbf{u}_{ij}}.
    $$
    In particular,
    \[
       \sum_{\mathbf{v} \in \{0,1\}^{\cE}} \left(-1\right)^{\lvert\mathbf{v}\rvert} \prod_{\{i,j\} \in \cE} \mathbf{1}\left[{F \cap (P_i\times P_j) \subseteq G_{\mathbf{v}_{i j}}}\right] = \left(-1\right)^{\lvert\mathbf{u}\rvert}\prod_{\{i,j\} \in \cE} \mathbf{1}\left[{F \cap (P_i\times P_j) \subseteq G_{\mathbf{u}_{i j}}}\right] . 
    \]
    As the left-hand side above is $S_{G_0,G_1}(F)$ and as the right-hand side clearly belongs to $\{\pm 1\}$, the proof is complete.
\end{proof}

 The next fact establishes a simple duality of the sign function when changing the base graphs $G_0$ and $G_1$. This will become useful when we consider $G_0$ and $G_1$ as random graphs. Indeed, it will imply that the expected sign of any fixed copy of $K_r$ is null.

\begin{fact}\label{fact:changing-G0-G1}
    Suppose that $\{i,j\}\in \cE$, fix a copy $F\subseteq K_n$ of $K_r$, and define the bipartite graph $W_{F}^{ij} = (V(F) \cap P_i)\times (V(F)\cap P_j)$. Then, letting 
    \[
        \bar{G}_0\coloneqq \left(G_0 \setminus W_{F}^{ij}\right)\cup \left(G_1\cap W_{F}^{ij}\right) \quad \text{and}\quad \bar{G}_1 \coloneqq \left(G_1 \setminus W_F^{ij}\right)\cup \left(G_0\cap W_F^{ij}\right)\subseteq \cP_{\cE},
    \]
    we have $S_{G_0,G_1}(F)=-S_{\bar{G}_0,\bar{G}_1}(F)$.
\end{fact}

\begin{proof}
    Fix a copy $F\subseteq K_n$ of $K_r$. By definition, for both $s\in \{0,1\}$ and any $\{i,j\}\neq \{x,y\}\in \cE$ we have
    \[
    \mathbf{1}\left[{F \cap (P_x\times P_y) \subseteq \bar{G}_{s}}\right] = \mathbf{1}\left[{F \cap (P_x\times P_y) \subseteq G_{s}}\right],
    \]
    and moreover,
    \[
    \mathbf{1}\left[{F \cap (P_i\times P_j) \subseteq \bar{G}_{s}}\right] = \mathbf{1}\left[{F \cap (P_i\times P_j) \subseteq G_{1-s}}\right].
    \]    
    
    In particular, for every $\mathbf{v}\in  \{0,1\}^{\cE}$ letting $\mathbf{\bar{v}}\in  \{0,1\}^{\cE}$ be the vector agreeing with $\mathbf{v}$ in all coordinates expect for the $\{i,j\}$-th coordinate, the proof is concluded as follows:
    \begin{align*}
        S_{\bar{G}_0, \bar{G}_1}(F) &=\sum_{\mathbf{v} \in  \{0,1\}^{\cE}} \left(-1\right)^{\lvert\mathbf{v}\rvert} \prod_{\{x,y\} \in \cE} \mathbf{1}\left[{F \cap (P_x\times P_y) \subseteq \bar{G}_{\mathbf{v}_{xy}}}\right] \\
        &= \sum_{\mathbf{v} \in  \{0,1\}^{\cE}} \left(-1\right)^{\lvert\mathbf{v}\rvert} \prod_{\{x,y\} \in \cE} \mathbf{1}\left[{F \cap (P_x\times P_y) \subseteq G_{\mathbf{\bar{v}}_{xy}}}\right] = -S_{G_0,G_1}(F).\qedhere
    \end{align*}
\end{proof}

The following corollary is an adjustment of Lemma \ref{lemma:the decopuling lemma new} to the setup that we use throughout the paper.
Recall that $X_r$ is the random variable counting the number of copies of $K_r$ in $G_{n,p}$. In what follows, we view $X_r(G)$ as a function of the edges of a graph $G$. For a particular sample $G\sim G(n,p)$, we will simply write $X_r=X_r(G)$. With some abuse of notation, for (not necessarily random) graphs $G_1,\ldots,G_\ell$ we denote by $X_r(G_1,\ldots,G_{\ell})$ the number of copies of $K_r$ in $G_1 \cup \dots \cup G_{\ell}$.

\begin{corollary}\label{cor: the decoupling lemma new}
    Let $p=p(n) \in [0,1]$ and suppose that $G_0 \sim G_{n, p}$ and $G_1\sim \left(\cP_{\cE}\right)_p$ are independent. Then, setting $\Tilde{X}_r \coloneqq \sum_{F \in \cR} S_{G_0, G_1}(F)$ and $X_r\coloneqq X_r(G_0)$, for every real number $t$ we have
    \begin{align*}
        \left\lvert \E_{G_0}\left[e^{it X_r/\sigma_r}\right]\right\rvert^{2^{\lvert \cE \rvert}} \le \left\lvert \E_{G_0, G_1}\left[ e^{it \Tilde{X}_r / \sigma_r}\right] \right\rvert.
    \end{align*}
\end{corollary}

\begin{proof}
    For every $\{i, j\} \in \cE$, define
    \[  
        Y_{i j}\coloneqq (P_i\times P_j)_p \quad \text{and} \quad X\coloneqq \left(K_n\setminus \cP_\cE\right)_p.
    \]
    Lemma \ref{lemma:the decopuling lemma new}, invoked with $t$ replaced by $t/\sigma_r$, with $(X,\mathbf{Y})$, where $\mathbf{Y} = (Y_{ij})$ is the collection of the above random variables and the indices are ordered lexicographically, and with $f=X_r$, implies the following for $m\coloneqq \lvert \cE\rvert$:
    \begin{equation}\label{eq: applying the decopuling new}
    \left\lvert\E_{X,\mathbf{Y}}\left[e^{itX_r(X,\mathbf{Y})/\sigma_r}\right]\right\rvert ^{2^{m}}\le \E_{X, \mathbf{Y^{*}}}\left[e^{it\alpha_m(X_r)(X,\mathbf{Y^*})/\sigma_r}\right],
    \end{equation}
    where $\mathbf{Y^*}\coloneqq \left(Y_{ij}^{0},Y_{ij}^{1} : \{i, j\}\in \cE\right)$ is a vector of mutually independent random variables, and moreover for every $\{i,j\}\in \cE$ the random variables $Y_{ij}^{0},Y_{ij}^{1},Y_{ij}$ are identically distributed.
    Viewing the sequence $\left(X,Y^0_{ij}:\{i,j\}\in \cE\right)$ as the edges of $G_0$, and $\left(Y_{ij}^1: \{i,j\}\in \cE\right)$ as the edges of $G_1$, we find the following due to the definition of $\alpha_m$:
    \begin{equation}\label{eq: alpha on mixed random graphs new}
        \alpha_m(X_r)\left(X,\mathbf{Y^*}\right) \coloneqq \sum_{\mathbf{v} \in \{0,1\}^m} (-1)^{\left\lvert\mathbf{v}\right\rvert} X_r\left(X,\mathbf{Y}^{\mathbf{v}}\right),
    \end{equation}
    where $\mathbf{Y}^{\mathbf{v}}=\left(Y_{ij}^{\mathbf{v}_{ij}}: \{i,j\}\in \cE\right)$. By the definition of $S_{G_0,G_1}(F)$ and by the remark after Claim \ref{claim: sign zero new}, we may rewrite \eqref{eq: alpha on mixed random graphs new} as follows:
    \begin{align*}
        \alpha_m(X_r)\left(X,\mathbf{Y^*}\right) &= \sum_{F\cong K_r}\mathbf{1}\left[{F \subseteq G_0\cup G_1}\right]\sum_{\mathbf{v} \in \{0,1\}^\cE}(-1)^{|\mathbf{v}|} \prod_{\{i,j\}\in \cE} \mathbf{1}\left[{F \cap (P_i \times P_j) \subseteq G_{\mathbf{v}_{ij}}}\right]\\
        & = \sum_{F\cong K_r}S_{G_0,G_1}(F) = \sum_{F\in \cR}S_{G_0,G_1}(F) = \Tilde{X}_r.
    \end{align*}
    Plugging this into \eqref{eq: applying the decopuling new}, the proof is complete.
\end{proof}

\begin{remark}\label{remark: the decoupling lemma new}
    Note that Corollary \ref{cor: the decoupling lemma new} can be adapted straightforwardly to yield 
    \begin{align*}
        \left\lvert \E_{G_0}\left[e^{it X_r(G_0)/\sigma_r}\mid G_0[U]=\Gamma \right]\right\rvert^{2^{\lvert \cE \rvert}} \le \left\lvert \E_{G_0, G_1}\left[ e^{it \Tilde{X}_r / \sigma_r}\mid G_0[U]=\Gamma\right] \right\rvert,
    \end{align*}
    for every $U\subseteq [n]$ such that $K_{n}[U]\cap \cP_{\cE} =\emptyset$ and every fixed graph $\Gamma$ .
\end{remark}

Although the definitions and claims in this section concern a general partition $\cE$, throughout the paper, we only use partitions of the following form.

\begin{definition}\label{def: (a,b,k)-partition new}
    Let $a,b, r,$ and $k$ be non-negative integers such that $0\le k\le r-2$. we say that a partition ~$\cP = (A_1,A_2,B_1,\ldots B_{r-k-2},Z)$ of $[n]$ is an \emph{$(a,b,k,r)$-partition} if
    \[
        \left\lvert A_1\right\rvert=\left\lvert A_2\right\rvert=a \quad \text{and}\quad \left\lvert B_1\right\rvert=\ldots =\left\lvert B_{r-k-2}\right\rvert=b.
    \]
    Further, define
    \[
        \cP_A \coloneqq A_1 \cup A_2, \quad \cP_B \coloneqq B_1 \cup \dots \cup B_{r-k-2}  \quad \text{and} \quad \cP_Z \coloneqq Z.
    \]
    Lastly, define 
    \[
        \cE_{\cP} = \begin{cases}
                        \left\{\{x,y\}:x\in [2],y\in [r-k]\setminus [2]\right\}&\text{if }k<r-2,\\
                        \left\{\{1,2\}\right\}& \text{if }k=r-2.
                    \end{cases}
    \]
\end{definition}

To avoid confusion, we remark that when $k = r-2$, the partition has only three sets: $A_1, A_2,$ and $Z$. In particular, the parameter $b$ is irrelevant in this case.

Let us also clarify the definition of $\cE_{\cP}$. Recall that the set $\cE$ serves as an encoding of the edge sets we use in the decoupling argument, i.e.\ when we apply of Corollary \ref{cor: the decoupling lemma new}. The edge sets encoded by $\cE_{\cP}$ are as follows: If $k<r-2$, then $\cE_{\cP}$ represents the edge sets $A_i\times B_j$ with $(i,j)\in [2]\times [r-k-2]$. If $k=r-2$, then $\cE_{\cP}$ represents solely the edge set $A_1\times A_2$.

As mentioned, we will mainly use $(a,b,k,r)$-partitions in the sequel. To ease the presentation, throughout the paper, when a definition or a statement uses an $(a,b,k,r)$-partition $\cP$, we will always assume that $\cE=\cE_{\cP}$. 

As the notion of $(\cP,\cE)$-rainbow is central in the paper, we close this section by repeating the definition of a rainbow clique in the case of an $(a,b,k,r)$-partition $\cP$ (and its corresponding $\cE_{\cP}$). 

\begin{definition}\label{def:rainbow-copies-abkr-partition}
    Let ~$\cP = (A_1,A_2,B_1,\ldots B_{r-k-2},Z)$ be an $(a, b, k, r)$-partition and let $F\subseteq K_n$ be a copy of $K_r$. If $k<r-2$, we say that $F$ is \emph{$\cP$-rainbow} if for every $(i,j)\in [2]\times [r-k-2]$ we have 
    \[
        V(F)\cap A_i\neq \emptyset \quad \text{and}\quad V(F)\cap B_{j}\neq \emptyset.
    \]
    If $k=r-2$, we say that $F$ is \emph{$\cP$-rainbow} if $V(F)\cap A_i\neq \emptyset$ for both $i\in \{1,2\}$. In both cases, recall that we denote the set of rainbow copies of $K_r$ by $\cR=\cR(\cP)$.
\end{definition}

\section{Typical properties of Random Graphs}\label{sec:typical-properties}
In this section, we prove several typical properties of random graphs, which will be frequently used throughout the proof of the main Lemmas \ref{lemma:characteristic function integrable}, \ref{lemma:characteristic function integrable - low freq and dense p}, and \ref{lemma:characteristic function integrable - low freq and low p}. 

We start with typical properties concerned with the sign function defined in Section~\ref {subsection: decoupling new}. More specifically, in the sequel, we will be interested in bounding expectations of the form $\E\left[ \left(\sum_{F}S_{G_0, G_1}(F)\right)^L\right]$, where $G_0$ and $G_1$ are random as in Corollary \ref{cor: the decoupling lemma new}, $F\subseteq K_n$ is a copy of $K_r$, and its sign is taken with respect to some fixed $(a,b,k,r)$-partition. This naturally leads to seek conditions on $F_1,\ldots ,F_L$ that imply $\E\left[\prod _{i=1}^{L} S_{G_0, G_1}(F_i)\right]=0$. We note that by Claim \ref{claim: sign zero new}, we may restrict only to the case where every $F_i$ is rainbow. 

Throughout the section, let us fix some non-negative integers $a,b,k,n$, and $r\geq3$. As long as no other condition is imposed assume that $a,b\leq n$ and $k\le r-3$. 
Further, fix an $(a,b,k,r)$-partition $\cP=(A_1,A_2,B_1,\ldots,B_{r-k-2},Z)$, set $\cE=\cE_{\cP}$, and write $\cP_{\cE}\coloneqq \bigcup_{\{i,j\}\in \cE}(\cP_i\times \cP_j)$. The following claim provides a condition ensuring that the above expectations vanish.

\begin{claim}\label{claim:sufficient condition for sign zero in expectation new}
     Let $p=p(n) \in [0,1]$ and suppose that $G_0 \sim G_{n, p},G_1\sim \left(\cP_{\cE}\right)_p$ are independent. Assume that $F_1,\ldots, F_s\subseteq K_n$ are copies of $K_r$ and that there is some $\{i, j\} \in \cE$ such that 
        \[
            \left(V(F_s) \cap \cP_i\right) \cap \left(\bigcup_{t \neq s} V(F_{t})\right) = \emptyset.
        \]    
    Then, for every $\Gamma\subseteq K_n$ and an event $\cD$ independent of the edges of $G_0$ and $G_1$ in $F_s\cap (\cP_i\times \cP_j)$,
    \[
        \E\left[\prod_{t=1}^{s}S_{G_0, G_1}(F_{t})\mid \cD, (G_0\cup G_1 =\Gamma)\right] = 0.
    \]
\end{claim}

\begin{proof}
        Fix a pair $\{i,j\}\in \cE$, a graph $\Gamma\subseteq K_n$, graphs $F_1,\ldots ,F_s$, and an event $\cD$ as in the claim. Let us also set $W = F_s\cap (\cP_i\times\cP_j)$ and for both $t\in\{0,1\}$ define
    \[  
        \bar{G}_t\coloneqq (G_t\setminus W)\cup (G_{1-t}\cap W).
    \]
    
    For every copy $F \subseteq K_n$ of $K_r$, the sign $S_{G_0,G_1}(F)$ is completely determined by $F\cap G_0$ and $F\cap G_1$. In particular, for every $t\neq s$, since $F_t\cap W=\emptyset$, the random variables $S_{G_0,G_1}(F_t),S_{\bar{G}_0,\bar{G}_1}(F_t)$ are identical. 
    Moreover, for $F_s$, Fact \ref{fact:changing-G0-G1} implies that $S_{G_0,G_1}(F_s)=-S_{\bar{G}_0,\bar{G}_1}(F_s)$ deterministically.   
    Therefore, for every pair of graphs $(\Gamma_0,\Gamma_1)$, upon setting the event 
    \[
    \cF(\Gamma_0,\Gamma_1)\coloneqq \cD \cap \left\{ \left(G_{0}\setminus W=\Gamma_0 \right)\ \land \ \left(G_{1}\setminus W=\Gamma_1\right)\  \land \ \left( G_0\cup G_1 =  \Gamma \right)\right\},
    \]
    we have
    \begin{align*}
        \E\left[\prod_{t=1}^{s}S_{G_0, G_1}(F_{t})\mid \cF (\Gamma_0,\Gamma_1)\right] &= -\E\left[\prod_{t=1}^{s}S_{\bar{G}_0, \bar{G}_1}(F_{t})\mid \cF (\Gamma_0,\Gamma_1)\right].
    \end{align*}

    Further, after conditioning on $\cF(\Gamma_0,\Gamma_1)$, the conditional distributions of $(G_0, G_1)$ and $(\bar{G}_0, \bar{G}_1)$ remain identical. Hence,
    \begin{align*}
        \E\left[\prod_{t=1}^{s}S_{G_0, G_1}(F_{t})\mid \cF (\Gamma_0,\Gamma_1)\right] &= -\E\left[\prod_{t=1}^{s}S_{\bar{G}_0, \bar{G}_1}(F_{t})\mid \cF (\Gamma_0,\Gamma_1)\right]\\
        &= -\E\left[\prod_{t=1}^{s}S_{{G}_0, {G}_1}(F_{t})\mid \cF (\Gamma_0,\Gamma_1)\right],
    \end{align*}
    implying that $\E\bigl[\prod_{t=1}^{s}S_{{G}_0, {G}_1}(F_{t})\mid \cD ,(G_0\cup G_1)=\Gamma \bigr]=0$ as required.    
\end{proof}

When analysing expectations of the form $\E\left[ \left(\sum_{F}S_{G_0, G_1}(F)\right)^L\right]$, it will become evident that it is sufficient to restrict the sum to copies $F \subseteq K_n$ of $K_r$ that contain the same number of vertices in every part of the partition $\cP$. Furthermore, Claim~\ref{claim:sufficient condition for sign zero in expectation new} provides a condition under which the expected value of each summand of $\left(\sum_{F}S_{G_0, G_1}(F)\right)^L$ vanishes. The following definitions are important step towards formalising this discussion.

\begin{definition}\label{def: i, j rainbow}
    For non-negative integers $i,j$, a graph $F\in \cR$ is called \emph{$\cP_{i,j}$-rainbow} if it has exactly $i$ vertices in $\cP_A$ as well as exactly $j$ vertices in $\cP_B$. Denote by $\cR_{i,j}= \cR_{i,j}(\cP)\subseteq \cR$ the set of $\cP_{i,j}$-rainbow copies of $K_r$. 
\end{definition}

\begin{definition}\label{def: count of good rainbow copies}
    Suppose that $L>0$ is an integer and that $p\in [0,1]$. 
    For every graph $G$, define $T_{\cP, L}(G)$ to be the collection of $L$-tuples $(F_1,\ldots ,F_L)$ satisfying the following for every $s \in [L]$:
    \begin{enumerate}
        \item\label{def-count-rainbow-item1} $F_s\in \mathcal{R}(\mathcal{P})$;
        \item $F_s\subseteq G$;
        \item every $v\in \cP_B \cup \cP_A$ does not belong to exactly one $F_t$.
    \end{enumerate}
    Moreover, for every edge $e \in K_n$, we define $\partial_e T_{\cP, L}(G)$ analogously, with the additional condition that $e \in F_s$ for every $s \in [L]$.    
    Similarly, for any pair of integers $i, j$, we define $T_{\cP, L}^{i, j}(G)$ and $\partial_e T_{\cP, L}^{i, j}(G)$ as above, except that condition~\eqref{def-count-rainbow-item1} is replaced by the requirement that $F_s \in \cR_{i, j}(\mathcal{P})$.     
\end{definition}

Throughout the paper, we omit the dependence on $\cP$ from the notation of the above definitions whenever the partition is clear from the context.
The next lemma bounds the expected size of $\partial_e T_{ 2L}^{2, r-k-2}(G_{n, p})$ conditioned on the presence of a small subgraph $H$.

\begin{lemma}\label{cor:upper bound for copies on f}
    Suppose that $p \in [0, 1]$ satisfies
    \begin{enumerate}[label=(A\arabic*)]
        \item \label{item: 1 high moments of alpha_f} $b^{r-k-2} p^{\binom{r-k}{2} - 1} \ge 1$;
        \item \label{item: 2 high moments of alpha_f}  $n^{k} p^{\binom{r}{2} - \binom{r-k}{2}} \ge n^{k'} p^{\binom{r}{2} - \binom{r-k'}{2}}$ for every integer $0 \le k' < k$.
    \end{enumerate}
    Then, for every $C>0$, non-negative integers $L_1,L_2$, an edge $e \in A_1 \times A_2$, and a graph $e \in H \subseteq K_n$ on at most $L_2$ vertices, we have
     \begin{equation}\label{eq:extensions2}
        \Pr\left(\left\lvert\partial_e T_{ 2L_1}^{2, r-k-2}(G_{n, p}) \right\rvert\ge  \left((\log n)^{5r} \cdot b^{r-k-2} n^{2k} p^{2\binom{r}{2} - \binom{r-k}{2} - 1}\right)^{L_1} \mid H \subseteq G_{n, p}\right) \le n^{-C}.
    \end{equation}
\end{lemma}

\begin{proof}
    For brevity, for an integer $\ell$ and a graph $G$, we abbreviate $\partial_e T_{\ell}^{2, r-k-2}(G)$ by $T_\ell(G)$ and set $\Lambda \coloneqq b^{r-k-2} n^{2k} p^{2\binom{r}{2} - \binom{r-k}{2} - 1}$.
    We begin by proving that for every integer $L>0$ (that might depend on $n$) there is $D_0 = O_{r}\left((L_2L)^{4r}\right)$ such that for every $e \in A_1 \times A_2$ and a graph $e\in H\subseteq K_n$ on at most $L_2$ vertices, we have
    \begin{equation}\label{eq:extensions1}
         \E\left[\left|T_{ 2L}(G_{n, p})\right| \mid H \subseteq G_{n, p}\right] =O\left( (D_0\Lambda)^{L}\right).
    \end{equation}
    To do so, we begin by estimating the number of ways to construct $(F_1,\ldots,F_{2L})\in T_{2L}(K_n)$.   
    Note that for every $i \in [2L]$ the sets $V(F_i) \cap \cP_A$ are equal and consist of the two endpoints of $e$.
    Thus, to construct the above sequence, it suffices to bound the number of ways one can specify the sets $V(F_i)\cap V(H),V(F_i) \cap (\cP_{B}\setminus V(H)),$ and $V(F_i) \cap (\cP_Z\setminus V(H))$ for every $i$.

    To specify the intersections $V(F_i)\cap V(H)$, we consider the vertices of the graphs $F_i$ one at a time. For each such vertex, we either choose an image in $V(H)$ or leave its image unspecified. Since $|V(H)|\le L_2$ and $\sum_{i=1}^{2L}|V(F_i)|\le 2rL$, it follows that the number of possible choices for the intersections $V(F_i)\cap V(H)$ is at most $(L_2+1)^{2rL}$.

    We continue by specifying the intersections $V(F_i) \cap (\cP_{B}\setminus V(H))$. Let $F=F_1 \cup\ldots \cup F_{2L}$ and note that $\left\lvert V(F) \cap \cP_B\right\rvert \le L(r-k-2)$, since by definition every vertex in $\cP_B\cap V(F)$ belongs to at least two cliques among $F_1,\ldots ,F_{2L}$. 
    Let $t \le L(r-k-2)$ denote the number of vertices in $V(F)\cap (\cP_{B}\setminus V(H))$. 
    Since $|\cP_B|  = rb$, there are at most $(rb)^{t}$ ways to choose those vertices, and then at most $(t+1)^{2rL} \le (rL)^{2rL}$ ways to choose the intersections $V(F_i) \cap (\cP_{B}\setminus V(H))$.

    It remains to choose the intersections $V(F_i) \cap (\cP_Z\setminus V(H))$. By the definition of $\cR_{2, r-k-2}(\cP)$, every graph $F_i$ contains $k$ vertices in $\cP_Z$. Then, for each $i\in [2L]$, given $F_1,\ldots,F_{i-1}$, there are at most $(2rL)^{r}$ ways to identify $u_i\leq k$ vertices in the existing set $V\left(\bigcup _{j=1}^{i-1}F_j\right)$ that also belong to $V(F_{i})\cap (\cP_Z \setminus V(H))$. The remaining $v_i\leq k-u_i$ vertices of $F_{i}$ are then chosen in at most $n^{v_i}$ ways from $\cP_Z\setminus\left(V(H)\cup V\left(\bigcup _{j=1}^{i-1}F_j\right)\right)$.
    Overall, given sequences $u_i,v_i$, these choices account for at most $(2rL)^{2rL} \prod_{i=1}^{{2L}} n^{{v_i}}$ options.

    Finally, for a given $(F_1, \dots, F_{2L})$, letting $F=F_1 \cup \ldots \cup F_{2L}$, we estimate the probability that $F \subseteq G_{n, p} \cup H$. As before, let $t = |V(F)\cap (\cP_{B}\setminus V(H))| \le L(r-k-2)$. 
    Note that every vertex $v \in V(F)\cap (\cP_B \setminus V(H))$ has at least $r - 1 - k$ neighbours in $\cP_A \cup \cP_B$ within $F$, while every vertex $v \in V(F) \cap \cP_A$ (i.e.\ an endpoint of $e$) has $t$ neighbours in $\cP_B \setminus V(H)$. Consequentially, the subgraph of $F$ induced on $V(F)\cap \left(\cP_{B}\cup \cP_A\right)$ has at least $\frac{t(r-k-1)+2t}{2}$ edges not in $H$. 
    For every $i\in [2L]$, recall that $v_i$ is the number of vertices in $V(F_i) \cap (\cP_Z \setminus V(H))$ not appearing in earlier $F_j$, i.e.\ not belonging to $V\left(\bigcup_{j<i} F_j\right)$. These vertices contribute at least $\binom{r}{2}-\binom{r-v_i}{2}$ edges to $F\setminus H$.

    Summing over all possible choices of $t$ and $v_i$, we obtain the following upper bound:
     \begin{align*}
        \E\big[|T_{2L}(G_{n,p})| \mid H \subseteq G_{n, p}\big] &\leq (D_0/r^3)^L   \sum_{t=0}^{L(r-k-2)} b^t p^{\frac{t(r-1-k) + 2t}{2}} \cdot \prod_{i=1}^{2L} \left(\sum_{v_i = 0}^{k} n^{v_i} p^{\binom{r}{2} - \binom{r-v_i}{2}}\right) \\
        &\leq  (D_0/r)^L \sum_{t=0}^{L(r-k-2)}  b^t p^{\frac{t(r+1-k)}{2}} \cdot \left( n^{k} p^{\binom{r}{2} - \binom{r-k}{2}}\right)^{2L},
    \end{align*}
    where in the second inequality we used \ref{item: 2 high moments of alpha_f}. Observe that by \ref{item: 1 high moments of alpha_f}, we have
    \[
        b p^{\frac{r-k+1}{2}} = {\left(b^{r-k-2} p^{\binom{r-k}{2} - 1}\right)^{\frac{1}{r-k-2}}} \ge 1.
    \]
    Therefore, as $rL\leq r^L$ always, we also have 
    \[
        \E\big[|T_{2L}(G_{n,p})| \mid H \subseteq G_{n, p}\big]  =O \left(\left(D_0 \cdot b^{r-k-2} \cdot n^{2k} \cdot  p^{\frac{(r-k-2)(r+1-k)}{2}}p^{2\binom{r}{2} - 2\binom{r-k}{2}}\right)^{L}\right),
    \]
    concluding the proof of \eqref{eq:extensions1} as $\frac{(r-k-2)(r+1-k)}{2}=\binom{r-k}{2}-1$.

    Turning to the proof of \eqref{eq:extensions2}, without loss of generality, assume that $C$ is a sufficiently large constant and let $L' \coloneqq \lfloor C^2 \log n \rfloor$.
    Let $M=L_1L'$, let $D_0 = O_r((L_2L)^{4r})$ be such that \eqref{eq:extensions1} holds with respect to $L$, and fix $D\coloneqq (3D_0)^{L_1}$. By Markov's inequality 
    \begin{align*}
        \Pr\left(|T_{2L_1}(G_{n,p})| \ge D\cdot \Lambda^{L_1} \mid H \subseteq G_{n, p}\right) &\le \frac{\E\left[\left|T_{2L_1}(G_{n,p})\right|^{L'} \mid H \subseteq G_{n, p}\right]}{D^{L'}\cdot  \Lambda^{L_1 L'}}.
    \end{align*}

    Note that
    \[
        \E\left[\left|T_{2L_1}(G_{n,p})\right|^{L'} \mid  H\subseteq G_{n, p}\right] \le \E\big[\left|T_{2L_1L'}(G_{n,p})\right| \mid H \subseteq G_{n, p}\big].
    \]
    Indeed, expanding the expectation on the left-hand side, we get $L'$ many $2L_1$-tuples of $\cP_{2, r-k-2}$-rainbow copies lying on $e$ with the property that every vertex in $\cP_B$ appears in either no copy or in at least a pair of copies. Concatenating $L'$ such tuples gives a tuple of length $2L_1L'$ without vertices in $\cP_B$ appearing in exactly one clique as well, i.e.\ belonging to $T_{2L_1L'}(G_{n,p})$. Then, we obtain
    \begin{align*}
        \Pr\left(\left|T_{2L_1}(G_{n,p})\right| \ge D\cdot\Lambda^{L_1} \mid H \subseteq G_{n, p}\right) &\le \frac{\E\big[\left|T_{2L_1L'}(G_{n,p})\right| \mid H \subseteq G_{n, p}\big]}{D^{L'}\cdot \Lambda^{L_1  L'}}.
    \end{align*}
    Thus, by \eqref{eq:extensions1},
    \[
        \Pr\left(|T_{2L_1}(G_{n,p})| \ge D \cdot \Lambda^{L_1} \mid H \subseteq G_{n, p}\right) =O\left( \left(\frac{D_0 \cdot \Lambda}{D^{1/L_1} \cdot \Lambda}\right)^{L_1 L'} \right)=  O\left(3^{-L_1L'}\right) \leq  n^{-C},
    \]
    where the last inequality holds by our assumption that $C$ is large enough. This completes the proof as $D_0 =O_r((L_2L_1L')^{4r}) =o((\log n)^{5r}) $ which implies that
    $D = (3 D_0)^{L_1} \le \left(\log n\right)^{5rL_1}$.
\end{proof}

Suppose that $G$ is a graph on the vertex set $[n]$ and suppose that $S\subseteq V(G)$. An $ r$-vertex graph $F\subseteq G$ is a \emph{$K_r$-extension of $S$} (in $G$) if $S\subseteq V(F)$, and $E(F) = \binom{V(F)}{2} \setminus \binom{S}{2}$. 
The following claim provides a condition ensuring that, typically, any $S\subseteq [n]$ of size $k$, extends to many copies of $K_r$ in $G_{n,p}$. For $S\subseteq [n]$, let $X_{S}$ denote the random variable counting the number of $K_r$-extensions of $S$ in $G_{n,p}$. We note that if $S = \emptyset$, then we simply have $X_{S} = X_r$.
The following is the well-known extension theorem of Spencer \cite{Spe1990}.
\begin{theorem}[\cite{Spe1990}]\label{claim: lower bound on extensions}
    Let $r$ and $0 \le k \le r$ be integers and let $C > 0$ be a constant. There exists a constant $C_0 > 0$ such that the following holds. Suppose that $n^{r-k} p^{\binom{r}{2} - \binom{k}{2}} \ge C_0 \log n$.
    Then, 
    \[
        \Pr\left(\exists S\subseteq [n] \text{ of size }k \text{ such that } X_{S}< \frac{1}{2r^r}n^{r-k}p^{\binom{r}{2}-\binom{k}{2}} \right) \leq n^{-C}.
    \]
\end{theorem}

\begin{remark}\label{remark: lower bound on extensions}
    Fix some $\epsilon>0$. By adjusting $C$ and $C_0$ in Theorem \ref{claim: lower bound on extensions} it is immediate to show that for a fixed set $Z\subseteq [n]$ of size at least $\epsilon n$:
    \[
        \Pr\left(\exists S\subseteq [n] \text{ of size }k \text{ such that } X'_{S}< \frac{1}{2r^{r}}|Z|^{r-k}p^{\binom{r}{2}-\binom{k}{2}} \right) \leq n^{-C},
    \]
    where $X_S'$ counts the number of $K_r$-extensions of $S\subseteq [n]$ where all the additional vertices were taken from $Z$.  
\end{remark}

Throughout the paper, we will also use a simple modification of Theorem \ref{claim: lower bound on extensions} where the cliques are taken from specified sets. 
\begin{claim}\label{claim: lower bound on extensions distinct}
    Let $r$ and $0 \le k \le r$ be integers and $\epsilon>0$. Suppose $B_1,\ldots ,B_{r-k}\subseteq [n]$ are each of size $b
    $, such that $b^{r-k} p^{\binom{r}{2} - \binom{k}{2}} \ge (\log n)^{(1+\epsilon)r}$.
    For every $S \subseteq [n]\setminus \bigcup_{i=1}^{r-k} B_{i}$ of size $k$ denote by $\hat{X}_{S}$ the number of extension of $S$ to an $r$-clique with one vertex in each $B_i$. Then, 
    \[
        \Pr\left(\exists S\subseteq [n]\setminus \bigcup_{i=1}^{r-k}B_i \text{ of size }k \text{ such that } \hat{X}_{S}< \frac{1}{2}b^{r-k}p^{\binom{r}{2}-\binom{k}{2}} \right) = n^{-\omega(1)}.
    \]
\end{claim}
\begin{proof}
    We prove this claim by a standard application of Theorem \ref{theorem: Janson} and the union bound.    
    Fix $S \subseteq [n] \setminus \bigcup_{i=1}^{r-k} B_{i}$ of size $k$. Denote by $\cF_S$ the family of $K_r$-extensions of $S$ taking exactly one vertex in each $B_i$. Towards application of Theorem \ref{theorem: Janson}, set 
    \[
        \mu = \sum_{F \in \cF_S} \Pr(F \subseteq G_{n, p}) \quad \text{and} \quad \Delta = \sum_{\substack{F_1 \neq F_2 \in \cF_S \\ E(F_1) \cap E(F_2) \neq \emptyset}} \Pr\left(F_1 \cup F_2 \subseteq G_{n, p}\right).
    \]
    Clearly, for every $F\in \cF_{S}$ we have $\Pr(F\subseteq G_{n,p}) = p^{\binom{r}{2}-\binom{k}{2}}$ and therefore, 
    \[
        \mu = \left\lvert \cF_S\right\rvert \cdot p^{\binom{r}{2} - \binom{r-k}{2}} = b^{r-k} p^{\binom{r}{2} - \binom{k}{2}} \ge (\log n)^{(1+\epsilon)r}.
    \]

   We now turn to the analysis of $\Delta$, beginning with the following observation: 
   \begin{align}\label{eq:Delta}
        \Delta \le \mu \cdot \max_{F_1 \in \cF_S} \sum_{\substack{F_1 \neq F_2 \in \cF_S \\ F_1 \cap F_2 \neq \emptyset}} \Pr\left(F_2 \subseteq G_{n, p} \mid F_1 \subseteq G_{n, p}\right).
    \end{align}
    To further bound \eqref{eq:Delta}, by symmetry, it suffices to fix an arbitrary $F_1 \in \cF_S$ and bound the term inside the maximum for this particular instance. Given such $F_1$, the number of graph $F_2\in \cF_{S}$ intersecting $F_1$ in $k< i < r$ vertices is at most $O\left(b^{r-i}\right)$, and moreover, $e(F_2\setminus F_1) = \binom{r}{2}-\binom{i}{2}$. Hence, 
    \begin{align*}
      \frac{\Delta}{\mu} = O\left( \sum_{i=k+1}^{r} b^{r-i} p^{\binom{r}{2} - \binom{i}{2}}\right) = O\left(\max\left\{1, b^{r-k-1} p^{\binom{r}{2} - \binom{k+1}{2}}\right\}\right) = O\left(\max\left\{1, \frac{\mu}{bp^{k}}\right\}\right) ,
    \end{align*}
    where the middle equality holds due to the convexity of $x \mapsto  b^{r-x} p^{\binom{r}{2} - \binom{x}{2}}$.

    Noting that $\binom{r}{2} - \binom{k}{2} \ge k(r-k)$, our assumption implies that
    \begin{align*}
        b p^{k} \ge \left(b^{r-k} p^{\binom{r}{2} - \binom{k}{2}}\right)^{\frac{1}{r-k}} \ge (\log n)^{1+\epsilon}.
    \end{align*}
    As a result, $\frac{\mu^2}{\Delta} = \Omega\left((\log n)^{1+\epsilon}\right)$, and Theorem \ref{theorem: Janson} applied with $\gamma=1/2$ implies that
    \[
        \Pr\left(\hat{X}_S \le \frac{1}{2} 
        b^{r-k} p^{\binom{r}{2} - \binom{r-k}{2}}\right) \le e^{-\Omega\left((\log n)^{1+\epsilon}\right)}.
    \]
    The proof now follows by applying the union bound over the at most $n^k$ choices of $S$.
\end{proof}

We will also require a bound on the upper tail probability of $X_S$. This is the content of the next claim, which follows from a simple adaptation of the argument of Janson, Oleszkiewicz, and Ruci\'nski \cite{JanOleRuc2004}.

\begin{claim}\label{claim:upper bound on extensions}
    Let $r \ge 3$ be an integer and let $p \in [0,1]$. Then, for every $C>0$, we have
    \[
        \Pr\left(\exists S\subseteq [n] \text{ such that }X_S \geq \max\left\{r^2 n^{r-|S|} p^{\binom{r}{2} - \binom{|S|}{2}}, r^2(rC\log n)^{r}\right\}\right) \le n^{-C}.
    \]
\end{claim}
\begin{proof}
    Observe that for every $S \subseteq [n]$ of size at least $r+1$, we have $X_S=0$. Therefore, by setting $M \coloneqq \max\left\{n^{r-s} p^{\binom{r}{2} - \binom{s}{2}}, (rC\log n)^{r}\right\}$ where $s=|S|$, it suffices to prove that for every $S\subseteq [n]$ of size at most $r$ we have
    \[
     \Pr\left(X_S \ge r^2 M \right) \le n^{-C}.
   \] 
    
    Fix $S\subseteq [n]$ of size $s\leq r$. We begin by showing that for every $L\leq C\log n\eqqcolon L^*$, we have $\E[X_S^{L}] \leq rM \cdot \E[X_S^{L-1}]$.
    This clearly holds for $L=1$ as 
    \[
        \E[X_S] = \binom{n-s}{r-s} p^{\binom{r}{2} - \binom{s}{2}} \le  n^{r-s} p^{\binom{r}{2} - \binom{s}{2}}\le M.
    \]
    For $L>1$ by setting $\cF_S$ to be the collection of $K_r$-extensions of $S$, we have
    \begin{align*}
        \E\left[(X_S)^L\right] &= \sum_{F_1, \dots, F_L \in \cF_S} \Pr\left(\bigcup_{i=1}^{L} F_i \subseteq G_{n, p}\right) \\
        &= \sum_{F_1, \dots, F_{L-1} \in \cF_S} \Pr\left(\bigcup_{i=1}^{L-1} F_i \subseteq G_{n, p}\right)  \cdot\sum_{F_L \in \cF_S} \Pr\left(F_L \subseteq G_{n, p} \ \bigg \vert\ \bigcup_{i=1}^{L-1} F_i \subseteq G_{n, p}\right).
    \end{align*}
    By setting $j$ to be the number of vertices $F_{L}$ shares with $\bigcup_{i=1}^{L-1}F_i$ (including vertices in $S$), we further have
    \begin{align*}
        \E\left[(X_S)^L\right] \le \E\left[(X_S)^{L-1}\right]  \cdot \sum_{j=s}^{r} (rL)^{j-s} n^{r-j} p^{\binom{r}{2} - \binom{j}{2}} 
        \le   r M \cdot \E\left[(X_S)^{L-1}\right],
    \end{align*}
    where the last inequality holds since $L\leq L^*$ and since $j \mapsto (rL^*)^{j}\cdot n^{r-j} p^{\binom{r}{2} - \binom{r-j}{2}}$ is convex and hence it attains its maximum in either $j=s$ or $j=r$. By repeatedly applying this inequality we obtain $\E\left[(X_S)^{L^*}\right] \le \left(r M\right)^{L^*}$. The proof is concluded by applying Markov's inequality as follows
    \[
        \Pr\left(X_S \ge r^2 M\right) \le\frac{\E\left[X_S^{L^*}\right]}{(r^2 M)^{L^*}} \le \frac{\left(r M\right)^{L^*}}{(r^2 M)^{L^*}}  = r^{-C  \log n} \le n^{-C}.\qedhere
    \]
\end{proof}

\section{Derivation of Theorem \ref{theorem:main theorem} from Lemmas \ref{lemma:characteristic function integrable}, \ref{lemma:characteristic function integrable - low freq and dense p}, and \ref{lemma:characteristic function integrable - low freq and low p}}\label{sec:quantitative-central-limit}

Our goal in this section is to derive Theorem \ref{theorem:main theorem} from our key Lemmas \ref{lemma:characteristic function integrable}, \ref{lemma:characteristic function integrable - low freq and dense p}, and \ref{lemma:characteristic function integrable - low freq and low p}. To this end, we will use the Fourier inversion theorem. The argument then splits in two: bounding the `high' frequencies and bounding the `low' frequencies. For the `high' frequencies, we will use our key Lemmas \ref{lemma:characteristic function integrable} ,\ref{lemma:characteristic function integrable - low freq and dense p}, and \ref{lemma:characteristic function integrable - low freq and low p}, each in different regime of $p$, whereas for the `low' frequencies we will use Stein's method \cite{Ros2011} and the quantitative central limit theorem of Barbour, Karo\'nski, and Ruci\'nski \cite{BarKarRuc1989}. Let us mention that we will mainly follow \cite[Section 4]{AraMat2023}, where we make the necessary adjustments and changes from the case of the triangle to an arbitrary clique.

As mentioned, the first step is to use the classical Fourier inversion formula (see e.g.\ Chapter~3 in \cite{Dur2010}). Recall that $X^*_r$ is supported on the lattice $\mathcal{L}_r \coloneqq \frac{1}{\sigma_r}\left(\mathbb{Z}-\E[X_r]\right)$. The Fourier inversion formula for lattice-supported random variables then asserts that for every $x\in \mathcal{L}_r$ we have
\[
    \sigma_r \cdot \mathbb{P}(X^*_r=x) = \frac{1}{2\pi} \int_{-\pi\sigma_r}^{\pi\sigma_r} e^{-itx}\E[e^{itX^*_r}]\diff t.
\]
Moreover, let $\mathcal{N}(x)\coloneqq \frac{1}{\sqrt{2\pi}}e^{-x^2/2}$ be the density function of the standard normal distribution, and recall that the characteristic function of the standard normal distribution is $t\mapsto e^{-t^2/2}$ (see Example 3.3.5 in \cite{Dur2010}). Then, using the Fourier inversion formula, this time for absolutely continuous random variables, for every $x\in \mathbb{R}$ we have, 
\[
    \mathcal{N}(x) = \frac{1}{2\pi}\int_{-\infty}^{\infty}e^{-itx}e^{-t^2/2}\diff t.
\]
As characteristic functions are always even, both $\left\lvert\E\left[e^{it X_r^*}\right]\right \rvert$ and $e^{-t^2/2}$ are even functions (in $t$). Hence, by the above and by the triangle inequality, for every $K \in [0, \pi \sigma_r]$ and $x \in \cL_r$, we have
\begin{equation}\label{eq: integral bound}
\begin{aligned}
    \big\lvert\sigma_r \cdot \mathbb{P}(X^*_r= x)& - \mathcal{N}(x)\big\rvert \\ &\leq\int^{K}_{-K}\left\lvert{\E\left[e^{it X_r^*}\right] -e^{-t^2/2}}\right\rvert \diff t +2\int^{\pi\sigma_r}_{K}\left\lvert\E\left[e^{itX^*_r}\right]\right\rvert \diff 
        t+2\int^{\infty}_{K}e^{-t^2/2}\diff t.
\end{aligned}
\end{equation}

Now we aim to bound the middle integral in \eqref{eq: integral bound}. When $p\ge n^{-\frac{1}{m_2(K_r)}}(\log n)^{^{\frac{4r}{e(K_r)-1}}}$ the straightforward bounds supplied by Lemmas \ref{lemma:characteristic function integrable} and \ref{lemma:characteristic function integrable - low freq and dense p} will suffice for our later use, while under the complementary assumption on $p$, a slightly technical application of Lemma~\ref{lemma:characteristic function integrable - low freq and low p} is required. We begin with the former case.

Fix a large constant $D>0$, assume that $n^{-\frac{1}{m_2(K_r)}}(\log n)^{^{\frac{4r}{e(K_r)-1}}} \le p \le \frac{1}{2}$, and let $c,\epsilon>0$ be some constants with $c<\frac{3}{2r}$.
As long as $p\leq n^{-c}$ and $n^{1+\epsilon}p^2 \le t^2 \le (\pi \sigma_r)^2$, Lemma \ref{lemma:characteristic function integrable} implies that $\left\lvert\E\left[e^{itX_r^*}\right]\right\rvert \leq t^{-D}$. In the other case, where $p\geq n^{-c}$, by combining Lemmas \ref{lemma:characteristic function integrable} and \ref{lemma:characteristic function integrable - low freq and dense p}, we have $\left\lvert\E\left[e^{itX_r^*}\right]\right\rvert \leq t^{-D}$ for every $n^{\epsilon} \le t^{2} \le (\pi \sigma_r)^2$.
In particular, if $p\geq n^{-c}$ and $K^2\geq n^{\epsilon}$ or if $p\leq n^{-c}$ and $K^2\geq n^{1+\epsilon}p^{2}$, the second integral in \eqref{eq: integral bound} satisfies the following:
\begin{equation}\label{eq:error-fourier1}
    \int^{\pi\sigma_r}_{K}\left\lvert\E[e^{itX^*_r}]\right\rvert \diff t = O\left(\sigma_r\cdot K^{-D}\right).
\end{equation}
Moreover, provided that $K$ is large enough, we have  
\begin{equation}\label{eq:error-fourier2}
     \int^{\infty}_{K}e^{-t^2/2} \diff t \leq 2e^{-K^2/2}.
\end{equation}
Thus, by taking $D$ to be large enough as a function of $\epsilon$, assuming that $p\leq n^{-c}$, and by combining \eqref{eq: integral bound},\eqref{eq:error-fourier1}, and \eqref{eq:error-fourier2} we find the following for every $K^2\geq n^{1+\epsilon}p^{2}$ and every $x\in \cL_r$: 
\begin{equation}\label{eq: final bound for diff}
    \big\lvert\sigma_r \cdot \mathbb{P}(X^*_r= x) - \mathcal{N}(x)\big\rvert \leq \int^{K}_{-K}\left\lvert{\E\left[e^{it X_r^*}\right] -e^{-t^2/2}}\right\rvert \diff t + o\left(n^{\epsilon} p^{3/2}\right),
\end{equation}
where the error term is in fact $n^{-\Omega(D)}$ and the precise term was chosen with foresight.
Similarly, letting $D$ be sufficiently large as a function of $\epsilon$, if $p\geq n^{-c}$ then by combining \eqref{eq: integral bound},\eqref{eq:error-fourier1}, and \eqref{eq:error-fourier2}, the following holds for every $K^{2} \geq n^{\epsilon}$ and every $x\in \cL_r$: 
\begin{equation}\label{eq: final bound for diff2}
    \big\lvert\sigma_r \cdot \mathbb{P}(X^*_r= x) - \mathcal{N}(x)\big\rvert \leq \int^{K}_{-K}\left\lvert{\E\left[e^{it X_r^*}\right] -e^{-t^2/2}}\right\rvert \diff t + o\left(\frac{1}{n^{1-\epsilon}\sqrt{p}}\right),
\end{equation}
where, as above, the error term is $n^{-\Omega(D)}$ and the above term was chosen with foresight.

We now turn to the case where $p < n^{-\frac{1}{m_2(K_r)}}(\log n)^{^{\frac{4r}{e(K_r)-1}}}$ where we wish to apply Lemma~\ref{lemma:characteristic function integrable - low freq and low p}. To this end, recall that $\tilde{n}\leq n/r$ is the largest integer such that $\tilde{n}^{r-2}p^{\binom{r}{2}-1}< (\log  n)^{-10}$. Under the above upper bound on $p$, we note that $\tilde{n} \ge  n/(\log n)^{30}$ and thereby $\sigma_r^{2}\leq  (\log n)^{40r}\cdot \tilde{n}^{r}p^{\binom{r}{2}}$. Thus, assuming $\epsilon$ is sufficiently small (as a function or $r$), Lemma \ref{lemma:characteristic function integrable - low freq and low p} applied with a sufficiently large $n$ asserts that for any $t\in [0,\pi\sigma_r]$, we have
\[
    \left|\E\left[e^{it \cdot X / \sigma_r}\right]\right| \le n^{-D} + \exp \left(-\frac{\epsilon t^2}{(\log n)^{40r}} \right).
\]
In particular, for $K \geq (\log n)^{21r}$, the middle integral in \eqref{eq: integral bound} satisfies
\begin{equation}\label{eq:error-fourier3}
     \int_{K}^{\pi \sigma_r} \left|\E\left[e^{it X_r / \sigma_r}\right]\right| \diff t \leq \int_{K}^{\pi \sigma_r} n^{-D} + e^{-\epsilon t^2/(\log n)^{40r}} \diff t = O\left(\sigma_r n^{-D}\right).
\end{equation}

The above computation can be improved if one assumes a stronger upper bound on $p$, which is crucial in order to tackle the problem when $p\leq n^{-1/m(K_r)}\cdot (\log n)^{O(1)}$. Indeed, assume that $n^{r-2} p^{\binom{r}{2} - 1} \le (\log n)^{-10}$, which implies that $\tilde{n} = n/r$ and that $\sigma_r^2 \le  L n^r p^{\binom{r}{2}}$ for some constant $L=L(r)>0$. Under this stronger upper bound on $p$, Lemma \ref{lemma:characteristic function integrable - low freq and low p} implies that for any $D>0$, provided that $n$ is large enough, for every $t\in [0,\pi\sigma_r]$ we have
\[
    \left|\E\left[e^{it \cdot X / \sigma_r}\right]\right| \le n^{-D}+e^{-\epsilon t^2}.
\]
This further implies that if $K$ is large as a function of $\epsilon$, we have 
\begin{equation}\label{eq:error-fourier4}
    \int_{K}^{\pi  \sigma_r} \left|\E\left[e^{it X_r / \sigma_r}\right]\right|  \diff t  \le \int_{K}^{\pi  \sigma_r} n^{-D} \diff t  + \int_{K}^{\infty}e^{-\epsilon t^2}  \diff t  \le 2e^{-\epsilon K^2},
\end{equation}
where the last inequality holds as the first integral is of order $o(1)$ for large enough $D$.

Finally, as in \eqref{eq: final bound for diff} and \eqref{eq: final bound for diff2}, letting $D$ be sufficiently large as a function of $\epsilon$, and by combining \eqref{eq: integral bound},\eqref{eq:error-fourier1} and \eqref{eq:error-fourier3} or \eqref{eq:error-fourier4}, the following holds:
If $K = (\log n)^{21r}$ then, for every $x\in \cL_r$ we have
\begin{equation}\label{eq: final bound for diff3}
    \big\lvert\sigma_r \cdot \mathbb{P}(X^*_r= x) - \mathcal{N}(x)\big\rvert \leq \int^{K}_{-K}\left\lvert{\E\left[e^{it X_r^*}\right] -e^{-t^2/2}}\right\rvert \diff t + o\left(\frac{(\log n)^{44r}}{\sqrt{n^{r}p^{\binom{r}{2}}}}\right),
\end{equation}
and if $K$ is sufficiently large as a function of $\epsilon$, then for every $x\in \cL_r$ we have
\begin{equation}\label{eq: final bound for diff4}
    \big\lvert\sigma_r \cdot \mathbb{P}(X^*_r= x) - \mathcal{N}(x)\big\rvert \leq \int^{K}_{-K}\left\lvert{\E\left[e^{it X_r^*}\right] -e^{-t^2/2}}\right\rvert \diff t + 2e^{-\epsilon K^2}.
\end{equation}
We remark that the error term in \eqref{eq: final bound for diff3} is again in fact much stronger. However, as will become clear next, the bottleneck will be evaluating the integral in \eqref{eq: final bound for diff3}.

We are now ready to argue for the `low frequencies', which we formulate in the following lemma. For the sake of completeness, we prove this lemma for every graph $H$ and arbitrary $K>0$ via the following simple optimisation problem. For every graph $H$ similarly to \cite[Section 3]{JanLucRuc2000} we define
\[
    \psi_H \coloneqq  \min_{F\subseteq H,e(F)>0} n^{v(F)}p^{e(F)}.
\]

\begin{lemma}\label{lem: low frequencies}
    Suppose that $H$ is a graph and that $n^{-1/m(H)}\leq p \leq 1/2$. Then, for every $K=K(n)>0$ we have
    \begin{align*}
        \int^{K}_{-K}\left|{\E\left[e^{it X_H^*}\right] -e^{-t^2/2}}\right| \diff t = O_H\left({K^2}/{\psi_H^{1/2}}\right).
    \end{align*}
\end{lemma}

Before we proceed with the proof of this lemma, let us conclude the proof of Theorem \ref{theorem:main theorem}. First, by noting that $x \mapsto n^{x} p^{x(x-1)/2}$ is convex, it is straightforwards to see that
\[
    \psi_r\coloneqq \psi_{K_r} = \begin{cases}
        n^2 p \quad&\text{if } n^{r-2} p^{\binom{r}{2} - 1} \ge 1,\\
        n^r p^{\binom{r}{2}} \quad &\text{otherwise}.
    \end{cases}
\]
We now use the different proven bounds depending on the value of $p$. Assume that $p\ge n^{-\frac{1}{m_2(K_r)}} (\log n)^{\frac{4r}{e(K_r)-1}}$. Then, by combining~\eqref{eq: final bound for diff} and \eqref{eq: final bound for diff2} with Lemma~\ref{lem: low frequencies} applied with $H=K_r$ and either $K^{2}=n^{\epsilon}$ or $K^2=n^{1+\epsilon}p^{2}$, we find that for every $x\in \cL_r$ we have
\[
    \big\lvert\sigma_r \cdot \mathbb{P}(X^*_r= x) - \mathcal{N}(x)\big\rvert = \begin{cases}
        O\left(\frac{1}{n^{1-\epsilon}\sqrt{p}}\right)\quad&\text{if }p\ge n^{-c},\\
        O\left(n^{\epsilon}p^{3/2}\right)\quad&\text{if }p\le n^{-c}.
    \end{cases} 
\]
On the other hand, if $n^{-\frac{1}{m(K_r)}}(\log n )^{128}\le p \le n^{-\frac{1}{m_2(K_r)}} (\log n)^{\frac{4r}{e(K_r)-1}}$, then $\psi_{r} =\Omega_r\left( n^r p^{\binom{r}{2}}/(\log n)^{4r}\right)$ and by combining \eqref{eq: final bound for diff3} and Lemma \ref{lem: low frequencies} applied with $H=K_r$ and $K=(\log n)^{21r}$ we find that for every $x\in \cL_r$ we have 
\[
    \big\lvert\sigma_r \cdot \mathbb{P}(X^*_r= x) - \mathcal{N}(x)\big\rvert = O\left(\frac{(\log n)^{44r}}{\sqrt{n^r p^{\binom{r}{2}}}}\right).
\]
Lastly, if $n^{-\frac{1}{m(K_r)}}\ll p\leq n^{-\frac{1}{m(K_r)}}(\log n )^{128}$, then we have $\psi_r = n^{r}p^{\binom{r}{2}}$, and by combining \eqref{eq: final bound for diff4} and Lemma \ref{lem: low frequencies} applied with $H=K_r$ and $K^2 = \frac{1}{2\epsilon}\log \left(n^{r}p^{\binom{r}{2}}\right)$, we find that for every $x\in \cL_r$ we have
\[
    \big\lvert\sigma_r \cdot \mathbb{P}(X^*_r= x) - \mathcal{N}(x)\big\rvert =O\left( \frac{\log \left(n^{r}p^{\binom{r}{2}}\right)}{\sqrt{n^{r}p^{\binom{r}{2}}}}\right).
\]

We now prove Lemma \ref{lem: low frequencies}. To this end, we introduce some relevant definitions and results~--- In particular, the quantitative central limit theorem due to Barbur, Karo\'{n}ski, and Ruci\'{n}ski~\cite{BarKarRuc1989} and Stein's method~\cite{Ros2011}. A function $f\colon \mathbb{R}\to \mathbb{C}$ is said to be $1$-Lipschitz if for any $x,y\in \mathbb{R}$ we have $\left|f(x)-f(y)\right|\le \left|x-y\right|$. Setting $\mathcal{W}$ to be the family of all real-valued $1$-Lipschitz functions, for two real-valued random variables $X,Y\in L_1(\mathbb{R})$, their \emph{Wasserstein distance} is defined as follows:
\[
    d_\mathcal{W}(X,Y) \coloneqq \sup_{f\in \mathcal{W}} \left\lvert\E\left[f(X)-f(Y)\right]\right\rvert.
\]
For a function $f\in L^{\infty}(\mathbb{R})$, denote its norm by $\lVert f\rVert$. We can now state the following key result about the standard normal distribution, which constitutes the kernel of Stein's method for approximating random variables by the standard normal distribution.

\begin{theorem}[Theorem 3.1 in \cite{Ros2011}]\label{thm: bound on W distance}
    If $X,Z$ are random variables such that $Z$ has a standard normal distribution, setting $\mathcal{F}=\{f : \lVert f\rVert,\lVert f''\rVert\leq 2,\lVert f'\rVert\le \sqrt{2/\pi}\}$ we have
    \[
        d_\mathcal{W}(X,Z)\le \sup_{f\in \mathcal{F}}\left\lvert\E\left[f'(X)-Xf(X)\right]\right\rvert.
    \]
\end{theorem}

As mentioned, we would like to use a theorem of Barbour, Karo\'nski, and Ruci\'nski \cite{BarKarRuc1989} concerning `decomposable' random variables. To avoid technicalities, we present an immediate corollary from their work. For the full details, see Lemma 1 and Equation $(3.10)$ in \cite{BarKarRuc1989}.

\begin{lemma}\label{lem: a bound on stein's function}
    Suppose that $H$ is a graph with at least one edge. Then, there exists a constant $L=L(H)>0$ such that the following holds. For every $p\le 1/2$ and every twice differentiable $f\colon \mathbb{R} \to \mathbb{R}$ with bounded first and second derivatives, we have
    \[
        \left\lvert\E\left[X^*_Hf(X^*_H)-f'(X_H^*)\right]\right\rvert \le L\lVert f''\rVert\cdot \frac{n^{v(H)}p^{e(H)}}{\sigma_H\psi_H}.
    \]
\end{lemma}

We are now ready to prove Lemma \ref{lem: low frequencies}. 

\begin{proof}[Proof of Lemma \ref{lem: low frequencies}]
Note that 
\[
    \sigma_H^2 = \Omega_H \left(\sum_{F\subseteq H,e(F)\ge 1} n^{2v(H)-v(F)}p^{2e(H)-e(F)}\right) = \Omega_H\left(\frac{n^{2v(H)}p^{2e(H)}}{\psi_H}\right).
\]
Thus, applying Lemma \ref{lem: a bound on stein's function} yields that, there exists some constant $L=L(H)>0$ such that for any twice differentiable $f\colon \mathbb{R} \to \mathbb{R}$ with bounded first and second derivatives, we have
\begin{equation}\label{eq: stein bound}
    \left|\E[X^*_Hf(X^*_H)-f'(X_H^*)]\right| \le {L\lVert f''\rVert}/{\psi_H^{1/2}}.
\end{equation}
In particular, letting $\mathcal{F}=\{f : \lVert f\rVert,\lVert f''\rVert\leq 2,\lVert f'\rVert\le \sqrt{2/\pi}\}$, Theorem \ref{thm: bound on W distance} and \eqref{eq: stein bound} imply that 
\[
    d_{\mathcal{W}}(X_r^*,Z)\leq \sup_{f\in \mathcal{F}}\left|\E[f'(X_H^*)-X_H^*f(X_H^*)]\right| \le 2L/{\psi_H^{1/2}}.
\]
To conclude the proof, note that $\theta\mapsto \cos(\theta)$ and $\theta\mapsto \sin(\theta)$ are $1$-Lipschitz, and thus
\begin{align*}
     \int_{-K}^{K}\left|\E\left[e^{itX_r^*}-e^{itZ}\right]\right|\diff t =&  \int_{-K}^{K}\left|\E\left[\cos\left(tX_r^*\right)-\cos\left(tZ\right)\right]\right|\diff t + \int_{-K}^{K}\left|\E\left[\sin\left(tX_r^*\right)-\sin\left(tZ\right)\right]\right|\diff t \\
     \leq & 2\int_{-K}^{K}d_{\mathcal{W}}(tX_r^*,tZ) \diff t \le 2d_{\mathcal{W}}(X_r^*,Z) \int_{-K}^{K} \left|t\right| \diff t = O_H\left({K^2}/{\psi_H^{1/2}}\right),
\end{align*}
where the second inequality holds since for any $1$-Lipschitz function $f$ the function $x\mapsto \frac{1}{|t|}f(tx)$ is also $1$-Lipschitz.
\end{proof}

\section{Proof of Lemma \ref{lemma:characteristic function integrable} for medium frequencies --- $n^{1/2 + \epsilon } p \le t \le n^{-\epsilon} \sigma_r$}\label{section:main lemma for medium frequencies}

Fix an integer $r\geq3$. Recall that $X_r$ is the random variable counting the number of copies of $K_r$ in $G_{n, p}$, and that $\sigma_r^2$ is its variance. The goal of this section is to prove Lemma \ref{lemma:characteristic function integrable} for a certain range of frequencies. This is stated as follows.

\begin{lemma}\label{lem: bound for mid frequencies}    
    Let $\gamma,K>0$ be constants. For every $p=p(n)$ with $n^{-\frac{1}{m_2(K_r)}}(\log n)^{\frac{4r}{e(K_r) - 1}}\leq p\leq 1/2$, and every $t\in[n^{1/2+\gamma}p,n^{-\gamma}\sigma_r]$, provided that $n$ is large we have
    \begin{equation}\label{eq: main bound for mid frequencies}
        \left\lvert\E\left[e^{it X_r/\sigma_r}\right]\right\rvert \le t^{-K}.
    \end{equation}    
\end{lemma}

The proof of this lemma will exploit the decoupling technique. For various choices of $(a, b, k, r)$-partitions, we will apply Corollary \ref{cor: the decoupling lemma new}, and then use Lemma \ref{lemma:main lemma for decoupling} to bound the resulting expectations.
As we shall see, each such use of Lemma \ref{lemma:main lemma for decoupling} yields a different interval of frequencies for which \eqref{eq: main bound for mid frequencies} holds. Lastly, we will show that the union of these intervals covers $[n^{1/2+\gamma}p,n^{-\gamma}\sigma_r]$, thereby concluding the proof of Lemma \ref{lem: bound for mid frequencies}.

Before we continue with the setup for this section, let us mention that Lemma \ref{lemma:main lemma for decoupling} is a key and novel component of our proof of Lemma \ref{lem: bound for mid frequencies} and thereby of Theorem \ref{theorem:main theorem}. For the sake of clarity, we will only state Lemma \ref{lemma:main lemma for decoupling} and defer its proof for later sections.

Throughout this section, fix a positive real $K$ and some $p \in (0,1/2]$ satisfying
\begin{align}\label{align: p assumption}
     p \ge n^{-\frac{1}{m_2(K_r)}} \cdot (\log n)^{\frac{4r}{e(K_r)-1}}.
\end{align} 
Moreover, fix $\gamma>0$, a sufficiently small constant $\delta = \delta(\gamma) > 0$ and a sequence of constants $\epsilon_{r-3} = \epsilon_{r-3}(\delta) > \epsilon_{r-4} > \dots > \epsilon_0 = 0$ satisfying the following. For every $i \in [r-3]$ the constant $\epsilon_{i-1}$ is sufficiently smaller than $\epsilon_i$, and $\epsilon_{r-3}$ is sufficiently smaller than $\delta$. Let us stress that the quantification of how small the above constants are is given inside the proof of Lemma \ref{lemma:main lemma for decoupling} below.
For every $k \in \{0, \dots, r-3\}$, set\footnote{Formally, $a_k\coloneqq \left\lfloor n^{1-\epsilon_{k}}\right\rfloor$ but as $a_k=\omega(1)$ for all $k$, we will ignore the floor signs, and think of $a_k=n^{1-\epsilon_k}$ as an integer.} $a_k \coloneqq n^{1-\epsilon_k}$, and for a positive integer $b \le a_k$, define
\begin{align*}
    \LB_{k, b} \coloneqq \sqrt{\frac{(\log n)^{14r}}{a_k^2 p} \cdot \frac{\sigma_r^2}{b^{r-k-2} n^{2k} p^{2\binom{r}{2} - \binom{r-k}{2} - 1}}},
\end{align*}
and
\begin{align}\label{align: I_k,a,b UB}
    \UB_{k, b} \coloneqq \begin{cases}
        \frac{1}{(\log n)^{10r}} \cdot \sqrt{\frac{\sigma_r^2}{b^{r-k-2} n^{2k} p^{2\binom{r}{2} - \binom{r-k}{2} - 1}}},& \quad \text{if }k=0, \\
        \frac{n^{-\epsilon_k(k-0.5)}}{(\log n)^{10r}} \cdot \sqrt{\frac{\sigma_r^2}{a_k^{2k+2} b^{r-k-2} p^{2\binom{r}{2} - \binom{r-k}{2}}}}, &\quad \text{if }k \ge 1.
    \end{cases}
\end{align}

The following lemma is a key ingredient in the proof of Lemma~\ref{lem: bound for mid frequencies}. It establishes \eqref{eq: main bound for mid frequencies} for every $t \in \left[\LB_{k, b}, \UB_{k, b}\right]$, under suitable conditions on $b$ and $k$.

\begin{lemma}\label{lemma:decoupling for k a b}
    Suppose that $k \in \{0, \dots, r-3\}$ and $0 < b \le a_k$ are integers satisfying
    \begin{enumerate}[label=(P\arabic*)]
        \item\label{item:b condition} $b^{r-k-2} p^{\binom{r-k}{2} - 1} \ge (\log n)^{4r}$, and
        \item\label{item:k condition} $n^{k}\cdot  p^{\binom{r}{2} - \binom{r-k}{2}} \ge n^{\delta}\cdot n^{k'}\cdot  p^{\binom{r}{2} - \binom{r-k'}{2}}$ for every $0\leq k'< k$.
    \end{enumerate}
     Then, for $t \in \left[\LB_{k, b}, \UB_{k, b}\right]$ we have
     \[
        \left\lvert\E\left[e^{it X_r/\sigma_r}\right]\right\rvert \le n^{-K}.
     \]
\end{lemma}

\begin{remark}\label{remark}
    Under the assumptions of Lemma \ref{lemma:decoupling for k a b}, $\UB_{k,b}= n^{O(1)}$. Therefore, since $K$ was chosen to be arbitrary, the above lemma implies that for every $K'>0$ and every $t \in \left[\LB_{k, b}, \UB_{k, b}\right]$, we also have $ \left\lvert\E\left[e^{it \cdot X_r/\sigma_r}\right]\right\rvert \le t^{-K'}$. 
\end{remark}

\begin{remark}\label{remark:values-of-t}
    We note that under the assumptions of Lemma \ref{lemma:decoupling for k a b}, we always have $\UB_{k,b}\ll \sigma_r$. Indeed, for $k=0$, by \ref{item:b condition} we have
    \[
        \UB_{0,b}\le \frac{\sigma_r}{(\log n)^{12r}}\ll \sigma_r.
    \]
    For $k>0$, we have 
    \[
        \UB_{k,b}\le \frac{n^{3\epsilon_k/2}}{(\log n)^{12r}} \cdot \sqrt{\frac{\sigma_r^2}{n^{2k+2}p^{2\binom{r}{2}-2\binom{r-k}{2}+1}}}\le \frac{n^{3\epsilon_k/2}}{(\log n)^{12r}}\cdot \sqrt{\frac{\sigma_r}{n^{2+2\delta}p}}\ll \sigma_r,
    \]
    where the first inequality follows from \ref{item:b condition}, the second follows from \ref{item:k condition} with $k'=0$, and the last inequality follows provided that $\epsilon_k$ is sufficiently small and the lower bound assumption on $p$. 
\end{remark}

The first step in proving Lemma \ref{lemma:decoupling for k a b}, which we do in the next subsection, is to employ Corollary~\ref{cor: the decoupling lemma new}. This reduces the problem to estimating the characteristic function of some signed subgraph counts. To this end, we will distinguish between two types of copies: `good' and `bad'. 
One should think of the good copies as those mainly determining the value of $\left\lvert\E\left[e^{it \cdot X_r/\sigma_r}\right]\right\rvert$.
Informally, we will bound the characteristic function of the sum of all signed copies by a skewed characteristic function of the sum of good signed copies, and a high moment of the sum of bad signed counts. This, again, reduces the problem to bounding these two terms separately, and is achieved in the following lemma. 

Let us recall some notations used frequently in this section and also in the following lemma. Fix an $(a_k, b, k, r)$-partition $\cP = (A_1, A_2, B_1, \dots, B_{r-k-2}, Z)$. We write $\cR(\cP)$ to denote the set of copies of $K_r$ in $K_n$ taking at least one vertex in each of the sets $A_1, A_2, B_1, \dots, B_{r-k-2}$, see Definition \ref{def:rainbow-copies-abkr-partition}. Further, recall that for any pair of integers $i, j$, the set $\cR_{i, j}(\cP) \subseteq \cR(\cP)$ is defined to be the collection of all copies of $K_r\in \cR(\cP)$ with exactly $i$ vertices in $\cP_A$ and $j$ vertices in $\cP_B$. For simplicity, from now on we write $\cR$ and $\cR_{i, j}$ for $\cR(\cP)$ and $\cR_{i, j}(\cP)$, respectively, when the partition is understood from the context. 
Lastly, we will also use the sign function defined in Definition \ref{definition:sign} with respect to the above partition and its corresponding $\cE_\cP$ as given by Definition~\ref{def: (a,b,k)-partition new}.

\begin{lemma}\label{lemma:main lemma for decoupling}
    For every $K>0$ and large enough $n$ the following holds for every $p\in (0,1/2]$ satisfying \eqref{align: p assumption}. Suppose that $k \in \{0, \dots, r-3\}$ and $0 < b \le a_k$ are integers satisfying Properties \ref{item:b condition} and \ref{item:k condition}.    
    Fix an $(a_k, b, k, r)$-partition $\cP$. Let $G_0 \sim G_{n, p}$ and $G_1 \sim \left(\cP_A \times \cP_B\right)_p$ be two independent random graphs and set
    \begin{equation}\label{eq:tildeXY}
        \Tilde{X} \coloneqq \sum_{F \in \cR_{2, r-k-2}} S_{G_0, G_1}(F) \quad \text{and} \quad \Tilde{Y} \coloneqq \sum_{F \in \cR \setminus \cR_{2, r-k-2}} S_{G_0, G_1}(F).
    \end{equation}
    Then, the following holds.
    \begin{enumerate}
        \item\label{item:part1} For every integer $L> 0$ and every $t \in \left[\LB_{k, b}, \UB_{k, b}\right]$, we have
        \begin{equation}\label{align:bound of first item in the main lemma}
             \left|\E_{G_0, G_1}\left[e^{it\Tilde{X} / \sigma_r} \cdot P_{2L-1}\left(\frac{it \cdot \Tilde{Y}}{\sigma_r}\right)\right]\right| \le n^{-K},
        \end{equation}
        where $P_{2L-1}$ is the Taylor polynomial of $e^x$ of order $2L-1$. 
        \item\label{item:part2} For every sufficiently large integer $L = L(\epsilon_1, r)$, we have
        \[
            \E\left[\Tilde{Y}^{2L}\right] = O_L\left(\left(n^{2\epsilon_k(k-1+0.1)} a_k^{2k+2} b^{r-k-2} p^{2\binom{r}{2} - \binom{r-k}{2}}\right)^{L}\right).
        \]
    \end{enumerate}

\end{lemma}

The proof of this lemma is deferred to Sections \ref{section:decoupling} and \ref{section:high moments} and the rest of this section is organised as follows.
In Subsection \ref{section:decoupling k a b}, we derive Lemma \ref{lemma:decoupling for k a b} from Lemma \ref{lemma:main lemma for decoupling}. 
In Subsection \ref{section:cover all range of t}, we show how to patch all intervals from Lemma \ref{lemma:decoupling for k a b} and thus showing that every $t\in [n^{1/2+\gamma}p,n^{-\gamma}\sigma_r]$ is covered, concluding the proof of Lemma \ref{lem: bound for mid frequencies}.
\subsection{Proof of Lemma \ref{lemma:decoupling for k a b}}\label{section:decoupling k a b}

Let $\cP = (A_1, A_2, B_1, \dots, B_{r-k-2}, Z)$ be an $(a_k, b, k, r)$-partition, let $G_0 \sim G_{n, p}$, and let $G_1 \sim \left(\cP_A \times \cP_B\right)_p$. By Corollary \ref{cor: the decoupling lemma new}, for every $t > 0$, 
\begin{align}\label{align:after decoupling}
    \left\lvert\E\left[e^{itX_r/\sigma_r}\right]\right\rvert^{2^{2(r-k-2)}} \le \left\lvert\E_{G_0, G_1}\left[e^{it(\tilde{X} + \tilde{Y})/\sigma_r}\right]\right\rvert,
\end{align}
where $\Tilde{X}$ and $\Tilde{Y}$ are as in \eqref{eq:tildeXY}.

Without loss of generality, assume that $K> 0$ is sufficiently large as a function of $r$. By \eqref{align:after decoupling}, in order to prove Lemma \ref{lemma:decoupling for k a b}, it suffices to show that the following holds for every $t \in [\LB_{b, k}, \UB_{b, k}]$:
\begin{align}\label{align:bound on characteristic}
    \left\lvert\E_{G_0, G_1}\left[e^{it(\tilde{X} + \tilde{Y})/\sigma_r}\right]\right\rvert \le n^{-2^{2r}  K}.
\end{align}

To this end, let $ L= L(K, r, \epsilon_1)$ be some large constant and note that by standard use of Taylor's theorem (see e.g.\ \cite[Lemma 3.3.19]{Dur2010}), for every $t > 0$, we have the following:
\begin{align}\label{align:splitting to main and error}
    \left|\E_{G_0, G_1}\left[e^{it(\tilde{X} + \tilde{Y})/\sigma_r}\right]\right| \le \left|\E_{G_0, G_1}\left[e^{it\Tilde{X}/\sigma_r} \cdot P_{2L-1}\left(\frac{it \Tilde{Y}}{\sigma_r}\right)\right]\right| + O_L\left(\E_{G_0, G_1}\left[\left(\frac{t  \Tilde{Y}}{\sigma_r}\right)^{2L}\right]\right),
\end{align}
where $P_{2L-1}$ is the Taylor polynomial of $e^x$ of order $2L-1$. By the first item in Lemma \ref{lemma:main lemma for decoupling}, for every $t \in [\LB_{k, b}, \UB_{k, b}]$ we may bound the first summand in \eqref{align:splitting to main and error} as follows 
\begin{equation}\label{align:first bound after decoupling}
    \left|\E_{G_0, G_1}\left[e^{it\Tilde{X} / \sigma_r} \cdot P_{2L-1}\left(\frac{it \cdot \Tilde{Y}}{\sigma_r}\right)\right]\right| \le n^{-2\cdot 2^{2r}K}.
\end{equation}
For the second summand in \eqref{align:splitting to main and error}, note that by applying the second item of Lemma \ref{lemma:main lemma for decoupling} we have
\begin{align*}
    \E\left[\Tilde{Y}^{2L}\right] &= O_L\left(\left(n^{2\epsilon_k(k-1+0.1)} a_k^{2k+2} b^{r-k-2} p^{2\binom{r}{2} - \binom{r-k}{2}}\right)^{L}\right).
\end{align*}
This in turn implies that for every $t \le \UB_{k, b}$, 
\begin{equation}\label{align:second bound after decoupling}
\begin{aligned}\E\left[\left({t\Tilde{Y}}/{\sigma_r}\right)^{2L}\right] &= O_L\left(\left(\frac{n^{-2\epsilon_k(k-1+0.5)}}{(\log n)^{20r}} \cdot n^{2\epsilon_k(k-1+0.1)}\right)^{L}\right) \\
   &= O_L\left(\left(\frac{n^{-0.8\epsilon_k}}{(\log n)^{20r} }\right)^L\right) 
    \le n^{-2 \cdot 2^{2r} K},
    \end{aligned}
\end{equation}
where the last inequality holds provided that $L = L(r, \epsilon_1,K)$ is sufficiently large.

The proof is concluded by combining \eqref{align:splitting to main and error}, \eqref{align:first bound after decoupling},  and \eqref{align:second bound after decoupling}, which for every $t \in \left[\LB_{k, b}, \UB_{k, b}\right]$ yield the desired inequality \eqref{align:bound on characteristic}.

\subsection{Proof of Lemma \ref{lem: bound for mid frequencies}}\label{section:cover all range of t}
    Our goal in this subsection is to derive Lemma \ref{lem: bound for mid frequencies} from Lemma~\ref{lemma:decoupling for k a b}. We call $t\in [n^{1/2+\gamma}p,n^{-\gamma}\sigma_r]$ \emph{covered} if there are choices of $k$ and $b$ satisfying Properties \ref{item:b condition} and \ref{item:k condition} such that $t\in[ \LB_{k,b}, \UB_{k,b}]$. With this notation, proving Lemma \ref{lem: bound for mid frequencies} is equivalent to showing that every $t\in [n^{1/2+\gamma} p,n^{-\gamma}\sigma_r]$ is covered.

    Let us start with $k=0$ and determine the values of $t$ which are covered by varying over all possible choices of $b$ in Lemma \ref{lemma:decoupling for k a b}. 
    Note that Property \ref{item:k condition} is vacuous in this case. Let us also note that Property \ref{item:b condition} holds for $b = a_0 = n$. Indeed, by \eqref{align: p assumption},
    \begin{equation}\label{eq:lower-bound-2-density}
       n^{r-2} p^{\binom{r}{2} - 1} \ge (\log n)^{4r}.
    \end{equation}
    Lemma \ref{lemma:decoupling for k a b} (see Remark \ref{remark}), applied with $k=0$ and an integer $b \le a_0 = n$ satisfying Property \ref{item:b condition}, covers the set of frequencies
    \[
        I_{0, b} \coloneqq  \left[\LB_{0, b}, \UB_{0, b}\right].
    \]
    Since $n^2 p \geq (\log n)^{35r}$, for every such $b\leq n-1$ we have $\LB_{0,b+1}< \LB_{0,b}< \UB_{0,b+1}< \UB_{0,b}$ implying that following interval of frequencies is covered\footnote{Here and in the sequel, when $b=\omega(1)$, despite $b$ being an integer, we will ignore rounding errors for simplicity of presentation.}: 
    \[
        I_0 \coloneqq  \bigcup_{\substack{b \le n \\ b^{r-2} p^{\binom{r}{2} - 1} \ge (\log n)^{4r}}} I_{0, b} = \left[\sqrt{\frac{(\log n)^{14r}}{n^2 p} \cdot \frac{\sigma_r^2}{n^{r-2} p^{\binom{r}{2} - 1}}} , \frac{\sigma_r}{(\log n)^{12r}}\right].
    \]  
    
    Note that the values of $k$ for which Lemma \ref{lemma:decoupling for k a b} is applicable depend on $p$. The following definition dictates the values of $k>0$ which we will use to cover every $t\in [n^{1/2+\gamma}p,n^{-\gamma}\sigma_r]$. Let $k^*$ be the smallest positive integer in $[r-2]$ for which the assertion of property \ref{item:k condition} is satisfied (with $k^*=k$), that is
   \begin{equation*}
        n^{k^*} p^{\binom{r}{2} - \binom{r-k^*}{2}} \ge n^{\delta} n^{k'} p^{\binom{r}{2} - \binom{r-k'}{2}} \quad \text{for every}\quad 0 \le k' < k^*.
    \end{equation*}
    If no such integer satisfies the above, set $k^* \coloneqq r-1$.    
    For technical reasons, we will distinguish between the case of $k^* \ge r-2$ and $k^*<r-2$. In the former case, we will show that 
     \[
        [n^{1/2+\gamma} p,n^{-\gamma}\sigma_r] \subseteq I_0,
    \]
    concluding the proof under this assumption. The latter case, where $k^*<r-2$ will be more involved. Indeed, for every $k^*\le k\le r-3$ we will introduce the interval $I_k$ consisting of all $t$ covered by Lemma \ref{lemma:decoupling for k a b} (see Remark \ref{remark}) applied with $k$ and varying over all valid values of $b$. We will then aim to show that 
    \begin{equation}\label{eq:Ik}
         [n^{1/2+\gamma} p,n^{-\gamma}\sigma_r] \subseteq I_0\cup \bigcup_{k=k^*}^{r-3}I_k.
    \end{equation}
    
    We start with the following technical claim bounding $n^{k} p^{\binom{r}{2} - \binom{r-k}{2}}$ for every $k< k^*\le r-1$.
    
    \begin{claim}\label{claim:minimality of k^*}
        For every integer $0 < k < k^*\leq r-1$, we have $ n^{k} p^{\binom{r}{2} - \binom{r-k}{2}} \le n^{\delta k}.$
    \end{claim}
    \begin{proof}
        If $k^*=1$, the claim is vacuous; hence, we assume that $k^*>1$.
        The proof proceeds by induction on $k$. First, note that for $k=1$, the minimality of $k^*$ implies the required inequality.
        For the induction step, fix $k<k^*$ and suppose that the claim holds for every positive integer $k_0 < k$. 
        Then, by the minimality of $k^*$ and the induction assumption, there exists an integer $0 \le k_0 < k$, such that
        \[
            n^{k} p^{\binom{r}{2} - \binom{r-k}{2}} < n^{\delta} n^{k_0} p^{\binom{r}{2} - \binom{r-k_0}{2}}\le n^{(k_0+1) \delta}\leq n^{k\delta}.
        \]
        This is as required.
    \end{proof}

    The next claim concludes the proof of Lemma \ref{lem: bound for mid frequencies} in the case $k^* \ge r-2$.
    
    \begin{claim}
        Suppose that $k^* \ge r-2$. Then, $[n^{1/2+\gamma} p,n^{-\gamma}\sigma_r]\subseteq I_0$.
    \end{claim}
    \begin{proof}
        By the definition of $I_0$ it is enough to show that $ \LB_{0,n}\ll n^{1/2+\gamma}p$, where we recall that $\delta$ is sufficiently small compared to $\gamma$.
        Indeed, since $\sigma_r^2=O\left(n^{2r-2}p^{2\binom{r}{2}-1}\right)$ and by applying Claim~\ref{claim:minimality of k^*} with $k=r-3$, we obtain: 
        \begin{align*}
            \LB_{n,0}&=\sqrt{\frac{(\log n)^{14r}}{n^2 p} \cdot \frac{\sigma_r^2}{n^{r-2} p^{\binom{r}{2} - 1}}} \\&=O\left(\sqrt{(\log n)^{14r} \cdot n^{r-3} p^{\binom{r}{2} - \binom{3}{2}} \cdot n p^2}\right) =  O\left(\sqrt{(\log n)^{14r} \cdot n^{(r-3) \delta} \cdot n p^2}\right).
        \end{align*}     
        Hence, provided that $\delta<{\gamma}/{r}$ we have $\LB_{n,0} \ll n^{1/2+\gamma}p$, as required.
    \end{proof}
  From now on, assume that $k^* < r-2$. By the definition of $k^*$, we have
    \[
        n^{k^*}p^{\binom{r}{2}-\binom{r-k^*}{2}} \geq n^{\delta}\cdot\max\left\{1, n^{k^*-1}p^{\binom{r}{2}-\binom{r-k^*+1}{2}}\right\}.
    \]
    As $x\mapsto n^xp^{\binom{r}{2}-\binom{r-x}{2}}$ is convex, the assertion of Property \ref{item:k condition} holds also for every $k^*\leq k \le r-3$. That is for every $ 0\leq k'< k$ we have
    \[
    n^{k}  p^{\binom{r}{2} - \binom{r-k}{2}}  \ge  n^{\delta}\cdot n^{k'}  p^{\binom{r}{2} - \binom{r-k'}{2}}.
    \]
    Then, by Lemma \ref{lemma:decoupling for k a b} (see Remark \ref{remark}), for every $k^*\le  k\le r-3$ and every $b \le a_k$ for which \ref{item:b condition} is satisfied, the following interval is covered,
    \[
        I_{k, b} \coloneqq \left[\LB_{k,b}, \UB_{k,b}\right].
    \]
    Observe that as $a_k=n^{1-\epsilon_k}$, the right-hand side of the interval $I_{k, b}$ is larger than the left-hand side by a multiplicative factor of 
    \[
        \frac{n^{-\epsilon_k(k-0.5)}}{(\log n)^{17r}} \cdot \left(\frac{n}{a_k}\right)^k = \frac{n^{0.5 \epsilon_k}}{(\log n)^{17r}} \gg 1,
    \]
    and thus $I_{k, b} \neq \emptyset$.  
    Set 
    \[
        \LB_k \coloneqq \sqrt{\frac{(\log n)^{14r}}{a_k^2 p} \cdot \frac{\sigma_r^2}{a_k^{r-k-2} n^{2k} p^{2\binom{r}{2} - \binom{r-k}{2} - 1}}},
    \]
    and
    \[
        \UB_k\coloneqq  \frac{n^{-\epsilon_k(k-0.5)}}{(\log n)^{12r}} \cdot \sqrt{\frac{\sigma_r^2}{a_k^{2k+2} p^{2\binom{r}{2} - 2\binom{r-k}{2} + 1}}}.
    \]
    The interval $I_k$ mentioned above (see \eqref{eq:Ik}) is finally defined as follows, 
    \begin{align*}
        I_{k} & \coloneqq \bigcup_{\substack{b \le a_k \\ b^{r-k-2} p^{\binom{r-k}{2} - 1} \ge (\log n)^{4r}}} I_{k, b} = \left[\LB_k, \UB_k\right].
    \end{align*}
    We note that $I_k$ is indeed an interval. This follows since $\LB_{k,b+1} < \LB_{k,b} < \UB_{k,b+1} < \UB_{k,b}$ for every $b$ as above. This also implies the equality above.
    We further emphasise that $I_k$ is covered, and therefore, it suffices to show that
    \[
        [n^{1/2+\gamma} p,n^{-\gamma}\sigma_r] \subseteq I_0\cup \bigcup_{k=k^*}^{r-3}I_k.
    \]
    
    As the value of $k$ increases, the interval $I_k$ covers smaller frequencies. In particular, to reach frequencies of the order of $n^{1/2+\gamma}p$, we will have to consider $I_{r-3}$.
    To prove the above, we will show that $I_0\cup \bigcup_{k=k^*}^{r-3}I_k$ is an interval with the left endpoint smaller than $n^{1/2+\gamma}p$ and the right endpoint larger than $n^{-\gamma}\sigma _r$ .
    This will be done as follows: first, we show that $\LB_{0,n}\le \UB_{k^*}$. Then, we show that for every $k^*\le k< r-3$ we have $\LB_{k}\le \UB_{k+1}$. 
    As hinted above, we lastly show that $\LB_{r-3}\leq n^{1/2+\gamma} p $, concluding the proof.
    
    \begin{claim}
        $\LB_{0,n}\ll \UB_{k^*}$.
    \end{claim}
    \begin{proof}
        Recalling the definitions of $\LB_{0,n},\UB_{k^*}$ and that $a_{k^*}=n^{1-\epsilon_{k^*}}$, the claim is equivalent to showing that 
        \[
            (\log n)^{14r} \cdot n^{2k^*} p^{2\binom{r}{2} - 2\binom{r-k^*}{2} + 1} \ll \frac{n^{3 \epsilon_{k^*}}}{(\log n)^{24r}} \cdot n^{r-2} p^{\binom{r}{2}}.
        \]
        Note that if $r = 4$, then $k^* = 1$ (as $0<k^* < r-2$ in this case). Thus, in this particular case, the above inequality is equivalent to
            \[
                (\log n)^{14 \cdot 4} \cdot n^{2} p^{2\binom{4}{2} - 2\binom{3}{2} + 1} \ll \frac{n^{3 \epsilon_{1}}}{(\log n)^{24 \cdot 4}} \cdot n^{2} p^{\binom{4}{2}},
            \]
        which clearly holds. From now on, assume that $r > 4$.
        By Claim \ref{claim:minimality of k^*}, $n^{k^* - 1} p^{\binom{r}{2} - \binom{r-k^*+1}{2}} \le n^{k^* \delta}$, and by noting that
        \begin{align*}
            n^{2k^*} p^{2\binom{r}{2} - 2\binom{r-k^*}{2} + 1} = n^{k^* + 1} p^{\binom{r}{2} - \binom{r-k^*-1}{2} + 2} \cdot n^{k^* - 1} p^{\binom{r}{2} - \binom{r-k^*+1}{2}},
        \end{align*}
        it is enough to show that
        \[
            (\log n)^{38r} \cdot n^{k^* + 1} p^{\binom{r}{2} - \binom{r-k^*-1}{2} + 1} \ll n^{3\epsilon_{k^*} - k^* \delta} \cdot n^{r-2} p^{\binom{r}{2} - 1}.
        \]
        
        Assume first that $k^* < r-3$. As $x\mapsto n^{x}p^{\binom{r}{2}-\binom{r-x}{2}}$ is convex and as $n^{k^*}p^{\binom{r}{2}-\binom{r-k^*}{2}}\geq n^{\delta}\geq  1$, we have $n^{k^* + 1} p^{\binom{r}{2} - \binom{r-k^*-1}{2}} \le n^{r-3} p^{\binom{r}{2} - \binom{3}{2}}$. Therefore, 
        \[
             (\log n)^{38r} \cdot n^{k^* + 1} p^{\binom{r}{2} - \binom{r-k^*-1}{2} + 1} \le (\log n)^{38r} \cdot n^{r-3} p^{\binom{r}{2} - \binom{3}{2}+1}  \ll n^{3\epsilon_{k^*} - k^* \delta} \cdot n^{r-2} p^{\binom{r}{2} - 1},
        \]
        where the last inequality holds since $np \gg  (\log n)^{38r} \cdot n^{k^* \delta}$  for all sufficiently small $\delta$.
        Assume now that $k^* = r-3$. Claim \ref{claim:minimality of k^*} implies that in this case $p$ satisfies $n^{r-4} p^{\binom{r}{2} - \binom{4}{2}} < n^{r \delta}$. In particular, $p$ is fairly small, that is, by taking $\delta$ to be small enough we have $p\le n^{-\Omega_r(1)}$ (as $r>4$).
        We have $n^{k^* + 1} p^{\binom{r}{2} - \binom{r-k^*-1}{2}} = n^{r-2} p^{\binom{r}{2} - \binom{2}{2}}$, which implies 
        \[
            (\log n)^{38r} \cdot n^{k^* + 1} p^{\binom{r}{2} - \binom{r-k^*-1}{2} + 1} = (\log n)^{38r} \cdot n^{r-2} p^{\binom{r}{2} - 1} \cdot p \ll n^{3\epsilon_{k^*} - k^* \delta} \cdot n^{r-2} p^{\binom{r}{2} - 1},
        \]
        where the last inequality holds by taking $\delta$ small enough.
    \end{proof}
    \begin{claim}
        For every $k^*\le k\le r-4$, we have $\LB_{k} \ll \UB_{k+1}$.
    \end{claim}
    \begin{proof}
        The claim is equivalent to showing that
        \[
            (\log n)^{14r} \cdot a_{k+1}^{2k + 4} p^{\binom{r}{2} -2\binom{r-k-1}{2} + 1} \ll \frac{n^{-\epsilon_{k+1}(2k+1)}}{(\log n)^{24r}} \cdot a_{k}^{r-k} n^{2k} p^{\binom{r}{2} - \binom{r-k}{2}}.
        \]
        Substituting $a_{k}=n^{1-\epsilon_{k}}$ and $a_{k+1}=n^{1-\epsilon_{k+1}}$, the above is equivalent to showing that 
        \begin{equation}\label{align:some condition}
           (\log n)^{38r} \cdot n^{k + 1} p^{\binom{r}{2} -2\binom{r-k-1}{2} + 1} \ll n^{w} \cdot n^{r - 3} p^{\binom{r}{2} - \binom{r-k}{2}},
        \end{equation}
        where 
        \begin{equation*}
             w \coloneqq -\epsilon_{k+1}(2k+1) + \epsilon_{k+1}(2k + 4) - \epsilon_{k}(r-k) = 3\epsilon_{k+1} - \epsilon_{k}(r-k).
        \end{equation*}
        Note that $w>0$ provided that $r\epsilon_k<\epsilon_{k+1}$, which we may assume.

        Observing that $\binom{r-k-1}{2} = \binom{r-k}{2}-(r-k-1)$, we have
        \begin{align*}
           p^{\binom{r}{2} - 2\binom{r-k-1}{2} + 1} &= p^{\binom{r}{2} - \binom{r-k-1}{2}} \cdot p^{- \binom{r-k}{2} + (r-k-1) + 1} \le  p^{\binom{r}{2} - \binom{r-k-1}{2}} \cdot p^{- \binom{r-k}{2} + 3},
        \end{align*}
        where the last inequality follows from the assumption that $r-k \ge 3$. Hence, to conclude the proof of \eqref{align:some condition}, it suffices to show that
        \[
            (\log n)^{38r}  \cdot n^{k+1} p^{\binom{r}{2} - \binom{r-k-1}{2}} \ll n^{w} \cdot n^{r-3} p^{\binom{r}{2} - \binom{3}{2}}.
        \]
	    Indeed, as Property \ref{item:k condition} holds for all values of $k$ exceeding $k^*$, and as $r-3\ge k+1 > k^*$ we have, $n^{r-3} p^{\binom{r}{2} - \binom{3}{2}} \ge n^{k+1} p^{\binom{r}{2} - \binom{r-k-1}{2}}$. Further, as $w > 0$, we obtain the required inequality:
        \[
            (\log n)^{38r} \cdot n^{k+1} p^{\binom{r}{2} - \binom{r-k-1}{2}} \le (\log n)^{38r} \cdot n^{r-3} p^{\binom{r}{2} - \binom{3}{2}} \ll n^{w} \cdot n^{r-3} p^{\binom{r}{2} - \binom{3}{2}}.\qedhere
        \]
    \end{proof}

    Let us finally show that $\LB_{r-3} \ll n^{1/2 + \gamma} p$, and complete the proof of Lemma \ref{lem: bound for mid frequencies}. Indeed, recalling that $\sigma_r^2=O\left(n^{2r-2}p^{2\binom{r}{2}-1}\right)$, we have
    \begin{align*}
        \LB_{r-3}=\sqrt{\frac{(\log n)^{14r}}{a_{r-3}^2 p} \cdot \frac{\sigma_r^2}{a_{r-3} n^{2r-6} p^{2\binom{r}{2} - \binom{3}{2} - 1}}} &= O\left(\sqrt{(\log n)^{14r} \cdot n^{3\epsilon_{r-3}} \cdot n p^2} \right) \ll n^{1/2 + \gamma} p,
    \end{align*}
    where the last inequality is provided that $3\epsilon_{r-3} < {\gamma}$.

\section{Proof of Lemma \ref{lemma:main lemma for decoupling} part \eqref{item:part1}}\label{section:decoupling}

In this section, we show that for every $K>0$, sufficiently large $n$, and $p\leq 1/2$ satisfying \eqref{align: p assumption}, the first assertion of Lemma \ref{lemma:main lemma for decoupling} holds.
We begin with a high-level overview of the argument.

\subsection{Proof overview}
The first step we take is the following simple yet crucial observation.
The polynomial $P_{2L-1}$ is of degree $2L-1$, and has coefficients bounded by a function of $L$. Thus, we can represent $P_{2L-1}(it \Tilde{Y} / \sigma_r)$ as a sum of monomials in $\{S_{G_0, G_1}(F_i):\text{$F_i$ is a copy of $K_r$}\}$ with coefficients bounded by some function of $L$. By linearity of expectation, as there are at most $n^{(2L-1)r}$ monomials as above, and as $t/\sigma_r \le \pi$, it suffices to show the following for a  large enough constant $D=D(K, L, r)$:
\[
    \left|\E_{G_0, G_1}\left[e^{it\Tilde{X} / \sigma_r} \cdot \mathcal{M}
    \right]\right| \leq n^{-D},
\]
where $\mathcal{M}$ is any monomial of degree at most $2L-1$ with variables taken from $\{S_{G_0, G_1}(F_i):\text{$F_i$ is a copy of $K_r$}\}$.

Crucially, every such $\mathcal{M}$ is a random variable taking values in $\{0,\pm 1\}$ that depends on at most $O_L(1)$ edges. As will be explained later, this implies that  
\[
    \left|\E_{G_0, G_1}\left[e^{it\Tilde{X} / \sigma_r} \cdot \mathcal{M}
    \right]\right| \approx \left|\E_{G_0, G_1}\left[e^{it\Tilde{X} / \sigma_r} \right]\right|,
\]
reducing the problem to controlling the right-hand side above. The strategy after the reduction is to approximate the right-hand side above by the characteristic function of a weighted sum of \emph{independent} Bernoulli random variables. The independence will allow us to evaluate the characteristic function of each summand individually and derive the required bound.

To achieve this approximation, we will expose the randomness of $(G_0, G_1)$ in several steps. We will first expose the union of $G_0$ and $G_1$ on the complement of $A_1\times A_2$. Noting that $(G_0 \cup G_1) \cap (\cP_A \times \cP_B)$ distributes as $(\cP_A\times \cP_B)_q$ with $q=2p-p^2$, we may use various concentration inequalities. In particular, we will establish that every edge in $A_1\times A_2$ extends to many $\cP_{2,r-k-2}$-rainbow copies of $K_r$ in $(G_0\cup G_1) \setminus (A_1\times A_2)$.
Then, exposing the randomness of $(G_0,G_1)$ except for the edges in $A_1\times A_2$, we will show that for a sufficiently large portion of the edges $f\in A_1\times A_2$ the following holds: The sum of $S_{G_0,G_1}(F)$ running over all $F\cong K_r$ containing $f$, is equal to some large deterministic weight (depending on $b,k,n$ and $r$) times the indicator $x_f=\mathbf{1}_{f\in G_0}$. This is an approximation by a weighted sum of independent Bernoulli random variables as required. Specifically, these random variables are the indicators $x_f$ of the above edges with their corresponding weights.

\subsection{Proof of Lemma \ref{lemma:main lemma for decoupling} part \eqref{item:part1}}\label{section:handling P_2L}
Fix some large constant $K=K(r)>0$, a large integer $n$, and  $p\in(0,1/2]$ that satisfies \eqref{align: p assumption}. Assume that $k \in \{0, \dots, r-3\}$ and $0 < b \le a_k$ are integers that satisfy Properties \ref{item:b condition} and \ref{item:k condition}. In addition, fix an $(a_k, b, k, r)$-partition $\cP = (A_1, A_2, B_1, \dots, B_{r-k-2}, Z)$ and let $G_0 \sim G_{n, p}$ and $G_1 \sim \left(\cP_A \times \cP_B\right)_p$ be two independent random graphs. Lastly, let $L > 0$ and $t \in [\LB_{k, b}, \UB_{k, b}]$.

We begin by showing how to get rid of $P_{2L-1}(it\Tilde{Y}/\sigma_r)$ in \eqref{align:bound of first item in the main lemma}. Recall that $\Tilde{Y}$ is the sum of $S_{G_0, G_1}(F)$ running over all $F \in \cR'\coloneqq \cR\setminus \cR_{2,r-k-2}$. As $P_{2L-1}$ is the Taylor polynomial of degree~$2L-1$ of $e^x$, we have
\[    
    P_{2L-1}(it \tilde{Y} / \sigma_r) = \sum_{\ell=0}^{2L-1} \frac{1}{\ell!} \left(it \tilde{Y} / \sigma_r\right)^\ell = \sum_{\ell=0}^{2L-1} \frac{1}{\ell!} (it/\sigma_r)^\ell \sum_{F_1, \dots, F_{\ell} \in \cR'} \prod_{j=1}^{\ell} S_{G_0, G_1}(F_j).
\]
By linearity of expectation, Remark \ref{remark:values-of-t} showing that $t \le \UB_{k,b}\le \pi \sigma_r$, and the triangle inequality,
\[
    \left\lvert \E\left[e^{it\Tilde{X} / \sigma_r} \cdot P_{2L-1}\left(it \Tilde{Y} / \sigma_r\right)\right] \right\rvert \le \pi^{2L-1} \sum_{\ell=0}^{2L-1} \sum_{F_1, \dots, F_{\ell} \in \cR'} \left\lvert\E\left[e^{it\Tilde{X} / \sigma_r} \prod_{j=1}^{\ell} S_{G_0, G_1}(F_j)\right]\right\rvert.
\]
Since there are at most $n^{(2L-1)r}$ sequences $F_1,\ldots,F_{2L-1}\in \cR'$, and the sign function of a graph always belongs to $\{0,\pm1\}$, we have
\begin{align*}
    \bigg\lvert\E\big[e^{it\Tilde{X} / \sigma_r}\cdot P_{2L-1}(it \Tilde{Y} / \sigma_r)\big]\bigg \rvert \le  3L(\pi n)^{(2L-1)r}  \max_{\substack{F_1, \dots, F_{\ell} \in \cR'\\ 0\le \ell\le 2L-1}} \left\lvert\E\left[e^{it\Tilde{X} / \sigma_r} \ \bigg \lvert\  \left|\prod_{j=1}^{\ell} S_{G_0, G_1}(F_j)\right| = 1\right]\right\rvert.
\end{align*}

Therefore, in order to prove \eqref{align:bound of first item in the main lemma}, it suffices to show that for some $D > (2L-1)r+K$, every~$\ell\le 2L-1$, and every $F_1, \dots, F_{\ell} \in \cR'$, we have
\begin{align}\label{align:bound of conditional characteristic function}
    \left\lvert\E\left[e^{it\Tilde{X} / \sigma_r}\  \bigg \lvert \ \left|\prod_{j=1}^{\ell} S_{G_0, G_1}(F_j)\right| = 1\right]\right\rvert \le n^{-D}.
\end{align}
In the rest of the section, we prove \eqref{align:bound of conditional characteristic function}. For this, let us fix $\ell \in \{0,1,\ldots,2L-1\}$, a constant $D>(2L-1)r+K$, some $t \in [\LB_{k, b}, \UB_{k, b}]$, and $F_1, \dots, F_{\ell} \in \cR'$. Moreover, for simplicity of presentation, denote by $\cD$ the event that $\left\lvert \prod_{j=1}^{\ell} S_{G_0, G_1}(F_j)\right\rvert = 1$.

Throughout the proof, Definition \ref{def: count of good rainbow copies} plays a key role. For an edge $f\in K_n$, an integer $s\ge 1$ and a graph $G$ we defined $\partial_f T_{\cP,s}^{2,r-k-2}(G)$ as the collection of $s$-tuples of $r$-cliques $(H_1,\ldots ,H_s)$ with the following properties: First, $H_i\in \cR_{2,r-k-2}$ and $f\in H_i \subseteq G$ for every $i \in [s]$. Second, every $v\in \cP_A \cup \cP_B$ either belongs to no $V(H_i)$ or belongs to $V(H_i)\cap V(H_j)$ for some $i\neq j$. For simplicity of presentation in this section, we will omit some of the dependencies and write $\partial_f T_{s}(G)$ to denote $\partial_f T_{\cP,s}^{2,r-k-2}(G)$.

We note that for every $f\in A_1\times A_2$ we have 
\[
    \E\left[\left\lvert\partial_f T_2(G_{n, p})\right\rvert\mid f\in E(G_{n,p})\right] = \Theta\left(b^{r-k-2} n^{k}p^{\binom{r}{2}-1} \cdot \sum_{i=0}^{k} n^{i}p^{\binom{r}{2}-\binom{r-i}{2}}\right),
\]
where the first term in the right-hand side accounts for the probability that $H_1\subseteq G_{n,p}\cup f$, and the sum accounts for the probability that $H_2\subseteq G_{n,p}\cup H_1$. The assumption \ref{item:k condition} then implies that $\E\left[\left\lvert\partial_f T_2(G_{n, p})\right\rvert\mid f\in E(G_{n,p})\right] = \Theta\left(\Lambda \right)$ where 
\[
    \Lambda\coloneqq  b^{r-k-2} n^{2k} p^{2\binom{r}{2} - \binom{r-k}{2} - 1}.
\]
We further introduce the following notation. For a constant $c_0> 0$ let $\cE(c_0)$ be the collection of graphs $G$ on $[n]$ for which the following properties hold: 
\begin{enumerate}[label=(E\arabic*)]
    \item\label{item: typical event lower bound} For every edge $f \in A_1 \times A_2$ we have $\left\lvert \partial_f T_{2}(G) \right\rvert \ge c_0  \Lambda$.
    \item\label{item: typical event upper bound} For every edge $f \in A_1 \times A_2$ and every $j \in \{1,2\}$ we have $\left\lvert \partial_f T_{2j}(G) \right\rvert \le \left((\log n)^{5r}  \Lambda\right)^j.$
\end{enumerate}
In the following claim, we show that by carefully choosing $c_0$, the graph $G\coloneqq G_0\cup G_1\cup (A_1\times A_2)$ typically belongs to $\cE(c_0)$.

\begin{claim}\label{claim: E_1 and E_2 are typical}
    For any sufficiently small constant $c_0>0$, we have
    \[
        \Pr\left(G\in \cE(c_0)\mid \cD\right) \ge 1 - n^{-2D}.
    \]
\end{claim}
\begin{proof}
    To prove the claim, it is clearly sufficient to show that conditioned on $\cD$, the graph $G$ satisfies each of \ref{item: typical event lower bound} and \ref{item: typical event upper bound} with probability at least $1-n^{-3D}$. We begin with several observations that we utilise in the proof of each of these properties. 
    
    First, note that every edge $f\in K_n$ is present in $G_0 \cup G_1$ with some probability $p_f\in [p, 2p]$, independently of the other edges. In particular, we may assume that $G_{n,p}\subseteq G_0\cup G_1 \subseteq G_{n,2p}$. Letting $F \coloneqq \bigcup_{i=1}^{\ell} F_i$, it is clear that $v(F) \le (2L-1) r$.

    Since we assume Property \ref{item:b condition}, we may apply Claim \ref{claim: lower bound on extensions distinct} (with $r$ replaced by $r-k$ and $k$ replaced by $2$) to the subgraph of $G_ {n,p}$ induced by the vertex set $W\coloneqq \cP_A\cup \left(\cP_B\setminus V(F)\right)$. This yields that provided that $c_0$ is small enough, with probability $1-n^{-4D}$ every $f\in (A_1\times A_2)$ admits at least $c_0^{1/3}\cdot  b^{r-k-2} p^{\binom{r-k}{2} - 1}$ extensions to a copy of $K_{r-k}$ with a single vertex from each $B_1, \dots, B_{r-k-2}$ in the induced subgraph of $G_{n,p}$ on $W$.
    Whenever $k > 0$, as we assume \ref{item:k condition}, we may apply Remark~\ref{remark: lower bound on extensions} to the subgraph of $G_{n,p}$ induced by $W'\coloneqq W\cup (\cP_Z\setminus V(F))$. This yields that provided that $c_0$ is small enough, with probability at least $1-n^{-4D}$, every $S\subseteq  W$ of size $r-k$, has at least $c_0^{1/3}\cdot  n^{k} p^{\binom{r}{2} - \binom{r-k}{2}}$ extensions to a copy of $K_r$ with all vertices in $Z$. By combining these estimates, and noting that
    \[
        b^{r-k-2} p^{\binom{r-k}{2} - 1} \cdot \left(n^{k} p^{\binom{r}{2} - \binom{r-k}{2}}\right)^2 = \Lambda,
    \]
    with probability at least $1-n^{-3D}$, the following holds for every $f\in A_1\times A_2$: The subgraph of $G_{n,p}$ induced by $W\cup (\cP_Z\setminus V(F))$ contains at least $c_0\Lambda$ pairs $(H_1,H_2)$ such that $H_1,H_2\in \cR_{2,r-k-2}$ are $K_r$-extensions of $f$ and $V(H_1)\cap \cP_B = V(H_2)\cap \cP_B$.
    Therefore, the assumption that $G_0\cup G_1\supseteq G_{n,p}$ yields that
     \begin{align*} 
        \Pr\left(G\text{ satisfies \ref{item: typical event lower bound}} \mid \cD\right) &\ge \Pr\left(G\cap K_n[W']\text{ satisfies \ref{item: typical event lower bound}} \mid \cD\right) \\
        &
        \ge \Pr\left(G_{n, p}\cup (A_1\times A_2)\text{ satisfies \ref{item: typical event lower bound}}
        \right) \ge 1-n^{3D},
    \end{align*}
    where $K_n[W']$ is the complete graph on the vertex set $W'$. 

    We now switch to $\Pr\left(G\text{ satisfies \ref{item: typical event upper bound}} \mid \cD\right)$. First, observe that for every edge $f\in F\cap (A_1\times A_2)$ and every $j\in \{1,2\}$, since the event $\left\{\left\lvert \partial_f T_{2j}(G) \right\rvert \ge \left((\log n)^{5r} \Lambda\right)^j \right\}$ depends only on $(G_0\cup G_1) \setminus (A_1 \times A_2)$, and since $\cD$ is independent of $(G_0\cup G_1)\setminus F$ and satisfies $\cD \subseteq \{F\subseteq G_0\cup G_1\}$, we have
    \begin{equation}\label{eq: D to FsubsetG}
         \Pr\left(\left\lvert \partial_f T_{2j}(G) \right\rvert \ge \left((\log n)^{5r}  \Lambda\right)^j \mid \cD\right)=\Pr\left(\left\lvert \partial_f T_{2j}(G) \right\rvert \ge \left((\log n)^{5r} \Lambda\right)^j \mid F\subseteq G\right).
    \end{equation}
    Recall that, under Properties \ref{item:b condition} and \ref{item:k condition}, we may apply Lemma \ref{cor:upper bound for copies on f}, implying that for every $f \in F \cap (A_1 \times A_2)$ and every $j \in \{1,2\}$, the following holds: 
    \begin{equation}\label{eq: E2 for one edge}
         \begin{aligned}
             \Pr\Bigl(\left\lvert \partial_f T_{2j}(G) \right\rvert \ge \left((\log n)^{5r}  \Lambda\right)^j &\mid F\subseteq G\Bigl) \\ & \le \Pr\left(\left\lvert \partial_f T_{2j}(G_{n, 2p}) \right\rvert \ge \left((\log n)^{5r} \Lambda\right)^j \mid F \subseteq G_{n, 2p}\right) \le n^{-4D}.
         \end{aligned}
    \end{equation}
    By letting $D$ be sufficiently large, combining \eqref{eq: D to FsubsetG}, \eqref{eq: E2 for one edge}, and by the union bound over the edges $f\in A_1\times A_2$ and the two choices of $j$ we have 
    \begin{equation*}
        \Pr\left(G\text{ satisfies \ref{item: typical event upper bound}} \mid \cD\right) 
        \ge \Pr\left(G_{n, 2p}\cup (A_1\times A_2)\text{ satisfies \ref{item: typical event upper bound}}
        \mid F \subseteq G_{n,2p}\right) \ge 1-n^{-3D}.\qedhere
    \end{equation*}
\end{proof}

Fix $c_{0}>0$ to be a small constant for which Claim \ref{claim: E_1 and E_2 are typical} holds. Then, by noting that $|e^{it\Tilde{X}/\sigma_r}|=1$ always, we have 
\begin{equation*}
      \left\lvert \E\left[e^{it\Tilde{X} / \sigma_r} \mid \cD\right] \right\rvert \le n^{-2D} + \left\lvert \E\left[e^{it\Tilde{X} / \sigma_r} \mid \cD \wedge \left\{G \in  \cE(c_{0})\right\}\right] \right\rvert.
\end{equation*}
Therefore, to show \eqref{align:bound of conditional characteristic function}, it suffices to give an upper bound for the conditional expectation on the right-hand side above. More specifically, from this point onwards, we fix $\Gamma \in \cE(c_0)$ and aim to show that
\begin{equation}\label{align:conditional characteristic function bound}
    \left\lvert \E\left[e^{it\Tilde{X} / \sigma_r} \mid \cD \wedge \left\{G = \Gamma\right\}\right] \right\rvert\leq n^{-2D}.
\end{equation}
For brevity, we denote by $E$ the event that $\cD$ holds as well as $G = \Gamma$. To prove \eqref{align:conditional characteristic function bound}, we introduce the following notation. For every edge $f \in A_1 \times A_2$, define the \emph{weight} of $f$ as 
\[
    w_f =w_f(G_0,G_1)\coloneqq  \sum S_{G_0\cup f, G_1}(H),
\]
where the sum ranges over all $r$-cliques $H$ such that $f\in H\in \cR_{2, r-k-2}$.
Crucially, since the graph $ G_1[\mathcal{P}_A]$ is empty, we have
\begin{equation}\label{eq:*}
     \Tilde{X} = \sum_{f \in A_1 \times A_2} w_f \cdot \mathbf{1}_{f \in G_0}.
\end{equation}
To continue our analysis, we will need the following concentration inequality for $|w_f|$. For notational convenience we write $F\coloneqq \bigcup ^{\ell}_{i=1}F_i$.

\begin{claim}\label{claim:good f probability}
    There exists a constant $c_{1}> 0$ such that, for every $f \in A_1 \times A_2$ with no endpoint in $V(F)$, we have 
        \[
            \Pr\left(c_{1} \Lambda^{1/2} \le \left|w_f\right| \le (\log n)^{10r} \Lambda^{1/2} \mid E\right) \ge \frac{1}{(\log n)^{11r}}.
        \]
\end{claim}

\begin{proof}
    Fix $f \in A_1 \times A_2$ with no endpoint in $V(F)$.
    For every copy $H \in \cR_{2,r-k-2}$ containing $f$ the event $\cD$ is independent of the edges of $H$ incident to $f$ in $G_0$ or $G_1$, and further the event~$\cE(c_0)$ is determined by $G_0\cup G_1$. Therefore, by Claim \ref{claim:sufficient condition for sign zero in expectation new} we have $\E\left[S_{G_0\cup f, G_1}(H) \mid E\right] = 0$, which by linearity of expectation implies that $\E\left[w_f \mid E\right] = 0$.    
    Moreover, Claim \ref{claim:sufficient condition for sign zero in expectation new} also implies the following for every integer $s$: Every $H_1, \ldots, H_s \in \cR_{2, r-k-2}$ containing $f$ and satisfying
    \[
        \E\left[S_{G_0\cup f, G_1}(H_1) \cdots S_{G_0\cup f, G_1}(H_s) \mid E\right] \neq 0,
    \]
    must also satisfy $H_1\cup\ldots \cup H_s \subseteq G$ and moreover every vertex $v\in \cP_A\cup \cP_B$  either does not belong to $\bigcup_{i=1}^{s} V(H_i)$ or belongs to $V(H_{s_1})\cap V(H_{s_2})$, for some $s_1\neq s_2$.
    Furthermore, if $s=2$, then such a pair satisfies $S_{G_0\cup f, G_1}(H_1)= S_{G_0\cup f, G_1}(H_2)\in \{0,\pm 1\}$, and $S_{G_0\cup f,G_1}(H_1)\neq 0$ if and only if $H_1$ and $H_2$ agree on $\cP_A$ and $\cP_B$. In this case we also have the following due to Claim \ref{claim: sign zero new}:
    \[
        \Pr\left(S_{G_0\cup f, G_1}(H_1)\cdot S_{G_0\cup f, G_1}(H_2) = 1 \mid E\right) = \left(1-\frac{p^2}{1-(1-p)^2}\right)^{2(r-k)}= \Theta(1).
    \]
    Thus, the weight of $f$ satisfies
    \begin{align*}
		\Omega\left(\E\left[ \lvert \partial_f T_{2}(G) \rvert \mid E\right] \right)&= \E\left[\left(w_f\right)^2 \mid E\right]\\&=\E\left[\sum_{H_1,H_2} S_{G_0\cup f, G_1}(H_1)\cdot S_{G_0\cup f, G_1}(H_2) \mid E\right]
        \le \E\left[ \lvert \partial_f T_{2}(G) \rvert \mid E\right],
	\end{align*}
	where the sum ranges over ordered pairs of graphs $ H_1,H_2\in \cR_{2, r-k-2}$ containing $f$. Similarly, $\E\left[\left(w_f\right)^4 \mid E\right] \leq \E\left[ \lvert \partial_f T_{4}(G) \rvert \mid E \right].$ 

    Recalling that $E$ implies $\cE(c_0)$, we get that
    \[
        c_{0} \Lambda \le \E\left[\left(w_f\right)^2 \mid E\right] \le (\log n )^{5r}  \Lambda \quad \text{and}\quad   \E\left[(w_f)^4 \mid E\right] \le (\log n)^{10r} \Lambda^2.
    \]
Thus, by Markov's inequality we have,
    \begin{align*}
        \Pr\left(\left|w_f \right| \ge (\log n)^{10r} \Lambda^{1/2} \mid E\right) &= \Pr\left(\left(w_f \right)^2 \ge (\log n)^{20r} \Lambda \mid E\right) 
        \le \frac{\E\left[\left(w_f\right)^2 \mid E\right]}{ (\log n)^{20r} \Lambda} \le \frac{1}{(\log n)^{15r}},
    \end{align*}
    and provided that $c_1$ is small enough, by the Paley-Zygmund inequality \ref{claim: Paley-Zygmund} we have 
    \begin{align*}
        \Pr\left(\left|w_f \right| \ge c_{1} \Lambda^{1/2} \mid E\right) \geq \Pr\left(\left(w_f \right)^2 \ge \frac{1}{2}\E\left[\left(w_f \right)^2 \mid E\right] \mid E\right) \ge \frac{\left(\E\left[\left(w_f \right)^2 \mid E\right]\right)^2}{4\E\left[\left(w_f \right)^4 \mid E\right]} 
        \ge  \frac{c_0^2}{4(\log n)^{10r}}.
    \end{align*}
    By combining the above, the proof follows:
    \begin{align*}
        \Pr\left(c_{1} \Lambda^{1/2} \le \left|w_f\right| < (\log n)^{10r} \Lambda^{1/2} \mid E\right) 
            &= \Pr\left(\left|w_f\right| \ge c_{1} \Lambda^{1/2} \mid E\right) - \Pr\left( \left|w_f\right| \ge (\log n)^{10r} \Lambda^{1/2} \mid E\right) \\
            &\ge \frac{c_0^4}{4(\log n)^{10r} } - \frac{1}{(\log n)^{15r}} \ge \frac{1}{(\log n)^{11r}}. \qedhere
    \end{align*}  
\end{proof}

Fix $c_{1}$ to be the constant guaranteed by Claim \ref{claim:good f probability}, and let $\cT$ be the collection of edges $f\in A_1 \times A_2$ with a `typical' weight, that is satisfying: 
\begin{align}\label{align:alpha_f good values}
    c_{1} \Lambda^{1/2} \le \left\lvert w_f \right\rvert < (\log n)^{10r} \Lambda^{1/2}.
\end{align}

In the following claim, we show that typically, a sufficiently large proportion of the edges in $A_1\times A_2$ belong to $\cT$.

\begin{claim}\label{claim:number of good f}
    There exists $\xi > 0$ such that
    \[
        \Pr\left(\left \lvert\cT \right\lvert \ge \frac{1}{(\log n)^{12r}} \cdot a_k^2 \mid E\right) \ge 1 - e^{-n^{\xi}}.
    \] 
\end{claim} 
\begin{proof}
    Fix a partition $M_1, \dots, M_{a_k}$ of $A_1\times A_2$ into edge-disjoint perfect matchings.  Note that if $f_1, f_2 \in A_1\times A_2$ are vertex disjoint and share no vertices with $V(F)$, the random variables $w_{f_1}$ and $w_{f_2}$ are subject to $E$.
Furthermore, note that for every $i \in [a_k]$,  the matching $M_i$ contains at most $v(F)\leq\log n$ edges sharing vertices with $V(F)$. Thus, by recalling that $a_k = n^{1-\epsilon_k}$, we find that for every $i$, the matching $M_i$ contains at least $0.99a_k$ edges disjoint from $V(F)$.

Now by Claim \ref{claim:good f probability}, for every $i\in [a_k]$ the random variable $|\cT\cap M_i|$ is stochastically dominated (from below) by $\textup{Bin}(0.99a_k, q)$ where $q = \frac{1}{(\log n)^{11r}}$. 
In particular, for every $i\in [a_k]$ by the Chernoff bound \ref{lemma:chernoff},
\begin{equation}\label{eq: stoch by bin}
	\Pr\left(\left\lvert \cT \cap M_i\right\lvert < 0.98 a_kq \mid E\right) \le \Pr(\textup{Bin}(0.99a_k, q) <0.98 a_kq) \le e^{-\Theta(a_kq)}.
\end{equation}
The union bound yields that
\begin{equation}\label{eq: stoch union bound}
\Pr\left(\left\lvert \cT \right\lvert < \frac{1}{(\log n)^{12r}}\cdot a_k^{2} \mid E\right)\leq \Pr\left(|\cT| < 0.98 a_k^2q \mid E\right) \le \sum_{i=1}^{a_k}\Pr\left(|\cT\cap M_i| < 0.98 a_k q \mid E\right).
\end{equation}
Lastly, by taking $\xi < 1 - \epsilon_k$ the required inequality is implied by  \eqref{eq: stoch by bin} and \eqref{eq: stoch union bound}:
\[
	\Pr\left(\left\lvert \cT \right\lvert < \frac{1}{(\log n)^{12r}}\cdot a_k^{2} \mid E\right)\le a_k \cdot e^{-\Theta(a_kq)} \le e^{-n^\xi}. \qedhere
\]
\end{proof}

To prove \eqref{align:conditional characteristic function bound}, we will need one more claim. For this, recall our assumption that $t\le  \UB_{k, b}$.

\begin{claim}\label{claim: closest to integer}
	For every $f\in \cT$ we have $\frac{t | w_f | }{\sigma_r} \le 1$.
\end{claim}
\begin{proof}
    First, note that $\UB_{k, b}^2 \le \frac{1}{(\log n)^{20r}} \cdot \frac{\sigma_r^2}{\Lambda}$. Indeed, if $k = 0$, we in fact have equality by definition. For $k \ge 1$, the inequality follows once we establish that
    \[
        {n^{-\epsilon_k(k-0.5)}} \cdot \sqrt{\frac{\sigma_r^2}{a_k^{2k+2} b^{r-k-2} p^{2\binom{r}{2} - \binom{r-k}{2}}}} \ll \sqrt{\frac{\sigma_r^2}{b^{r-k-2} n^{2k} p^{2\binom{r}{2} - \binom{r-k}{2} - 1}}},
    \]
	which is equivalent to $n^{3 \epsilon_k} \ll  n^2 p$. This inequality holds provided that $\epsilon_k$ is sufficiently small and by \eqref{align: p assumption} --- our assumption on $p$.

 Then, for any $f\in \cT$ we obtain the required inequality as follows: 
\[
	t\leq \UB_{k,b} \leq  \frac{\sigma_r}{(\log n)^{10r} \cdot \Lambda^{1/2}} \leq   \frac{\sigma_r}{| w_f |}.\qedhere
\]
\end{proof}

Finally, we are ready to complete the proof of \eqref{align:conditional characteristic function bound}.
Note that by Claim \ref{claim:number of good f} and by \eqref{eq:*}, upon setting the event $E^*\coloneqq E\cap \bigl\{|\cT|\geq \frac{1}{(\log n)^{12r}}\cdot a_k^2\bigr\}$, we have
\begin{align*}
    \left\lvert \E\left[e^{it\Tilde{X}/\sigma_r}  \mid E \right] \right\rvert \le e^{-n^{\xi}} + \left\lvert \E\left[e^{it\Tilde{X}/\sigma_r} \mid E^*\right] \right\rvert\leq  e^{-n^{\xi}}+ \left\lvert \E\left[\prod_{f \in A_1\times A_2}e^{it\left( w_f \cdot \mathbf{1}_{f \in G_0}\right)/\sigma_r} \mid E^*\right] \right\rvert .
\end{align*}
Further, note that conditioned on (the random variable) $\Gamma_{0,1}\coloneqq (G_0\setminus (A_1\times A_2),G_1\setminus (A_1\times A_2))$ the weights $w_f$ are fixed for every $f\in A_1\times A_2$, and thereby the set $\cT$ is also fixed. Moreover, conditioned on $\Gamma_{0,1}$, the random variables $\{\mathbf{1}_{f\in G_0}:f\in A_1\times A_2\}$ are mutually independent, and we have
\begin{align*}
    \left\lvert \E\left[\prod_{f \in A_1\times A_2}e^{it\left( w_f \cdot \mathbf{1}_{f \in G_0}\right)/\sigma_r} \mid E^*\right] \right\rvert &= \left\lvert \E\left[\E\left[\prod_{f \in A_1\times A_2}e^{it\left( w_f \cdot \mathbf{1}_{f \in G_0}\right)/\sigma_r} \mid \Gamma_{0,1}\right] \mid E^*\right] \right\rvert \\
    &\le \E\left[\prod_{f \in \cT} \left\lvert\E\left[e^{it\left( w_f \cdot \mathbf{1}_{f \in G_0}\right)/\sigma_r} \mid \Gamma_{0,1}\right] \right\rvert \mid E^*\right].
\end{align*}
Therefore, to prove \eqref{align:conditional characteristic function bound} it suffices to show that, for every value $\Gamma$ of $\Gamma_{0, 1}$ satisfying $E^*$ the following holds
\begin{align*}
    \prod_{f \in \cT} \left\lvert\E\left[e^{it\left( w_f \cdot \mathbf{1}_{f \in G_0}\right)/\sigma_r} \mid \Gamma_{0,1} = \Gamma\right] \right\rvert  \le n^{-3D}.
\end{align*}

By Claim \ref{claim:char-bounds}, we have
\begin{align*}
    \prod_{f \in \cT} \left\lvert\E\left[e^{it\left( w_f \cdot \mathbf{1}_{f \in G_0}\right)/\sigma_r} \mid \Gamma_{0,1} = \Gamma\right] \right\rvert
    \le \prod_{f \in \cT} \left(1-8p(1-p)\left\|\frac{t\cdot w_f}{2\pi \sigma_r}\right\|^2\right),
\end{align*}
where we recall that $\left\|x\right\|$ is the distance between $x$ and $\mathbb{Z}$. This, Claim \ref{claim: closest to integer} (asserting that for every $f\in \cT$, we have $\left\|\frac{t\cdot w_f}{2\pi \sigma_r}\right\| =\frac{t\cdot |w_f|}{2\pi \sigma_r}$), and the lower bound in \eqref{align:alpha_f good values} imply that 
\begin{align*}
    \prod_{f\in \cT} \exp\left({-8p(1-p)\left(\frac{t\cdot w_f }{2 \pi \sigma_r}\right)^2}\right)  
    \le \exp \left({-0.1p(1-p)\cdot \frac{c^2_{1} \Lambda t^2\cdot a_k^2}{\sigma_r^2\cdot (\log n)^{12r}}}\right).
\end{align*}
By the assumption that $t \ge \LB_{k,b}\ge  \sqrt{\frac{(\log n)^{14r}}{a_k^2 p} \cdot \frac{\sigma_r^2}{\Lambda}}$, and by the assumption that $1-p\geq 1/2$, we arrive at the following inequality
\[
\prod_{f \in \cT} \left\lvert\E\left[e^{it\left( w_f \cdot \mathbf{1}_{f \in G_0}\right)/\sigma_r} \mid \Gamma_{0,1} = \Gamma\right] \right\rvert \le e^{-\Theta\left((\log n)^{2r}\right)} \le n^{-3D}.
\]
This is as required, and therefore, the proof of the first assertion of Lemma \ref{lemma:main lemma for decoupling} is completed.

\section{Proof of the second item of Lemma \ref{lemma:main lemma for decoupling}}\label{section:high moments}
In this section, our goal is to prove the second assertion of Lemma \ref{lemma:main lemma for decoupling}. Fix $p\in(0,1/2]$ satisfying \eqref{align: p assumption}, and assume that $k \in \{0, \dots, r-3\}$ and $0 < b \le a_k$ are integers that satisfy Properties \ref{item:b condition} and \ref{item:k condition}. Moreover, fix an $(a_k, b, k, r)$-partition $\cP$ and let $G_0 \sim G_{n, p}$ and $G_1 \sim \left(\cP_A \times \cP_B\right)_p$ be two independent random graphs. We will show that for any sufficiently large integer $L = L(\epsilon_1, r) > 0$ we have
\begin{align}\label{align:bound of second item in the main lemma}
    \E\left[\Tilde{Y}^{2L}\right] = O_L\left(\left(n^{2\epsilon_k(k-1+0.1)} a_k^{2k+2} b^{r-k-2} p^{2\binom{r}{2} - \binom{r-k}{2}}\right)^{L}\right).
\end{align}
Whenever $k=0$, we have $\tilde{Y} = 0$ deterministically since $\mathcal{R}=\mathcal{R}_{2,r-2}$ by definition; hence, from here onwards we assume that $k > 0$.

Recalling the definition of $\Tilde{Y}$ and setting $\cR'\coloneqq \cR\setminus \cR_{2,r-k-2}$, it is clear that
\begin{equation*}
    \Tilde{Y}^{2L} = \sum_{F_1,\ldots ,F_{2L}\in \cR'} \prod_{i=1}^{2L}S_{G_0,G_1}(F_i).
\end{equation*}
Note that $F_1,\ldots ,F_{2L}$ might not all have the same number of vertices in each of the parts of $\cP$. This is a technically issue which we take care of right away. For integers $i, j$ set
\[
    \Tilde{Y}_{i, j} \coloneqq \sum_{F \in \cR_{i, j}} S_{G_0, G_1}(F),
\]
and observe that
\[
    \Tilde{Y} = \sum_{(i,j)\in I} \Tilde{Y}_{i, j} \quad \text{where} \quad I\coloneqq \left\{(i,j) \in [r]^2:\ \substack{i\ge 2,\ j\ge r - k - 2,\ \\ \text{and}\ r-k < i + j \le r}\right\}.
\]
Since for non-negative reals $x_1, \dots, x_\ell$ we have $(x_1 + \dots + x_\ell)^{2L} \le \ell^{2L} \cdot \max\{x_i^{2L}:i\in [\ell]\}$, we obtain the following bound
\begin{align}\label{align:Y high moment}
    \E\left[\Tilde{Y}^{2L}\right] \le |I|^{2L} \cdot \max_{(i,j)\in I} \E\left[\Tilde{Y}_{i, j}^{2L}\right] \le r^{4L} \max_{(i,j)\in I} \E\left[\Tilde{Y}_{i, j}^{2L}\right].
\end{align}
Therefore, to achieve \eqref{align:bound of second item in the main lemma}, it suffices to show that for any $(i,j)\in I$ we have 
\begin{equation}\label{eq: high moments for one type}
    \E\left[\Tilde{Y}^{2L}_{i,j}\right] = O_L\left(\left(n^{2\epsilon_k(k-1+0.1)} a_k^{2k+2} b^{r-k-2} p^{2\binom{r}{2} - \binom{r-k}{2}}\right)^{L}\right).
\end{equation}

Towards proving the above, fix a pair $(i,j)\in I$. As in the previous section, Definition \ref{def: count of good rainbow copies} will play a crucial role.
When $G$ or $\cP$ is clear from the context, we will omit it from the above notation.
The importance of Definition \ref{def: count of good rainbow copies} can be seen from Claim \ref{claim:sufficient condition for sign zero in expectation new}, which implies that for any $(F_1,\ldots ,F_{2L})\in \cR_{i,j}^{{2L}}\setminus T_{2L}^{i, j}(K_n)$ we have 
\[
    \E\left[S_{G_0, G_1}(F_1) \cdot \ldots \cdot S_{G_0, G_1}(F_{2L})\right] = 0.
\]
Crucially, as the sign function takes values in $\{0,\pm1\}$, by the above and by the fact that for $F\not \subseteq G_0\cup G_1$ the sign function $S_{G_0,G_1}(F)$ vanishes, we have 
\begin{align*} 
    \E\left[\Tilde{Y}_{i, j}^{2L}\right] &=\sum_{F_1, \dots, F_{2L} \in \cR_{i, j}} \E\left[S_{G_0, G_1}(F_1) \cdot \ldots \cdot S_{G_0, G_1}(F_{2L})\right] \le \E\left[\left\lvert T_{2L}^{i, j}(G_0 \cup G_1) \right\rvert \right].
\end{align*}
Further, as in the previous section, we may couple $G_0,G_1,$ and $G_{n,2p}$ such that $G_0\cup G_1\subseteq G_{n,2p}$, since every edge $e\in K_n$ belongs to $G_0\cup G_1$ with probability at most $2p$, independently of the other edges.  This yields that
\[
    \E\left[\Tilde{Y}_{i, j}^{2L}\right] \le \E\left[\left\lvert T_{2L}^{i, j}(G_{n, 2p}) \right\rvert \right] = O_L\left(\E\left[\left\lvert T_{2L}^{i, j}(G_{n, p}) \right\rvert \right]\right),
\]
where the second inequality holds as $e\left(\bigcup_{\ell=1}^{2L} F_\ell\right) = O_L(1)$ for any $(F_1,\ldots ,F_{2L})\in T^{i,j}_{2L}(K_n)$. 
Now it is clear that to prove \eqref{eq: high moments for one type}, it is enough to prove the following, which is the main challenge in this section:
\begin{equation}\label{align:bound of second item in the main lemma for i and j}
     \begin{aligned}\E\left[\left\lvert T_{2L}^{i, j}(G_{n, p}) \right\rvert \right]&= \sum_{(F_1, \dots, F_{2L}) \in T_{2L}^{i, j}(K_n)} \Pr\left(\bigcup_{\ell=1}^{2L} F_\ell \subseteq G_{n, p}\right)\\
     &= O_L\left(\left(n^{2\epsilon_k(k-1+0.1)} a_k^{2k+2} b^{r-k-2} p^{2\binom{r}{2} - \binom{r-k}{2}}\right)^{L}\right).
     \end{aligned}
\end{equation}
The rest of this section is divided into two subsections. In the first one, we analyse the structure of $F= \bigcup_{\ell =1}^{2L}F_\ell$ where $(F_1,\ldots ,F_{2L})$ is an arbitrary element of $T_{2L}^{i,j}(K_n)$.
In particular, we will be concerned with the number of vertices of such $F$ in $\cP_A, \cP_B$, and $\cP_Z$ as well as the number of edges of $F$.
This is then used in the second subsection to complete the proof of \eqref{align:bound of second item in the main lemma for i and j}.

\subsection{The structure of $F$}\label{subsec:structure-of-maximum}
Fix $(i,j)\in I$ for the remainder of this subsection and for every sequence $\mathbf{F}=(F_1,\ldots ,F_{2L})\in T_{2L}^{i,j}(K_n)$ let
\[
    n_A(\mathbf{F}) \coloneqq \left|\bigcup_{\ell=1}^{2L} V(F_\ell)  \cap \cP_A\right|, \quad  n_B(\mathbf{F}) \coloneqq \left|\bigcup_{\ell=1}^{2L} V(F_\ell) \cap \cP_B\right|,\quad \text{and}\quad n_Z(\mathbf{F}) \coloneqq \left|\bigcup_{\ell=1}^{2L} V(F_\ell) \cap \cP_Z\right|.
\]
We will frequently omit the dependency on $\mathbf{F}$ as it will be clear from the context. Let us start with some na\"ive bounds on $n_A,n_B$, and $n_Z$.

\begin{observation}\label{claim:number of vertices in F}
    For every $(F_1,\ldots ,F_{2L})\in T_{2L}^{i,j}(K_n)$ the following three inequalities hold:
    \begin{enumerate}
        \item\label{item: na} $i \le n_A  \le (2i - 2)L$.
        \item\label{item: nb} $j \le n_B \le (2j - (r-k-2))L$.
        \item\label{item: nz} $r - i - j \le n_Z  \le 2(r-i-j)L$.
    \end{enumerate}
\end{observation}
\begin{proof}
    We first note that the lower bound in each of the cases trivially holds as $F_1\in \cR_{i,j}$. Moreover, the upper bound in item \eqref{item: nz} is also trivial. Thus, we focus on the upper bounds in \eqref{item: na} and \eqref{item: nb}. We will prove only the upper bound in \eqref{item: na}, as the proof of the upper bound in \eqref{item: nb} is verbatim.
    
   For every $\ell \in [2L]$, set
    \[
        k_\ell \coloneq \left\lvert \left(\cP_A \cap V(F_\ell)\right) \cap \left(\bigcup_{m \neq \ell} V\left(F_{m}\right)\right) \right \rvert\ge 2,
    \]
    where the lower bound follows from the fact that in each of $A_1$ and $A_2$ the graph $F_{\ell}$ shares a vertex with another copy.
    Then, for every $\ell \in [2L]$, we have
    \[
        \left|\bigl(\cP_A \cap V(F_\ell)\bigr)\setminus \bigcup_{m\neq \ell} V(F_{m})\right| = i-k_\ell \quad\text{and}\quad \left|\bigl(\cP_{A}\cap V(F_\ell)\bigr)\cap \bigcup_{m\neq \ell}V(F_{m})\right|=k_{\ell}.
    \]
    This implies the required inequality as follows
    \begin{equation*}
        n_A=\left|\cP_A \cap \bigcup_{\ell=1}^{2L} V(F_\ell) \right| \le \sum_{\ell=1}^{2L} (i - k_\ell) + \sum_{\ell=1}^{2L} \frac{k_\ell}{2} = \sum_{\ell=1}^{2L} \left(i - \frac{k_\ell}{2}\right) \le (2i-2) L.\qedhere
    \end{equation*}
\end{proof}

For integers $\alpha,\beta$ and $\gamma$ let $\varphi(\alpha,\beta,\gamma)\coloneqq \varphi_{i,j}(\alpha,\beta,\gamma)$ denote the following quantity:
\[
    \min\left\{e(F):F=\bigcup_{\ell=1}^{2L} F_\ell\text{ where }(F_1, \dots, F_{2L}) \in T_{2L}^{i,j}(K_n) \text{ and } n_A = \alpha, n_B =\beta,\text{ and }n_Z=\gamma\right\},
\]
where we use the convention $\min \emptyset \coloneqq \infty$. When $\gamma=0$, we will use the shorthand notation $\varphi(\alpha,\beta)=\varphi(\alpha,\beta,0)$. In the rest of this subsection, we only consider the case where $\gamma=0$. Note that by Observation \ref{claim:number of vertices in F} if $i+j<r$ and $\gamma=0$, then $\varphi \equiv \infty$, and therefore, we further restrict the discussion to the cases where $i+j=r$.
The following lemma, which is the main technical lemma of this subsection, provides a fairly tight lower bound for $\varphi(\alpha,\beta)$.

\begin{lemma}\label{claim:number of edges in F}
    Suppose that $i \ge 2$ and $j \ge r - k - 2$ are integers satisfying $i + j = r$. Then, for every integers $n_A$ and $n_B$ the following inequalities hold:
    \begin{align}\label{eq:main-lemma-1}
        \varphi_{i,j}(n_A, n_B) \ge \frac{(r-1)(n_A + n_B) + (r-k)\left(\max\{0, n_A - i L\} + \max\{0, n_B - j L\}\right)}{2}
    \end{align}
    and 
    \begin{equation}\label{eq:main-lemma-2}
    \begin{aligned}
        \varphi_{i,j}(n_A, n_B) &\ge \max\{j\cdot n_A , i \cdot n_B \} + \frac{(i - 1)n_A + 2\max\{0, n_A- i L\}}{2} \\&+ \frac{(j - 1)n_B + (r-k-2)\max\{0, n_B- j L\}}{2}.
    \end{aligned}
    \end{equation}
\end{lemma}

To prove Lemma \ref{claim:number of edges in F}, we will require the following technical claim.
\begin{claim}\label{claim:number of edges in F with D}
    Suppose that $\ell_1,\ell_2,s,s_1,s_2$ are non-negative integers such that $s=s_1+s_2\geq 3$. Let $F_1, \dots, F_{2L}$ be cliques, each of size $s$, set $F\coloneqq\bigcup_{s=i}^{2L}F_i$ and fix a partition $\{W_1,W_2\}$ of $V(F)$. 
    Moreover, assume that for every $i \in [2]$ and $j \in [2L]$ we have 
    \[
        \left|W_i\cap V(F_j)\right| = s_i\quad \text{and}\quad\left|W_i\cap V(F_j)  \cap \left(\bigcup_{\ell \neq j} V(F_\ell)\right)\right| \ge \ell_i.
    \]
    Then,
    \[
        e\left(F\right) \ge \frac{(s-1)\left|V(F)\right| + \left(\ell_1 + \ell_2\right) \left(\sum_{i=1}^{2} \max\left\{0, \left|W_i\right| - s_i \cdot L\right\}\right)}{2}.
    \]
\end{claim}

\begin{proof}
    Denote by $\cD$ the set of vertices in $V(F)$ that belong to a single clique among $F_1, \dots, F_{2L}$. Clearly,
        \begin{align}\label{align:F sum of degrees}
            2 e(F) = \sum_{v \in \cD} d_F(v) + \sum_{v \in V(F) \setminus \cD} d_F(v) = (s-1) |\cD| + \sum_{v \in V(F) \setminus \cD} d_F(v).
        \end{align}
        For every vertex $v \not \in \cD$, let $w_v$ denote the size of the $F$-neighbourhood of $v$ in $\cD$, that is $w_v \coloneqq \left|N_F(v) \cap \cD\right|$. We claim that every vertex $v \not\in \cD$ satisfies
        \begin{align}\label{align:degree of v not in D}
            d_F(v) \ge s-1 + w_v/2.
        \end{align}
        Indeed, since $v \notin \cD$, it belongs to at least two cliques among $F_1, \dots, F_{2L}$, and it also satisfies
        \[
            \sum_{i=1}^{2L} \left|N_F(v) \cap V(F_i) \cap \cD\right| = w_v.
        \]
        Hence, there is some $i$ for which $v \in V(F_i)$ and $\left|N_F(v) \cap V(F_i) \cap \cD\right| \le w_v/2$.
        This implies \eqref{align:degree of v not in D} as then $v$ has $s-1$ neighbours belonging to $F_i$ and at least $w_v/2$ more neighbours in $\cD\setminus V(F_i)$.

        Substituting \eqref{align:degree of v not in D} in \eqref{align:F sum of degrees} we have
        \begin{align*}
            2 e(F) &\ge (s-1) |\cD| + \sum_{v \in V(F) \setminus \cD} \left(s-1 + w_v/2\right) 
            =  (s-1)\left|V(F)\right| + \frac{1}{2}\sum_{v \in V(F) \setminus \cD} w_v.
        \end{align*}
        By the assumption of the claim, in every clique $F_i$, there are at least $\ell_1 + \ell_2$ vertices that are not in $\cD$. Therefore, each vertex in $\cD$ has at least $\ell_1 + \ell_2$ neighbours in $V(F) \setminus \cD$, implying that
         \begin{align*}
            2 e(F) &\ge (s-1)\left|V(F)\right| + \frac{1}{2}(\ell_1+\ell_2)|\cD|.
        \end{align*}
        It is then enough to show that for every $i\in [2]$ we have $\left|\cD \cap W_i\right|\ge 2\max\{0, |W_i| - s_i \cdot L\}.$

       To this end, fix $i \in [2]$. If $\max\{0, |W_i| - s_i \cdot L\} = 0$, the inequality is trivial; therefore, we may assume the opposite, that is, $|W_i| > s_i \cdot L$. For every vertex $v \in V(F)$, denote by $c_v$ the number of cliques containing $v$ among $F_1, \dots, F_{2L}$. Since every such clique contains exactly $s_i$ vertices in $W_i$, we have 
        \[
            2s_i \cdot L = \sum_{v \in W_i} c_v = |\cD \cap W_i| + \sum_{v \in W_i \setminus \cD} c_v \ge |\cD \cap W_i| + 2(|W_i| - |\cD \cap W_i|),
        \]
        where the last inequality holds as $c_v \ge 2$ for every $v \in W_i \setminus \cD$. Hence the desired inequality is obtained
        \[
            |\cD \cap W_i| \ge 2(|W_i| - s_i \cdot L).\qedhere
        \]
\end{proof}

\begin{proof}[Proof of Lemma \ref{claim:number of edges in F}]
    Fix $i,j,n_A,n_B$ to be some non-negative integers such that $i+j=r$ and such that $i \ge 2$ and $j \ge r-k-2$. If $\varphi_{i,j}(n_A,n_B)=\infty$, the claim is vacuous. Hence, by Observation \ref{claim:number of vertices in F} we may further assume that $i \le n_A \le (2i - 2)L$ and $j \le n_B \le (2j - (r-k-2))L$.
    
    Let us also fix some graph $F=\bigcup_{\ell=1}^{2L} F_\ell$ such that $(F_1, \dots, F_{2L}) \in T_{2L}^{i, j}(K_n)$ satisfies $\left|V(F) \cap \cP_A\right| = n_A$ as well as $\left|V(F) \cap \cP_B\right| = n_B$. By definition, for every $\ell \in [2L]$ we have 
    \[
     \left|\cP_A\cap V(F_\ell)\right| = i\quad \text{and}\quad\left|\cP_A\cap V(F_\ell)  \cap \left(\bigcup_{m \neq \ell} V(F_m)\right)\right| \ge 2,
    \]
    and 
    \[
     \left|\cP_B\cap V(F_\ell)\right| = j\quad \text{and}\quad\left|\cP_B\cap V(F_\ell)  \cap \left(\bigcup_{m \neq \ell} V(F_m)\right)\right| \ge r-k-2.
    \]
    Thus, \eqref{eq:main-lemma-1} follows by applying Claim \ref{claim:number of edges in F with D} with $s=r$, the partition $\{W_1 = \cP_A, W_2 = \cP_B\}$, and the parameters $\ell_1 = 2, \ell_2=r-k-2, s_1=i, s_2 = j,T=L$.

    Continuing to the proof of \eqref{eq:main-lemma-2}, for every $\ell \in [2L]$, let $F_\ell(A)$ and $F_\ell(B)$ denote the subgraphs of $F_{\ell}$ induced by $V(F_\ell)\cap \cP_A$ and $V(F_{\ell})\cap \cP_B$, respectively. Note that $F$ can be partitioned into three graphs: $F(A)\coloneqq \bigcup_{\ell=1}^{2L} F_\ell(A)$, $F(B)\coloneqq \bigcup_{\ell=1}^{2L} F_\ell(B)$, and $F(A,B)$ which consists of all edges having a single endpoint in each $\cP_A$ and $\cP_B$. By applying Claim \ref{claim:number of edges in F with D} with $s=i$, the cliques $F_1(A), \dots, F_{2L}(A)$, the partition $\{W_1 = \cP_A, W_2 = \emptyset\}$, and the parameters $s_1 = i, s_2 = 0, \ell_1 = 2, \ell_2 = 0, T=L$, we have
    \[
        e\left(F(A)\right) \ge \frac{(i - 1) n_A + 2 \max\{0, n_A - i \cdot L\}}{2}.
    \]
    Similarly, we have
    \[
        e\left(F(B)\right) \ge \frac{(j - 1)n_B  + (r-k-2) \max\{0, n_B - j \cdot L\}}{2}.
    \]

    Note that every $v\in V(F)\cap \cP_A$ belongs to some $V(F_{\ell})$, and that $|V(F_{\ell})\cap \cP_{A}|=i$, implying that $d_{F(A,B)}(v)\geq r-i=j$. Combining this with the bipartitness of $F(A,B)$ and the assumptions of the lemma, we have 
    \[
        e(F(A,B)) = \sum_{v\in V(F)\cap \cP_A} d_{F(A,B)}(v) \geq j\cdot n_A. 
    \] 
    By symmetry we also have $e(F(A,B))\geq i\cdot n_B$, implying that $e(F(A,B))\geq \max\{j\cdot n_A, i \cdot n_B \}$. By combining the bounds for $e(F(A)),e(F(B))$, and $e(F(A,B))$, the proof is complete.
\end{proof}

\subsection{Completing the argument}
Our goal in this subsection is to show \eqref{align:bound of second item in the main lemma for i and j} for every $(i,j)\in I$. Our first step is to show that without loss of generality, it suffices to prove \eqref{align:bound of second item in the main lemma for i and j} only for $(i,j)\in I$ such that $i+j=r$. Fix $(i,j)\in I$ and recall that $|\cP_Z|/|\cP_A|=\Theta(n^{\epsilon_k})$ since $a_k = n^{1-\eps_k}$ and $b \le a_k$.
One can surjectively map each sequence $(F_1,\ldots, F_{2L})\in T_{\cP, 2L}^{i, j}(K_{n})$ to a sequence $(F_1',\ldots, F_{2L}')\in T_{\cP, 2L}^{r-j, j}(K_{n})$ by replacing every vertex in $\cP_Z$ with a vertex in $\cP_A$ such that the following holds:
\begin{itemize}
    \item $e\left(\bigcup_{\ell=1}^{2L}F_{\ell}\right)=e\left(\bigcup_{\ell=1}^{2L}F'_{\ell}\right)$, and 
    \item The preimage of each sequence $(F_1',\ldots, F_{2L}')\in T_{\cP, 2L}^{r-j, j}(K_{n})$ is of order $O(n^{2\epsilon_k(r-i-j)L})$.
\end{itemize}
This implies that 
\begin{align*}
    \E\left[\left\lvert T_{\cP, 2L}^{i, j}(G_{n, p}) \right\rvert \right] &=O\left( n^{2\epsilon_k(r-i-j)L} \E\left[\left\lvert T_{\cP, 2L}^{r-j, j}(G_{n, p}) \right\rvert\right]\right) = O\left( n^{2\epsilon_k(k-1)L} \E\left[\left\lvert T_{\cP, 2L}^{r-j, j}(G_{n, p}) \right\rvert\right]\right),
\end{align*}
where the second inequality follows as $i + j \ge r-k+1$. We conclude that to prove \eqref{align:bound of second item in the main lemma for i and j} it is sufficient to restrict to the case where $i + j = r$, and prove that
\begin{align}\label{align:upper bound of N(i, j)}
    \E\left[\left\lvert T_{\cP, 2L}^{i, j}(G_{n, p}) \right\rvert\right] \le n^{C_r} \left(a_k^{2k+2} b^{r-k-2} (2p)^{2\binom{r}{2} - \binom{r-k}{2}}\right)^{L},
\end{align}    
where $C_r>0$ is some constant that might depend on $r$.
Indeed, this implies \eqref{align:bound of second item in the main lemma for i and j} as long as $L = L(\epsilon_1, r)$ is large enough so that $C_r <  0.2 L \epsilon_1  \le 0.2 L \epsilon_k $.

From now on, assume that $(i,j)\in I$ and that $i+j=r$. By Observation \ref{claim:number of vertices in F} we have
\begin{align*}
    \E\left[\left\lvert T_{\cP, 2L}^{i, j}(G_{n, p}) \right\rvert\right] \le \sum_{n_A= i}^{(2i-2)L} \sum_{n_B=j}^{(2j-(r-k-2))L}  a_k^{n_A}  b^{n_B} p^{\varphi_{i, j}(n_A, n_B)}.
\end{align*}
This, combined with Lemma \ref{claim:number of edges in F} (applying \eqref{eq:main-lemma-1} in the first sum, and \eqref{eq:main-lemma-2} in the rest), yields
\begin{align*}
        \E\left[\left\lvert T_{\cP, 2L}^{i, j}(G_{n,p}) \right\rvert\right] &\le \sum_{n_A \ge i L, n_B \ge j L} a_k^{n_A} b^{n_B} p^{\frac{(n_A + n_B)(r-1) + (r-k)\left((n_A - i L) + (n_B - j L)\right)}{2}} \\
        &+ \sum_{n_A \ge i L, n_B \le j L} a_k^{n_A} b^{n_B} p^{n_A \cdot j + \frac{n_A(i - 1) + 2(n_A- i L)}{2} + \frac{n_B(j - 1)}{2}} \\
        &+ \sum_{ n_A \le i L, n_B \ge j L} a_k^{n_A} b^{n_B} p^{n_B \cdot i + \frac{n_A(i - 1)}{2} + \frac{n_B(j - 1) + (r-k-2)(n_B- j L)}{2}} \\
        &+ \sum_{n_A \le i L, n_B \le j L} a_k^{n_A} b^{n_B} p^{\max\{n_A\cdot j,n_B \cdot i\} + \frac{n_A(i - 1)}{2} + \frac{n_B(j - 1)}{2}},
\end{align*}
where in each sum we also require that $i \le n_A \le (2i - 2)L$ and $j \le n_B \le (2j - (r-k-2))L$. 
The following technical claim bound the respective sums above, thereby finishing the proof of \eqref{align:upper bound of N(i, j)} and of the second item of Lemma \ref{lemma:main lemma for decoupling}.

\begin{restatable}{claim}{ClaimAppendix}\label{claim:bound on expectations of T_2L}There exists $C_r>0$ depending on $r$ only such that the following holds for every $i \le n_A \le (2i - 2)L$ and $j \le n_B \le (2j - (r-k-2))L$.
\begin{enumerate}[label=(\roman*)]
    \item\label{n_A and n_B large} If $i L \le n_A \le (2i - 2)L$ and $j L \le n_B \le (2j - (r-k-2))L$. Then,
    \[
        a_k^{n_A} b^{n_B} p^{\frac{(n_A + n_B)(r-1) + (r-k)\left((n_A - i L) + (n_B - j L)\right)}{2}} \le n^{C_r} \left(a_k^{2k+2} b^{r-k-2} p^{2\binom{r}{2} - \binom{r-k}{2}}\right)^{L}.
    \]
    \item\label{n_A large n_B small} If $i L \le n_A \le (2i-2)L$ and $j \le n_B \le j L$. Then,
    \[
        a_k^{n_A} b^{n_B} p^{n_A \cdot j + \frac{n_A(i - 1) + 2(n_A- i L)}{2} + \frac{n_B(j - 1)}{2}} \le n^{C_r}  \left(a_k^{2k+2} b^{r-k-2} p^{2\binom{r}{2} - \binom{r-k}{2}}\right)^{L}.
    \]
    \item\label{n_A small and n_B large} If $i \le n_A \le i L$ and $j L \le n_B \le (2j - (r-k-2))L$. Then,
    \[
        a_k^{n_A} b^{n_B} p^{n_B \cdot i + \frac{n_A(i - 1)}{2} + \frac{n_B(j - 1) + (r-k-2)(n_B- j L)}{2}} \le n^{C_r}\left(a_k^{2k+2} b^{r-k-2} p^{2\binom{r}{2} - \binom{r-k}{2}}\right)^{L}.
    \]
    \item\label{n_A and n_B small} If $i \le n_A \le i L$ and $j \le n_B \le j L$. Then,
    \[
        a_k^{n_A} b^{n_B} p^{\max\{n_A\cdot j,n_B \cdot i\} + \frac{n_A(i - 1)}{2} + \frac{n_B(j - 1)}{2}} \le n^{C_r}\left(a_k^{2k+2} b^{r-k-2} p^{2\binom{r}{2} - \binom{r-k}{2}}\right)^{L}.
    \]
\end{enumerate}
\end{restatable}
For brevity, the straightforward computational proof is deferred to Appendix \ref{app:Appendix}.

\section{Proof of Lemma \ref{lemma:characteristic function integrable} for high frequencies --- $n^{-\gamma} \sigma_r \le t \le \pi \sigma_r$}\label{sec:high-freq}
In this section, we prove Lemma \ref{lemma:characteristic function integrable} for high frequencies. That is, we show that, for a sufficiently small $\gamma>0$, any $K > 0$, any $p\in [0,1/2)$ satisfying \eqref{align: p assumption}, and every $n^{-\gamma} \sigma_r\le t \le \pi \sigma_r$, for large enough $n$, we have
\[
    \left \lvert \E\left[e^{itX_r/\sigma_r}\right] \right\rvert \le n^{-K},
\]
which concludes the proof of Lemma \ref{lemma:characteristic function integrable} by Remark \ref{remark}.
The structure of the proof in this section is similar to that in Section~\ref{section:main lemma for medium frequencies}.
We will define an $(a, b, 0, r)$-partition $\cP = (A_1, A_2, B_1, \dots, B_{r-2}, Z)$ and apply Corollary \ref{cor: the decoupling lemma new} with this partition and the collection of edge sets $\cE = \{A_i \times B_j \colon i \in [2],\ j \in [r-2]\}$. As in Section \ref{section:main lemma for medium frequencies}, we aim to replace $X_r$ with a weighted sum of Bernoulli random variables. Similarly to Section \ref{section:main lemma for medium frequencies}, the Bernoulli random variables will correspond to the indicators of $f \in A_1 \times A_2$ belonging to $G_0$, and the weights will correspond to a signed sum of $\cP$-rainbow copies lying on $f$.

The main difficulty in this frequency range is that we cannot allow large edge weights. More specifically, for an $f\in A_{1}\times A_{2}$ to use Claim \ref{claim:char-bounds} with $\omega_f \cdot \mathbf{1}_{f\in G_0}$, we need to require that each weight satisfies $\lvert \omega_f \cdot t / \sigma_r \rvert \le \frac{1}{2}$. To ensure this, we will choose $b$ to be such that the expected number of $\cP$-rainbow copies in $G_0 \cup G_1$ lying on an edge of $A_1 \times A_2$ is of order $O(1)$. This will allow us to show that with high enough probability, sufficiently many edges in $A_1 \times A_2$ have a $\pm 1$ weight.

Let $\lambda > 1$ be a constant and let $b > 0$ be the minimal integer such that 
\[
    b^{r-2} p^{\binom{r}{2} - 1} \ge  \lambda.
\]
As $p<1 \le \lambda$ we have $b>1$, and therefore, the minimality of $b$ implies that
\begin{align}\label{align:bound on 2 density with b}
     b^{r-2} p^{\binom{r}{2} - 1} = \left(\frac{b}{b-1}\right)^{r-2}\cdot (b-1)^{r-2} p^{\binom{r}{2} - 1} \leq 2^{r-2}\lambda.
\end{align}
We will also use the fact that $b \ll n$, which follows from \eqref{align: p assumption}.

The definition of our partition $\cP$ will exploit a typical property of $G_{n,p}$.
Due to technical reasons, this property differs slightly depending on whether or not $p\geq n^{-O(1)}$.
To define this property, partition $[n]$ into sets $A_1, A_2$, and $B$, each of size $n/3$, and fix a small constant $\xi=\xi(r)>0$ to be determined later. Define the event $\cT=\cT_p$ as follows.
\begin{enumerate}
    \item If $p \ge n^{-\xi}$, then $\cT_p$ is the event that there exists a copy of $K_{r-2}$ in $G_{n, p}[B]$.
    \item If $p < n^{-\xi}$, then $\cT_p$ is the event that for prescribed pairwise disjoint sets $B_1, \dots, B_{r-2} \subseteq B$ with $|B_i|=b$ for each $i$ the following holds. 
    Let $\cH$ be an auxiliary $(r-2)$-uniform hypergraph with vertex set $\bigcup_{i=1}^{r-2}B_i$, and $(r-2)$-edges corresponding to copies of $K_{r-2}$ in $G_{n, p}$ with a single vertex in each of $B_1, \dots, B_{r-2}$. Formally, $\{v_1, \dots, v_{r-2}\}$ is an edge if and only if $v_i \in B_i$ for every $i \in [r-2]$ and if $G_{n, p}[\{v_1, \dots, v_{r-2}\}] \cong K_{r-2}$. Then, we define $\cT_p$ to be the event that $\cH$ satisfies the following conditions:
    \begin{itemize}
        \item $\frac{1}{2r^r} b^{r-2} p^{\binom{r-2}{2}} \le e(\cH) \le r^2 b^{r-2} p^{\binom{r-2}{2}}$.
        \item For every $j \in [r-3]$ denoting by $\Delta_j(\cH)$ the $j$-th maximum degree\footnote{As usual, the $j$-th maximum degree of $\cH$ is defined as the maximum, taken over all subsets $J\subseteq V(\cH)$ with $|J|=j$, of the number of edges containing $J$.} of $\cH$, we have, $\Delta_j(\cH) \le \max\left\{r^2 b^{r-2-j} p^{\binom{r-2}{2} - \binom{j}{2}}, (\log n)^{r+1}\right\}$.
    \end{itemize}
\end{enumerate}

The typicality of $\cT$ in $G_{n,p}$ is formulated in the following claim. 

\begin{claim}\label{claim:T-typical}
    $\Pr(\cT) \ge 1 - n^{-2K}.$
\end{claim}
\begin{proof}
    First, suppose that $p \ge n^{-\xi}$. Applying Theorem \ref{claim: lower bound on extensions} with $r-2$ instead of $r$, $k=0$, and $C=2K$, with probability at least $1-n^{-2K}$ the graph $G_{n, p}[B]$ contains a copy of $K_{r-2}$, as required.

    Next, let $K'>0$ be some large constant, and suppose that $p < n^{-\xi}$. By Claim \ref{claim:upper bound on extensions}, with probability at least $1-b^{-K'}$, for every $S \subseteq B$, the number of extensions of $S$ to a copy of $K_{r-2}$ in $G_{n, p}[B]$ intersecting each $B_i$ once, is bounded from above by
    \[
        \max\left\{(r-2)^2 b^{r-2-|S|} p^{{\binom{r-2}{2}} - \binom{|S|}{2}}, \left(\log n\right)^{r}\right\}.
    \]
    Noting that $b^{r-2} p^{\binom{r-2}{2}} = n^{\Omega(1)}$, we may take $S=\emptyset$ above, and obtain that the required upper bound on $e(\cH)$ in $\cT$ holds with probability at least $1-b^{-K'}$. Similarly, the upper bound on $\Delta_j(\cH)$ in $\cT$ holds with probability at least $1-b^{-K'}$. 
    
    Lastly, by Theorem \ref{claim: lower bound on extensions}, with probability at least $1-b^{-K'}$ we have $e(\cH) \ge \frac{1}{2r^{r}} b^{r-2} p^{\binom{r-2}{2}}$. Since $b = n^{\Omega(1)}$, the claim follows by the union bound upon taking $K'$ sufficiently large.        
\end{proof}

By Claim \ref{claim:T-typical}, we conclude that for every real $t$,
\begin{equation}\label{align:move to typical high frequencies}
        \begin{aligned}
            \left|\E\left[e^{it X_r / \sigma_r}\right]\right| &= \left|\Pr(\cT^{c}) \cdot \E\left[e^{it X_r / \sigma_r}\right] \mid \cT^{c}  + \Pr(\cT)\cdot  \E\left[e^{it X_r / \sigma_r} \mid \cT\right]\right| \\
            &\le n^{-2K} + \left|\E\left[e^{it X_r / \sigma_r} \mid \cT\right]\right|.
        \end{aligned}
    \end{equation}
Therefore, it is enough to show that for every fixed $\Gamma \subseteq K_n[B]$ that satisfies the event $\cT$, we have
\[
    \left|\E\left[e^{it X_r / \sigma_r} \mid G_{n, p}[B] = \Gamma\right]\right| \leq n^{-2K}.
\]

To this end, throughout this section, fix pairwise disjoint sets $B_1, \dots, B_{r-2} \subseteq B$ and fix a graph $\Gamma$ on vertex set $B$, such that the following holds. If $p \ge n^{-\xi}$, then $|B_i| = 1$ for every $i \in [r-2]$, $\Gamma\left[\bigcup_{i=1}^{r-2} B_i\right]$ is an $(r-2)$-clique, and if $p < n^{-\xi}$, then $|B_i|=b$ for every $i\in [r-2]$, and the $(r-2)$-uniform hypergraph $\cH$ representing the copies of $(r-2)$-cliques with vertices in each of $B_1,\ldots, B_{r-2}$ is as described by the event $\cT$. Finally, define $\cP$ to be the partition $(A_1, A_2, B_1, \dots, B_{r-2}, Z)$ of $[n]$, where $Z = [n] \setminus \left(\bigcup_{i=1}^{2} A_i \cup \bigcup_{j=1}^{r-2} B_j\right)$.

Recall that $\cR = \cR(\cP)$ is the set of copies of $K_r$ in $K_n$ taking at least (in this case exactly) one vertex in each of the sets $A_1, A_2, B_1, \dots, B_{r-2}$. Let $G_0 \sim G_{n, p}$ and $G_1 \sim (\cP_A \times \cP_B)_p$ be two independent random graphs.
By Remark \ref{remark: the decoupling lemma new} for every $t > 0$, 
\begin{align}\label{align:after decoupling high frequencies}
    \left|\E\left[e^{it X_r / \sigma_r}  \mid G_{n,p}[B] = \Gamma\right]\right|^{2^{2(r-2)}} \le \left|\E\left[e^{it \Tilde{X} / \sigma_r} \mid G_0[B] = \Gamma\right]\right|,
\end{align}
where $\Tilde{X} = \sum_{F \in \cR} S_{G_0, G_1}(F)$. 
For every edge $f \in A_1 \times A_2$, let $\cR_f$ denote the collection of graphs $F\in \cR$ with $f\in F$, and define the weight of $f$ as 
\[
    w_f \coloneqq \sum_{\substack{F \in \cR_f}} S_{G_0 \cup f, G_1}(F).
\]
We note that as in \eqref{eq:*}, we have 
\begin{align}\label{align:def of X tilde high freq}
    \Tilde{X} = \sum_{f \in A_1 \times A_2} w_f \cdot \mathbf{1}_{f \in G_0}.
\end{align}

In the following two claims, we prove that typically there are significantly many edges $f \in A_1 \times A_2$ with $|w_f| = 1$. We then follow the same approach as in the previous sections to estimate the characteristic function.

\begin{claim}\label{claim:alpha_f high frequency}
    There exists some constant $\beta = \beta(\lambda) > 0$ such that, for every $f \in A_1 \times A_2$, we have
    \[
        \Pr\bigl(\left|w_f\right| = 1 \mid G_0[B] = \Gamma\bigr) \ge \begin{cases}
            p^{2(r-2)}, \quad &p \ge n^{-\xi}, \\
            \beta, \quad &p < n^{-\xi}.
        \end{cases}
    \]    
\end{claim}
\begin{proof}
    Fix $f \in A_1 \times A_2$. To prove the claim, we bound the (conditional) probability that there exists a unique $F \in \cR$ lying on $f$ such that $F \setminus f \subseteq G_0 \triangle G_1$. This is sufficient as in such a case, we have $\left|w_f\right| = 1$. Indeed, by Claim~\ref{claim: sign zero new}, in this case we would have $S_{G_0, G_1}(F) \in \{\pm 1\}$ and $S_{G_0, G_1}(F') = 0$ for every $F \neq F' \in \cR$. 
    For every $F \in \cR$, denote by $S_F$ the event that $F$ extends $f$ to a $K_r$ in $G_0 \triangle G_1$. 
    With this notation, our goal is to bound from below the probability of $\cL$, the event that $S_F$ holds for exactly one $F \in \cR$.

    If $p \ge n^{-\xi}$, then there is a single vertex in each of the sets $B_1, \dots, B_{r-2}$ and $\Gamma\left[\cP_B\right]$ is an $(r-2)$-clique. The probability that an edge between an endpoint of $f$ and a vertex in $\cP_B$ belongs to $G_0 \triangle G_1$ is $2p(1 - p)\geq p$, and there are $2(r-2)$ such edges. Thus, in this case,
    \[
        \Pr\bigl(\left|w_f\right| = 1 \mid G_0[B]=\Gamma\bigr) \ge \Pr\bigl(\cL  \mid G_0[B]=\Gamma\bigr)= \left(2p(1-p)\right)^{2(r-2)} \ge p^{2(r-2)},
    \]    
    
    From now assume that $p < n^{-\xi}$. Writing $\Pr_\Gamma$ for the probability measure $\Pr\bigl(\ \cdot \mid G_0[B]=\Gamma\bigr)$, we have
    \begin{align*}
        \Pr_\Gamma\bigl(\cL \bigr) = \sum_{F \in \cR_f} \Pr_{\Gamma}\left(S_F\right) \cdot \Pr_{\Gamma}\left(\cap_{{F' \in \cR_f\setminus \{F\}}}  S_{F'}^c \mid S_F\right).
    \end{align*}
    Moreover, observe that
    \[
        \sum_{F \in \cR_f} \Pr_{\Gamma}\left(S_F\right) \ge e(\cH) \cdot p^{2(r-2)} = \Theta\left(b^{r-2} p^{\binom{r-2}{2} + 2(r-2)}\right) = \Theta\left(b^{r-2} p^{\binom{r}{2} -1}\right) = \Omega_\lambda(1),
    \]
    where $\cH$ is the hypergraph in the definition of $\cT$ with respect to $G_0$, the second equality holds by the assumption on $e(\cH)$, and the last equality holds by our choice of $b$. Thus, to conclude the proof, it suffices to show that
    \[
        \min _{F\in \cR_f}\Pr_{\Gamma}\left(\cap_{F' \in \cR_f\setminus\{F\}}  S_{F'}^c \mid S_F\right) = \Omega_{\lambda}(1).
    \]

    For every $F \in \cR$, the distribution of $G_0\triangle G_1$ conditioned on the event $\{G_0[B]=\Gamma\} \wedge S_F$ is a product measure, and in this probability space, $S_{F'}$ is a monotone event for every $F'\in \cR$. Thus, setting $\Delta_0(\cH) \coloneqq e(\cH)$, we may apply Theorem \ref{theorem: Harris} and get 
    \begin{align*}
        \Pr_{\Gamma}\left(\cap_{F' \in \cR_f\setminus \{F\}}  S_{F'}^c \mid S_F\right) &\ge \prod_{\substack{\ F' \in \cR_f \\ F' \neq F}} \Pr_{\Gamma}\left(S_{F'}^c \mid S_F\right) 
        \ge \prod_{i=0}^{r-3} \prod_{\substack{F' \in \cR_f \\ |V(F') \cap V(F)| = i}} \left(1-\left(2p(1-p)\right)^{2(r-2-i)}\right) \\
        &\ge \prod_{i=0}^{r-3} \left(e^{-2^{2r}p^{2(r-2-i)}}\right)^{\Delta_i(\cH)} \\
        &\ge e^{-\sum_{i=0}^{r-3} 2^{2r} p^{2(r-2-i)} \cdot \max\left\{r^2 b^{r-2-i} p^{\binom{r-2}{2} - \binom{i}{2}}, (\log n)^{r+1}\right\}},
    \end{align*}
    where the third inequality holds as $p=o(1)$, and the last inequality holds due to the assumed properties of $\cH$.     
    Since $p < n^{-\xi}$, for every $0 \le i \le r-3$ it is clear that $p^{2(r-2-i)} (\log n)^{r+1} = o(1)$.

    Let us now bound from above $g(x) \coloneqq  p^{2(r-2-x)} b^{r-2-x} p^{\binom{r-2}{2} - \binom{x}{2}}$ in the interval $[0,r-2]$. The convexity of $g(x)$ implies that its maximum in $[0,r-2]$ is attained at either $x = 0$ or $x = r-2$. For $x= 0$, by \eqref{align:bound on 2 density with b} we have
    \[
        g(0) = b^{r-2} p ^{2(r-2) + \binom{r-2}{2}} = b^{r-2} p^{\binom{r}{2} - 1} \le 2^{r-2}\lambda,
    \]
    and for $x = r-2$, we have $g(r-2) = 1$. 
    We conclude that 
    \[
        \sum_{i=0}^{r-3} 2 p^{2(r-2-i)} \cdot \max\left\{r^2 b^{r-2-i} p^{\binom{r-2}{2} - \binom{i}{2}}, (\log n)^{r+1}\right\} = O_\lambda (1),
    \]
    concluding the proof of the claim.
\end{proof}

Fix $\beta$ as guaranteed by Claim \ref{claim:alpha_f high frequency}. Further, denote by $S \subseteq A_1 \times A_2$ the (random) set of edges $f \in A_1 \times A_2$ with $\left|w_f\right| = 1$.
\begin{claim}\label{claim:many good f high frequencies}
    Conditioned on $\{G_0[B]=\Gamma\}$, with probability at least $1 - e^{-\sqrt{n}}$, we have
    \[
        |S| \ge \begin{cases}
            0.01 \cdot n^2 p^{2(r-2)}, \quad &p \ge n^{-\xi}, \\
            0.01 
            \cdot \beta n^2, \quad &p < n^{-\xi}.
        \end{cases}
    \]
\end{claim}
\begin{proof}
    Let $M_1, \dots, M_{n/3}$ be a partition of $A_1\times A_2$ into edge-disjoint perfect matchings. For every $i \in [n/3]$, denote by $S_i$ the set of edges $f \in M_i$ with $\left|w_f\right| = 1$. We claim that, for every $i \in [n/3]$, with conditional probability at least $1 - e^{-n^{0.6}}$ we have
    \[
        |S_i| \ge \begin{cases}
            \frac{1}{2} p^{2(r-2)}\cdot   \frac{n}{3}, \quad &p \ge n^{-\xi}, \\
            \frac{1}{2} \beta  \cdot \frac{n}{3}, \quad & p < n^{-\xi}.
        \end{cases}
    \]

    Fix $i \in [n/3]$. As $M_i$ is a matching of size $n/3$, the events $\{\left|w_f\right| = 1\}$ where $f\in M_i$ are pairwise (conditionally) independent. Therefore, if $p \ge n^{-\xi}$, then by Claim \ref{claim:alpha_f high frequency} and Claim \ref{lemma:chernoff}, we have 
    \begin{align*}
        \Pr\left(|S_i| <  \frac{1}{2} p^{2(r-2)}  \cdot \frac{n}{3} \mid G_0[B] = \Gamma\right) &\le \Pr\left(\text{Bin}\left(\frac{n}{3}, p^{2(r-2)}\right) < \frac{1}{2} p^{2(r-2)}  \cdot \frac{n}{3}\right) \\&\le e^{-\Theta\left(\frac{n}{3} \cdot p^{2(r-2)}\right)} \le e^{-n^{0.6}},
    \end{align*}
    where in the last inequality we use the assumption that $\xi$ is small as a function of $r$. Similarly, if $p < n^{-\xi}$, then Claim \ref{claim:alpha_f high frequency} and Claim \ref{lemma:chernoff} imply that
    \begin{align*}
        \Pr\left(|S_i| <  \frac{1}{2}\beta  \cdot \frac{n}{3} \mid G_0[B] = \Gamma\right) \le \Pr\left(\text{Bin}\left(\frac{n}{3}, \beta\right) < \frac{1}{2}\beta  \cdot \frac{n}{3} \right) \le e^{-\Theta(n)} \le e^{-n^{0.6}}.
    \end{align*}
    The claim now follows by a simple union bound over the choices of $i \in [n/3]$.
\end{proof}

Finally, we bound the characteristic function of $X_r$. For this, denote by $\cS$ the event that the assertion of Claim \ref{claim:many good f high frequencies} holds, that is 
\[
        |S| \ge \begin{cases}
            0.01 \cdot n^2 p^{2(r-2)}, \quad &p \ge n^{-\xi}, \\
            0.01 \cdot \beta n^2, \quad &p < n^{-\xi}.
        \end{cases}
\]
Denoting by $\cS^*$ the event $\cS\cap \{G_0[B]=\Gamma\}$, applying Claim \ref{claim:many good f high frequencies}, and recalling \eqref{align:def of X tilde high freq}, we have 
\begin{align*}
    \left|\E\left[e^{it \Tilde{X} / \sigma_r}\mid G_0[B]=\Gamma\right]\right| &\le e^{-\sqrt{n}}+\left|\E\left[ \prod _{f\in A_1\times A_2}e^{it\left( w_f \cdot \mathbf{1}_{f \in G_0}\right)/\sigma_r} \mid \cS^*\right]  \right|.
\end{align*}
Conditioned on (the random variable) $(\Gamma_{0},\Gamma_1)\coloneqq (G_0\setminus (A_1\times A_2),G_1\setminus (A_1\times A_2))$ the weights $w_f$ are fixed for every $f\in A_1\times A_2$, and thereby the set $S$ is also fixed. In addition, conditioned on $(\Gamma_{0},\Gamma_{1})$, the random variables $\{\mathbf{1}_{f\in G_0}:f\in A_1\times A_2\}$ are mutually independent. We have
\begin{align*}
    \left\lvert \E\left[\prod_{f\in  A_1\times A_2}e^{it\left(w_f \cdot \mathbf{1}_{f \in G_0}\right)/\sigma_r}  \mid  \cS^*\right] \right\rvert &= \left\lvert \E\left[\E\left[\prod_{f\in  A_1\times A_2}e^{it\left(w_f \cdot \mathbf{1}_{f \in G_0}\right)/\sigma_r} \mid \Gamma_0, \Gamma_1 \right]  \mid  \cS^* \right] \right\rvert \\
    &\le  \E\left[\left\lvert\E\left[\prod_{f\in  A_1\times A_2}e^{it\left(w_f \cdot \mathbf{1}_{f \in G_0}\right)/\sigma_r} \mid \Gamma_0, \Gamma_1 \right] \right\rvert  \mid  \cS^* \right] \\
    &\le \E\left[\left\lvert \prod_{f\in  S} \E\left[e^{it\left(w_f \cdot \mathbf{1}_{f \in G_0}\right)/\sigma_r} \mid \Gamma_0, \Gamma_1 \right] \right\rvert  \mid  \cS^* \right].
\end{align*}
Therefore, it suffices to show that the above is bounded by $n^{-3K}$.

Assume that $(\Gamma_0,\Gamma_1)$ is fixed and that it satisfies $\cS^*$. Let $\cD$ be the event that for both $i\in[2]$ we have $G_i\setminus (A_1\times A_2) = \Gamma_i$. Claim \ref{claim:char-bounds} yields
\begin{align}\label{eq:bound-conditioned-on-good-edges}
    \left\lvert \prod_{f\in S}\E\left[e^{it\left(w_f \cdot \mathbf{1}_{f \in G_0}\right)/\sigma_r} \ \big\lvert \  \cD\right] \right\rvert 
    &\le \prod_{f\in S} \left(1-8p(1-p) \cdot \left\| \frac{t w_f}{2 \pi \sigma_r}\right\|^2\right)\le e^{-8p(1-p)\cdot \frac{t^2}{(2\pi \sigma_r)^2} \cdot |S|},
\end{align}
where in the last inequality we used the definition of $S$, which implies that for every $f \in S$, we have $\left|w_f\right| = 1$ and thus $\left|\frac{t  w_f}{2\pi \sigma_r}\right| \leq \frac{1}{2}$ since $t \le \pi \sigma_r$.

The proof is now concluded as follows. First, since $(\Gamma_0,\Gamma_1)$ satisfies $\cS^*$, we may assume the assertion of Claim \ref{claim:many good f high frequencies}. Then, if $p \ge n^{-\xi}$, we have
\begin{align}\label{eq:bound-conditioned-on-good-edges-plugged1}
    8p(1-p)\cdot \frac{t^2}{(2\pi \sigma_r)^2} \cdot |S| &\ge 0.04\cdot  n^2 p^{2(r-2)+1} \cdot t^2 / (2\pi \sigma_r)^2= \Omega\left(n^{2-2\gamma} p^{2(r-2) + 1}\right) \gg \sqrt{n},
\end{align}
where the second inequality holds as $t \ge n^{-\gamma} \cdot \sigma_r$ and the last inequality holds provided $\xi$ and $\gamma$ are sufficiently small.
On the other hand, if $p < n^{-\xi}$, then,
\begin{align}\label{eq:bound-conditioned-on-good-edges-plugged2}
    8p(1-p)\cdot \frac{t^2}{(2\pi \sigma_r)^2} \cdot |S| &\ge  0.01\cdot \beta n^2 p\cdot t^2 / (\pi \sigma_r)^2 \ge  \Omega_{\lambda}(n^{2-2\gamma} p) \gg \sqrt{n},
\end{align}
where the second inequality follows as $t \ge n^{-\gamma} \cdot \sigma_r$ and the last inequality follows provided that $\gamma$ is sufficiently small.
Plugging \eqref{eq:bound-conditioned-on-good-edges-plugged1} and \eqref{eq:bound-conditioned-on-good-edges-plugged2} in \eqref{eq:bound-conditioned-on-good-edges} concludes the proof of Lemma \ref{lemma:characteristic function integrable} for high frequencies.

\section{Proof of Lemma \ref{lemma:characteristic function integrable - low freq and dense p}}\label{section:low frequencies}
In this section, we prove Lemma \ref{lemma:characteristic function integrable - low freq and dense p}. Fix positive constants $c < \frac{3}{2r}$ and a sufficiently small $\epsilon$. We show that for every constant $K > 0$ and for every $n^{-c} \le p \le \frac{1}{2}$, the following holds:
\begin{align}\label{align:bound of char function low frequencies}
    \left\lvert \E\left[e^{itX_r/\sigma_r}\right]\right \rvert \le n^{-K}\leq t^{-K/\epsilon} \quad \text{for every} \quad n^{\epsilon}\leq t \leq n^{1/2+\epsilon}p.
\end{align} 
We note that, as $K$ is allowed to depend on $\epsilon$, the above indeed suffices.

Similarly to the previous sections, we bound the characteristic function of $X_r$, by a characteristic function of a sum of weighted and independent Bernoulli random variables. As we are dealing with `low' frequencies, we must require that the weights are large. For this, we slightly alter our proof scheme. Indeed, we will use a partition of $[n]$ to three sets $\cP = (A_1, A_2, Z)$ with $|A_1| = |A_2|$ and $|Z|=\Theta(n)$. In this case we employ Corollary \ref{cor: the decoupling lemma new} with the partition $\cP$ and with $\cE=\{A_1 \times A_2\}$. This reduces the problem to bounding the characteristic function of some signed count of cliques in $(G_0,G_1)$ where $G_0\sim G_{n,p}$ and $G_1\sim (A_1\times A_2)_p$.

As in the proof of Lemma \ref{lemma:characteristic function integrable} for the medium frequency range (see Section \ref{section:main lemma for medium frequencies}), we split the signed sum into two --- the sum of singed cliques that use \emph{exactly} one edge from $A_1 \times A_2$, and its complement. 
We then essentially reduce the problem to bounding the characteristic function of the `first' sum and the moments of the `second' sum, this will be done as in \eqref{align:splitting to main and error}.

Let us remark that the bound of the characteristic function follows an approximation by weighted and independent Bernoulli random variables. Indeed, we assign to each $f \in A_1 \times A_2$, a weight $w_f$ in the following way: $w_f=0$ for $f\not\in G_0\triangle G_1$ and for the complementary case we let $|w_f|$ be the number of cliques lying on $f$ with additional vertices taken from $Z$.
The sign of $w_f$ is defined to be positive if  $f \in G_0$ and negative otherwise. Then, using standard concentration inequalities we control the number of edges $f$ with a `typical' edge weight. An important point here is that a typical weight is of the order of the expectation of the number of extensions of $f$ to an $r$-clique rather than the standard deviation as was exploited earlier.

The method described above requires $t$ to be bounded from below. On the other hand, the application of the method of moments in the complementary case requires an upper bound on $t$. Both bounds on $t$ depend on $|A_1|$ ($=|A_2|$). It remains to show that for every $n^{\epsilon}\leq t\leq n^{1/2+\epsilon}p$ there is a choice of $a$ such that for arbitrary fixed sets $A_1$ and $A_2$, each of size $a$, both of the above bounds are achieved.

\subsection{Proof of \eqref{align:bound of char function low frequencies}}
Let $n^{3/4-\epsilon}p^{-1/2}\leq a \leq n^{1-\frac{\epsilon}{2}} $ be an integer, which exists as $p\gg n^{-1/2}$, and let $\cP=(A_1, A_2, Z)$ be an $(a, 0, r-2, r)$-partition of $[n]$. Recall that $\cR =\cR(\cP)$ is the set of copies of $K_r$ in $K_n$ with at least one vertex in both $A_1$ and $A_2$, and that furthermore, $\cR_{2}\coloneqq \cR_{2, 0} \subseteq \cR$ is the set of copies of $K_r$ in $\cR$ with \emph{exactly} two vertices in $\cP_A =A_1 \cup A_2$. For convenience, set $\cR' \coloneqq \cR \setminus \cR_{2}$.

Letting $G_0 \sim G_{n, p}$ and $G_1 \sim \left(A_1 \times A_2\right)_p$ be two independent random graphs, Corollary \ref{cor: the decoupling lemma new} applied with the $\cE=\{A_1 \times A_2\}$ yields that
\begin{align}\label{align:after decoupling low freq}
    \left|\E\left[e^{itX_r / \sigma_r}\right] \right|^2 \le \left|\E\left[e^{it(\Tilde{X} + \Tilde{Y})/\sigma_r}\right]\right|,
\end{align}
where
\[
    \Tilde{X} = \sum_{F \in \cR_{2}} S_{G_0, G_1}(F) \quad \text{and} \quad \Tilde{Y} = \sum_{F \in \cR'} S_{G_0, G_1}(F).
\]
Thus, to prove 
\eqref{align:bound of char function low frequencies} in the required frequency range, it suffices to show the following:
\begin{align}\label{align:bound of char function low frequencies after dec}
    \left|\E\left[e^{it(\Tilde{X} + \Tilde{Y})/\sigma_r}\right]\right| \le n^{-2K} \quad \text{for every} \quad \log n \cdot \frac{n}{a} \le t \le \frac{n^{2-\epsilon}}{a^2}.
\end{align}
Indeed, by varying over all integers $n^{3/4-\epsilon}p^{-1/2}\leq a\leq n^{1-\frac{\epsilon}{2}}$ we obtain the required inequality:
\begin{align*}
    \left|\E\left[e^{it(\Tilde{X} + \Tilde{Y})/\sigma_r}\right]\right| \le n^{-2K} \quad \text{for every} \quad n^{\epsilon}\le t \leq n^{1/2+\epsilon}p.
\end{align*}

To this end, let $L = L(K, \epsilon) > 0$ be a sufficiently large integer and note that by Taylor's theorem (see e.g.\ \cite[Lemma 3.3.19]{Dur2010}), for every $t > 0$, we have the following:
\begin{align}\label{align:splitting to main and error11}
    \left|\E\left[e^{it(\tilde{X} + \tilde{Y})/\sigma_r}\right]\right| \le \left|\E\left[e^{it\Tilde{X}/\sigma_r} \cdot P_{2L-1}\left(\frac{it \Tilde{Y}}{\sigma_r}\right)\right]\right| + O_L\left(\E\left[\left(\frac{t  \Tilde{Y}}{\sigma_r}\right)^{2L}\right]\right),
\end{align}
where $P_{2L-1}$ is the Taylor polynomial of $e^x$ of order $2L-1$.
In the rest of the section, we prove that each of the above summands is bounded by $n^{-3K}$. 

We start with the first summand in \eqref{align:splitting to main and error11}. 
Similarly to the proof of the first item of Lemma~\ref{lemma:main lemma for decoupling} (see Section \ref{section:decoupling}), it suffices to prove the claim below. The reduction here is analogous to the reduction in the proof of the first item in Lemma \ref{lemma:main lemma for decoupling} to \eqref{align:bound of conditional characteristic function}. For simplicity of presentation, we omit the proof of this reduction.

\begin{claim}\label{claim:decoupling bound low freq}
   Suppose that $D > 2Lr + K$ is a constant and that $\ell \le 2L-1$ is an integer. Then, for every $\log n \cdot \frac{n}{a} \le t \le  \frac{n^{2-\epsilon}}{a^{2}}$ and every $F_1, \dots, F_{\ell} \in \cR'$, we have    
    \[
        \left|\E\left[e^{it\Tilde{X}/\sigma_r} \mid \left\lvert \prod_{j=1}^{\ell} S_{G_0, G_1}(F_j) \right\rvert = 1\right]\right| \le n^{-D}.
    \]    
\end{claim}
\begin{proof}
    Throughout the proof we write $\Pr^*$ to denote the probability measure of $(G_0,G_1)$ conditioned on the event $\left\lvert \prod_{j=1}^{\ell} S_{G_0, G_1}(F_j) \right\rvert = 1$. Further, we write $\E^*$ for the corresponding (conditional) expectation operator.
    
    Let $\cH$ be the $(r-2)$-uniform hypergraph with vertex set $Z$ and edge set consisting of copies of $(r-2)$-cliques in $G_0$, that is $S \in \cH$ if and only if $G_0[S] \cong K_{r-2}$. Denote by $\cE^*$ the event that the following holds:
    \begin{enumerate}[label=(F\arabic*)]
        \item\label{item:sec10-typical-1} $\frac{1}{2r^r} n^{r-2} p^{\binom{r-2}{2}} \le e(\cH) \le r^2 n^{r-2} p^{\binom{r-2}{2}}$, and
        \item\label{item:sec10-typical-2} For every $j \in [r-2]$, denoting by $\Delta_j(\cH)$ the $j$-th maximum degree of $\cH$, we have $\Delta_j(\cH) \le \max\left\{r^2 n^{r-2-j} p^{\binom{r-2}{2} - \binom{j}{2}}, (\log n)^{r+1}\right\}$.
    \end{enumerate}
    We remark that $\cE^*$ is determined by $G_0[\cP_Z]$ and note that for every $j \in [r-2]$ we have $n^{r-2} p^{\binom{r-2}{2}} \gg \max\left\{n^{r-2-j} p^{\binom{r-2}{2} - \binom{j}{2}}, (\log n)^{r+1}\right\}$. Hence, by applying Theorem \ref{claim: lower bound on extensions} and Claim~\ref{claim:upper bound on extensions} for extensions with vertices outside $V(F)$, one can easily verify that $\Pr^*(\cE^*) \ge 1 - n^{-2D}$. 
    \[
        \left|\E^*\left[e^{it\Tilde{X}/\sigma_r}\right]\right| \le n^{-2D} +  \E^* \left|\E^*\left[e^{it\Tilde{X}/\sigma_r} \mid G_0[\cP_{Z}]\right]  \mid  \cE^*\right|,
    \]    
    and it suffices to bound the latter term by $n^{-2D}$.

    Fix $\Gamma \subseteq K_n[\cP_Z]$ satisfying the event $\cE^*$ and let $\cH=\cH(\Gamma)$ be the hypergraph as defined in the event.
    To conclude the proof of the claim it is enough to prove the following:
    \[
    \left|\E^*\left[e^{it\Tilde{X}/\sigma_r} \mid G_0[\cP_Z] = \Gamma \right]\right|\leq n^{-2D}.
    \]
    As in previous sections, for every $f \in A_1 \times A_2$, set    
     \[
    w_f \coloneqq \sum_{\substack{F \in \cR_{2} \\ \text{s.t. } f \in F}} S_{G_0 \cup \{f\}, G_1 \setminus \{f\}}(F),
    \]    
    and note that by Definition \ref{definition:sign} we have 
    \begin{align*}
        \Tilde{X} =\sum_{f\in A_1\times A_2}\sum_{\substack{F \in \cR_{2} \\ \text{s.t. } f \in F}} S_{G_0 \cup f, G_1 \setminus f}(F) \cdot \bigl(\mathbf{1}_{f \in G_0}-\mathbf{1}_{f \in G_1}\bigl)=\sum_{f \in A_1 \times A_2} w_f \cdot \bigl(\mathbf{1}_{f \in G_0}-\mathbf{1}_{f \in G_1}\bigl).
    \end{align*}
    Note that, for every edge $f \in \cP_A \times \cP_Z$ that do not share vertices with $F \coloneqq \bigcup_{i=1}^{\ell} F_i$, and $i\in \{0,1\}$ the event $\{f \in G_i\}$ is independent from the events $\cE^*$ and \{$G_0[\cP_Z] = \Gamma\}$. Hence, for every such edge $f \in A_1 \times A_2$ and every $j \in \{2, 4\}$, we have
    \begin{equation}\label{eq:sec10-nd/4th-moment}
        \E^*\left[w_f^j \mid G_0[\cP_Z] = \Gamma\right] =\sum_{H_1, \dots, H_j \in \cH} p^{2v(H_1 \cup \dots\cup H_j)}.
    \end{equation}

    The following claim provides tight bounds to the quantity in \eqref{eq:sec10-nd/4th-moment}.
    \begin{claim}\label{claim:bound on Delta_d}
        For every $j \in \{2, 4\}$, we have
         \[
            \left(e(\cH) p^{2(r-2)}\right)^j \le \E^*\left[w_f^j \mid G_0[\cP_Z] = \Gamma\right] \le \left(r^{3r} \cdot e(\cH) p^{2(r-2)}\right)^j.
        \] 
    \end{claim}
    \begin{proof}
        We first show that, for every $d \in \{0, \ldots, r-3\}$, we have
        \begin{equation}\label{eq:max-degree-edges}
            \Delta_{d}(\cH) p^{2(r-2-d)} \le e(\cH) p^{2(r-2)},
        \end{equation}
        where we abuse notation and write $\Delta_0(\cH)$ to denote $e(\cH)$.
        For $d=0$, by definition we have equality, and so it suffices to prove the claim for $d > 0$.
        Fix $d>0$ and assume that $\Delta_d(\cH) \le (\log n)^{r+1}$. Then, $\Delta_{d}(\cH) p^{2(r-2-d)} \le (\log n)^{r+1} p^{2}$ and hence it suffices to show that $(\log n)^{r+1} p^{2} \ll e(\cH) p^{2(r-2)}$. Then, as we condition on \ref{item:sec10-typical-1}, this is equivalent to showing that $(\log n)^{r+1} \ll n^{r-2} p^{\binom{r-2}{2} +2r - 6}$.
        Since $p \ge n^{-c}$ and $c < \frac{2}{3r}$ is a constant,
        \[
           n^{r-2} p^{\binom{r-2}{2} +2r - 6} \gg n^{r-2-\frac{2}{3r}\left(\frac{r^2-5r+6+4r-12}{2}\right)}  = n^{\frac{2r^2-5r+6}{3r}} \gg (\log n)^{r+1}.
        \]
        
        Next, assume that $\Delta_d(\cH) > (\log n)^{r+1}$. Let $f(x) \coloneqq n^{r-2-x} p^{\binom{r-2}{2}-\binom{x}{2} + 2(r-2-x)}$, and note that by \ref{item:sec10-typical-1} and \ref{item:sec10-typical-2}, in this case, $f(0)=\Theta(e(\cH)p^{2(r-2)})$ and $f(d)=\Omega\left(\Delta_d(\cH) p^{2(r-2-x)}\right)$. Moreover, note that $f$ is convex and thus, $f(d) \le \max\{f(1),f(r-3)\}$. Thus, to conclude the proof of \eqref{eq:max-degree-edges}, it is enough to show that $f(0)\gg \max\{f(1),f(r-3)\}$.

        Indeed, $f(0)\gg f(1)$ is equivalent to $p\gg n^{-1/2}$ which holds as $p\ge n^{-c}$ with $c<\frac{3}{2r}\le \frac{1}{2}$. Further, $f(0)\gg f(r-1)=np^{r-1}$ is equivalent to $n^{r-3} p^{\binom{r-2}{2} + r - 3} \gg 1$ which we verify next. Indeed, since $p \ge n^{-c} \gg n^{-\frac{3}{2r}}$, we have
        \begin{align*}
            n^{r-3} p^{\binom{r-2}{2} + r - 3} \gg n^{r-3 - \frac{3}{2r}\left(\binom{r-2}{2} + r - 3\right)} = n^{\frac{r-3}{4}} \ge 1.
        \end{align*}
        
        Finally, for every $j \in \{2, 4\}$
    \begin{align*}
        \left(e(\cH) p^{2(r-2)}\right)^j \le\sum_{H_1, \dots, H_j \in \cH} p^{2v(H_1 \cup \dots\cup H_j)} \le e(\cH) p^{2(r-2)} \left(\sum_{d=0}^{r-3} \binom{j \cdot r}{d} \Delta_d(\cH) p^{2(r-2-d)}\right)^{j}.
    \end{align*}  
    This, combined with \eqref{eq:sec10-nd/4th-moment} and \ref{eq:max-degree-edges}, as well as the assumption that $j\leq r$, completes the proof of the claim.
    \end{proof}

    By Markov's inequality and Claim \ref{claim:bound on Delta_d},
    \[
        \Pr^*\left(|w_f| > r^{10r} e(\cH) p^{2(r-2)} \mid G_0[\cP_Z] = \Gamma\right) \le \frac{\E^*[|w_f|^2\mid G_0[\cP_Z]=\Gamma]}{\left(r^{10r} e(\cH) p^{2(r-2)}\right)^2}\leq \frac{1}{r^{14r}}.
    \]  
    In addition, by the Paley-Zygmund inequality \ref{claim: Paley-Zygmund}, we have        
    \[
        \Pr^*\left(|w_f| \geq \frac{1}{2}  e(\cH) p^{2(r-2)} \mid G_0[\cP_Z] = \Gamma\right) \ge \left(1-\frac{1}{4}\right)^2 \frac{\left(\E^*[|w_f|^2\mid G_0[\cP_Z]=\Gamma]\right)^2}{\E^*[|w_f|^4\mid G_0[\cP_Z]=\Gamma]}\geq \frac{1}{2r^{12r}}.
    \]  
    Therefore,
    \begin{align}\label{align:alpha_f good values low frequencies}
        \Pr^*\left(\frac{1}{2}e(\cH) p^{2(r-2)} \le |w_f| \le r^{10r} e(\cH) p^{2(r-2)} \mid G_0[\cP_Z] = \Gamma\right) \ge \frac{1}{2r^{12r}} - \frac{1}{r^{14r}}\geq \frac{1}{r^{14r}}.
    \end{align}

    Denote by $S\subseteq A_1\times A_2$ the set of edges $f$ satisfying $\frac{1}{2} e(\cH) p^{2(r-2)} \le |w_f| \le r^{10r} e(\cH) p^{2(r-2)}$ and also $f \cap V(F) = \emptyset$. In addition, denote by $\cS$ the event that $|S| \ge \frac{a^2}{2r^{14r}}$. Given the events $\{G_0[\cP_Z] = \Gamma\}$ and $\cE^*$, the weights corresponding to any pair of vertex-disjoint edges $f,g\in A_1\times A_2$ are independent. Indeed, subject to this condition, the random variables $w_f$ and $w_g$ depend only on the edges between the endpoints of $f$ and $\cP_Z$, and between the endpoints of $g$ and $\cP_Z$, respectively. Since $f$ and $g$ are vertex-disjoint, these sets of underlying edges are disjoint, which implies that $w_f$ and $w_g$ are conditionally independent.
    Thus, similarly to the proof of Claim \ref{claim:number of good f}, by decomposing $A_1\times A_2$ into perfect matchings, one can show that with probability at least $1 - e^{-\Theta(a)}\geq 1-n^{-3D}$, the event $\cS$ is satisfied.
    Hence,    
\begin{align*}
    \left|\E^*\left[e^{it\Tilde{X}/\sigma_r} \mid G_0[\cP_Z] = \Gamma\right]\right| \le  n^{-3D} + \left|\E^*\left[e^{it\Tilde{X}/\sigma_r} \mid \cS \wedge \{G_0[\cP_Z] = \Gamma\}\right]\right|,
\end{align*}
while setting the random variables $W_0=G_0\cap ((A_1\times A_2)^{c}\cup F)$ and $W_1=G_1\cap F$, the second term above equals
\begin{align*}
    &\left\lvert \E^*\left[\E^*\left[e^{it \sum_{f \in A_1\times A_2} w_f \cdot \left(\mathbf{1}_{f \in G_0} - \mathbf{1}_{f \in G_1}\right)} \mid W_0,W_1\right] \mid \cS \wedge  \{G_0[\cP_Z] = \Gamma\}\right]\right\rvert\\
    &\le \E^*\left|\E^*\left[e^{it \sum_{f \in S} w_f \cdot \left(\mathbf{1}_{f \in G_0} - \mathbf{1}_{f \in G_1}\right)} \mid W_0,W_1]\right| \mid \cS \wedge  \{G_0[\cP_Z] = \Gamma\} \right] \\
    &\le \E^*\left|\E^*\left[e^{it \sum_{f \in S} w_f \cdot \mathbf{1}_{f \in G_0}} \mid W_0,W_1]\right| \mid \cS \wedge  \{G_0[\cP_Z] = \Gamma\} \right].
\end{align*}

Note that for $f\in S$, the random variables $\mathbf{1}_{f\in G_0}$ are pairwise independent in the probability measure $\Pr^*$ conditioned on $\cS \wedge \{G_0[\cP_{Z}]=\Gamma\}$ and any assignment for $(W_0,W_1)$. Further, by the assumption that
\[
    t \leq n^{1/2+\epsilon}p \ll  n \sqrt{p} = \frac{\sigma_r}{n^{r-2} p^{\binom{r}{2} - 1}},
\]
we have $\left\|\frac{w_f\cdot t}{\sigma_r}\right\|=\frac{w_f\cdot t}{\sigma_r}$ for every $f \in S$. Thus, similarly to the last part of Section~\ref{sec:high-freq}, Claim \ref{claim:char-bounds} gives
\begin{align*}
    \left|\E^*\left[e^{it\Tilde{X}/\sigma_r} \mid G_0[\cP_Z] = \Gamma\right]\right| &\le \exp\left(-8p(1-p)\cdot \left(\frac{e(\cH) p^{2(r-2)} \cdot t}{4\pi\sigma_r}\right)^2\cdot \frac{a^2}{2 r^{14r}} \right).
\end{align*}
The proof is concluded by using the assumption that $t\geq \log n\cdot \frac{n}{a}$ and $e(\cH) \ge \frac{1}{2r^r} n^{r-2} p^{\binom{r-2}{2}}$, as we then obtain 
\begin{align*}
    \left|\E^*\left[e^{it\Tilde{X}/\sigma_r} \mid G_0[\cP_Z] = \Gamma\right]\right| &\le \exp\left(-\Omega\left((\log n)^2\right)\right) \leq n^{-3D}.\qedhere
\end{align*}
\end{proof}

The following claim shows that the second term in \eqref{align:splitting to main and error11} is bounded by $n^{-3K}$.

\begin{claim}\label{claim:high moments bound low freq}
    Suppose that $t \le  \frac{n^{2-\epsilon}}{a^2}$. Then, for a sufficiently large $L = L(K, \delta)$, we have
    \[
        \E\left[\left(\frac{t \Tilde{Y}}{\sigma_r}\right)^{2L}\right] \le n^{-3K}.
    \]    
\end{claim}
\begin{proof}
    We first show that 
    \begin{equation}\label{eq:high-moments-low-freq-dense}
        \E\left[\Tilde{Y}^{2L}\right] = O_L\left(\left( \frac{a^{4}}{n^{4}}\cdot n^{2r-2} p^{2\binom{r}{2} - 1}\right)^{L}\right).
    \end{equation}
    Upon setting $\cR_i\coloneqq \cR_{i,0}$, we note that 
    \[
        \Tilde{Y} = \sum_{F \in \cR'} S_{G_0, G_1}(F) = \sum_{i=3}^{r} \sum_{F \in \cR_{i}} S_{G_0, G_1}(F).
    \]
    For every $i \in \{3, 4, \dots, r\}$, set $\Tilde{Y}_i \coloneqq \sum_{F \in \cR_{i}} S_{G_0, G_1}(F)$
    and note that similarly to Section \ref{section:high moments} (see \eqref{align:Y high moment}), we have
    \[
        \E\left[\Tilde{Y}^{2L}\right] = \sum_{i=3}^{r} O_L\left(\E\left[\Tilde{Y}_i^{2L}\right]\right).
    \]    
    
    Fix $i \in \{3, 4, \dots, r\}$, and recall from Definition \ref{def: count of good rainbow copies} that $T^{i}_{2L}(G)\coloneqq T_{\cP, 2L}^{i, 0}(G)$ denotes the collection of $2L$-tuples $(F_1,\ldots,F_{2L})$ of $r$-cliques in $\cR_{i}$ that are contained in $G$, and such that for every $i'\in [2L]$ the set $V(F_{i'})$ intersects both  $A_1\cap \bigcup_{j\neq i'} V(F_{j})$ and $A_2\cap \bigcup_{j\neq i'}V(F_{j})$. Further, for every $F \in \cR_{i}$, the condition $S_{G_0, G_1}(F) \neq 0$ implies that $F \subseteq G_0 \cup G_1$. 
    Hence, by Claim \ref{claim:sufficient condition for sign zero in expectation new},
    \begin{align*}
        \E\left[\Tilde{Y}_i^{2L}\right] &= \sum_{F_1, \dots, F_{2L} \in \cR_{i}} \E\left[\prod_{j=1}^{2L} S_{G_0, G_1}(F_j) \right] \\&\le \E\left[\left\lvert T_{2L}^{i}(G_{n, 2p}) \right\rvert\right] = \sum_{(F_1, \dots, F_{2L}) \in T_{2L}^{i}(K_n)} \Pr\left( \bigcup_{j=1}^{2L} F_j \subseteq G_{n, 2p}\right).
    \end{align*}
    
    Recall that $|\cP_Z|/|\cP_A|=\Theta(n/a)$.
    Due to size consideration, there is a surjection mapping each sequence $(F_1,\ldots, F_{2L})\in T_{2L}^{i}(K_{n})$ to a sequence $(F_1',\ldots, F_{2L}')\in T_{2L}^{r}(K_{n})$ by replacing every vertex in $\cP_Z$ with a vertex in $\cP_A$ such that the following holds:
\begin{itemize}
    \item $e\left(\bigcup_{\ell=1}^{2L}F_{\ell}\right)=e\left(\bigcup_{\ell=1}^{2L}F'_{\ell}\right)$, and 
    \item the preimage of each sequence $(F_1',\ldots, F_{2L}')\in T_{2L}^{r}(K_{n})$ is of order $O\left((n/a)^{2(r-i)L}\right)$.
\end{itemize}
This together with Observation \ref{claim:number of vertices in F}, imply that 
    \begin{align*}
        \E\left[\left\lvert T_{2L}^{i}(G_{n, 2p}) \right\rvert\right] &= O\left(\left(\frac{n}{a}\right)^{2(r-i)L} \E\left[\left\lvert T_{2L}^{r}(G_{n, 2p}) \right\rvert\right]\right) \\&= O\left(\left(\frac{n}{a}\right)^{2(r-i)L} \sum_{m=r}^{(2r-2)L} (2a)^{m}(2p)^{\varphi_{r,0}(m,0)}\right),
    \end{align*}
    where $\varphi_{i,j}$ is as defined in the paragraph above Lemma \ref{claim:number of edges in F}.  
    
    For every $m\in \{r,\ldots ,(2r-2)L\}$ let $F_m$ be a minimiser of $\varphi_{r,0}(m,0)$. Claim \ref{claim:number of edges in F with D} applied with
    \[
        s =  s_1 = r, \ s_2 = 0,\  \ell_1 = 2,\ \ell_2=0,\ W_1 = V(F_m),\ \text{and} \ W_2 = \emptyset,
    \]
    yields
    \[
        \varphi_{r,0}(m,0) \ge \frac{(r-1)m + 2 \cdot \max\{0, m-rL\}}{2}.
    \]
    Therefore,
    \begin{align*}
        \E\left[\left\lvert T_{2L}^{i}(G_{n, 2p}) \right\rvert\right] &= O\left(\left(\frac{n}{a}\right)^{2(r-i)L} \sum_{m=r}^{(2r-2)L} (2a)^{m} (2p)^{\frac{(r-1)m + 2 \cdot \max\{0, m-rL\}}{2}}\right) \\
        &= O\left(\left(\frac{n}{a}\right)^{2(r-i)L} \left( a^r p^{\binom{r}{2}} + \left(a^{r} p^{\binom{r}{2}}\right)^L + \left(a^{2r-2} p^{2\binom{r}{2} - 1}\right)^L\right)\right).
    \end{align*}
    
    As $c < \frac{3}{2r}$ and as $\epsilon$ is sufficiently small we have $a^{r-2}p^{\binom{r}{2}-1}\geq 1$. Therefore, combined with the fact that $i\geq 3$ we conclude \eqref{eq:high-moments-low-freq-dense} since
    \begin{align*}
        \E\left[\left\lvert T_{2L}^{i}(G_{n, 2p}) \right\rvert\right]  = O\left(\left(\frac{n}{a}\right)^{2(r-3)L} \left(a^{2r-2} p^{2\binom{r}{2} - 1}\right)^L\right)= O\left(\left( \left(\frac{a}{n}\right)^{4} \cdot n^{2r-2} p^{2\binom{r}{2} - 1}\right)^L\right).
    \end{align*}
    
    The proof is now complete by letting $L$ be a sufficiently large constant and recalling the assumption that $t\leq \frac{n^{2-\epsilon}}{a^2}$:  
    \[
         \E\left[\left(\frac{t \Tilde{Y}}{\sigma_r}\right)^{2L}\right] =O_L\left(\left( \frac{t^{2}\cdot a^{4}\cdot n^{2r-2}p^{2\binom{r}{2}-1}}{n^{4}\cdot n^{2r-2}p^{2\binom{r}{2}-1}}\right)^{L}\right) = O_L\left(n^{-2\epsilon L}\right) \leq n^{-3K}.\qedhere
    \]    
\end{proof}

Finally we derive \eqref{align:bound of char function low frequencies after dec} completing the proof of Lemma \ref{lemma:characteristic function integrable - low freq and dense p} by combining \eqref{align:after decoupling low freq}, \eqref{align:splitting to main and error11}, Claim~\ref{claim:decoupling bound low freq}, and Claim~\ref{claim:high moments bound low freq}.

\section{Proof of Lemma \ref{lemma:characteristic function integrable - low freq and low p}}\label{sec:low p}
The goal of this section is to prove Lemma \ref{lemma:characteristic function integrable - low freq and low p}. We begin by setting some notation. Let $n^{-1/m(K_r)} \ll p \le C n^{-1/m_2(K_r)} (\log n)^{\frac{4r}{e(K_r)-1}}$ where $C>0$ is some absolute constant, and let $\tilde{n} \le n/r$ be the largest integer satisfying $\tilde{n}^{r-2} p^{\binom{r}{2}-1} \le (\log n)^{-10}$. 

As in previous sections, we will use the decoupling scheme laid out in Section \ref{subsection: decoupling new}.
To this end, partition $[n]$ into $\cP= (A_1,A_2,B_1, \dots, B_{r-2},Z)$, such that for every $(i,j)\in [2]\times [r-2]$ we have $|A_i|=|B_j|=\tilde{n}$; this is possible due to the choice of $\tilde{n}\leq n/r$. 
Define the following family of edge sets $\cE \coloneqq  \{A_i \times B_j : (i,j) \in [2]\times [r-2]\}$ to be used when applying the decoupling argument.
Lastly, for convenience, we let $B \coloneqq \bigcup_{i=1}^{r-2} B_i$.

Recall that $\cR = \cR(\cP)$ is the set of copies of $K_r$ in $K_n$ taking at least (in this case exactly) one vertex in each of the sets $A_1, A_2, B_1, \dots, B_{r-2}$. Let $G_0 \sim G_{n, p}$ and let $G_1 \sim (\cP_A \times \cP_B)_p$ be two independent random graphs.
By Corollary \ref{cor: the decoupling lemma new}, 
\begin{align}\label{align:after decoupling high frequencies low p}
    \left|\E\left[e^{it X_r / \sigma_r} \right]\right|^{2^{2(r-2)}} \le \left|\E_{G_0, G_1}\left[e^{it \Tilde{X} / \sigma_r}\right]\right|,
\end{align}
where $\Tilde{X} = \sum_{F \in \cR} S_{G_0, G_1}(F)$. For every edge $f \in A_1 \times A_2$, define its weight as
\[
    w_f \coloneqq \sum_{f \in F \in \cR} S_{G_0 \cup f, G_1}(F),
\]
and note that 
\[
    \Tilde{X} = \sum_{f \in A_1 \times A_2} w_f \cdot \mathbf{1}_{f \in G_0}.
\]

Denote by $S \subseteq A_1 \times A_2$ the (random) set of edges $f \in A_1 \times A_2$ with $\left|w_f\right| = 1$.
\begin{lemma}\label{claim:many good f high frequencies low p}
    There exists $c > 0$ such that $\Pr\left(|S| \ge c \tilde{n}^{r} p^{\binom{r}{2} - 1}\right) \ge 1 - n^{-\omega(1)}.$
\end{lemma}
\begin{proof}
    As $K_r$ is strictly balanced, as $p\gg n^{-1/m(K_r)}$ by assumption, and as $\tilde{n}\geq n/(\log n)^{30}$ we have $\tilde{n}^{r} p^{\binom{r}{2} - 1} \ge n^{\Omega(1)}$. In addition, $K_r$ minus an edge is strictly-balanced, and thus, a standard use of Janson's inequality \ref{theorem: Janson} implies that with probability at least $1-n^{-\omega(1)}$, the number of $\cP$-rainbow copies of $K_r$ minus an edge from $A_1 \times A_2$ in $G_0 \triangle G_1$ is at least $\Omega\left(\tilde{n}^{r} p^{\binom{r}{2} - 1}\right)$.

    Let $\cF$ be the family of pairs of distinct $\cP$-rainbow copies of $K_r$ sharing an edge in $A_1\times A_2$. 
    \begin{claim}
        With probability at least $1-n^{-\omega(1)}$ the random graph $(G_0\triangle G_1) \cup (A_1\times A_2)$ contains at most $o\left(\tilde{n}^{r} p^{\binom{r}{2} - 1}\right)$ copies from $\cF$. 
    \end{claim}
    
    Before proving the claim, we note that the lemma follows trivially from it. Indeed, the claim implies that the set of $\cP$-rainbow copies of $K_r$ in $(G_0 \triangle G_1) \cup (A_1 \times A_2)$ minus the set of copies of $K_r$ that appear in $\cF$ is with probability at least $1-n^{-\omega(1)}$ of order $\Omega\left(\tilde{n}^{r} p^{\binom{r}{2} - 1}\right)$.
    In particular, this shows that with probability $1-n^{-\omega(1)}$, the graph $(G_0\triangle G_1)\cup(A_1\times A_2)$ contains at least $\Omega\left(\tilde{n}^{r}p^{\binom{r}{2}-1}\right)$ many $\cP$-rainbow copies that do not share an edge from $A_1\times A_2$ with any other $\cP$-rainbow copy, and thus $\Omega\left(\tilde{n}^{r}p^{\binom{r}{2}-1}\right)$ edges in $A_1\times A_2$ with weight $\pm 1$, as required.
    
    \begin{proof}
    To prove the claim, we consider the inhomogeneous random graph $H$ defined as follows: independently retain every edge in $A_i\times B_j$ and in $B_i\times B_j$ with probability $2p$, and every edge in $A_1 \times A_2$ with probability $1$.    
    Let $\cF^*$ be the set of graphs obtained as a union of two distinct $\cP$-rainbow copies of $K_r$ that share an edge in $A_1\times A_2$.
    The claim follows once we show that with probability at least $1-n^{-\omega(1)}$ the graph $H$ contains $o\left(\tilde{n}^{r} p^{\binom{r}{2} - 1}\right)$ many copies in $\cF^*$. For every $F \in \cF^*$, there exists a unique pair $\{F_1, F_2\}$ of $\cP$-rainbow copies of $K_r$ such that $F_1 \cup F_2 = F$. Denote by $\cF^{**}$ the collection of such pairs. 
    
    To prove the above, we will use the Kim--Vu inequality, i.e.\ Theorem \ref{theorem: Kim Vu}, which requires us to bound $E'$ in Theorem \ref{theorem: Kim Vu}. 
    In our setting
    \[
    E' = \max\left\{\E[Y_J]: J\subseteq K_n \text{ with }1\leq e(J)\leq 2\binom{r}{2}-1\right\},
    \]
    where $Y_J$ is the number of extensions of $J$ to a copy $F\in \cF^*$ in $H$. 
    To bound $E'$ we will bound each $\E[Y_J]$ individually. 
    Let $\Delta_1$ be the maximum of 
    \[
        \sum_{F_1}\Pr(F_1\subseteq H\mid J_1\subseteq H)
    \]    
    over $J_1\subseteq K_n$ with $1\leq e(J_1)\leq \binom{r}{2}$, where the sum ranges over all $\cP$-rainbow copies $F_1\supseteq J_1$.
    In addition, let $\Delta_2$ be the maximum of
    \[
       \sum_{F_2}\Pr(F_2\subseteq H\mid F_1\cup J_2\subseteq H),
    \]    
    over $J_2\subseteq K_n$ with $e(J_2) \leq \binom{r}{2}-1$ and all $\cP$-rainbow copies $F_1$, where the sum ranges over all $\cP$-rainbow copies $F_2\supseteq J_2$ satisfying $\{F_1,F_2\}\in \cF^{**}$.
    With these definitions for every $J\subseteq K_n$ with $1 \le e(J)\le 2\binom{r}{2}-1$, we may bound $\E[Y_J]$ as follows:
    \begin{align*}
        \E[Y_J] \leq \sum_{J_1,J_2}  \sum_{ F_1}\Pr(F_1\subseteq H\mid J_1\subseteq H)\sum_{F_2}\Pr(F_2\subseteq H\mid F_1\cup J_2\subseteq H)= O(\Delta_1 \Delta_2),
    \end{align*}
    where in the first sum $J_1$ and $J_2$ satisfy $J_1\cup J_2 =J$ and $e(J_1)\geq 1$, in the second sum the $\cP$-rainbow copy $F_1$ contains $J_1$, and in the third sum $F_2$ is a $\cP$-rainbow copy containing $J_2$ and satisfying $\{F_1,F_2\}\in \cF^{**}$. We conclude that $E' = O(\Delta_1\Delta_2)$.
    
    We now bound $\Delta_1$ and $\Delta_2$. Note that for every $\cP$-rainbow copy $F_1$ and every $J_1 \subseteq F_1$ with $k \le r-1$ vertices, we have $e(F_1 \setminus J_1) \ge \binom{r}{2} - 1 - \binom{k}{2}$. Therefore,
     \[
        \Delta_1 = O\left(1 + \sum_{k=2}^{r-1} \tilde{n}^{r-k} p^{\binom{r}{2} - \binom{k}{2} - 1}\right).
     \]
     By the convexity of $x\mapsto \tilde{n}^{r-x}p^{\binom{r}{2}-\binom{x}{2}-1}$, we further have
    \[
        \Delta_1= O\left(1 + \tilde{n}^{r-2} p^{\binom{r}{2} - 2} + \tilde{n} p^{r-2}\right)\leq \tilde{n}^{r} p^{\binom{r}{2} - 1} \cdot \tilde{n}^{-\xi},
    \]
    where $\xi$ is some positive real depending on $r$ only. In a similar way,
    \[
        \Delta_2 = O\left(\sum_{k=2}^{r} \tilde{n}^{r-k} p^{\binom{r}{2} - \binom{k}{2}}\right) = O\left(\tilde{n}^{r-2} p^{\binom{r}{2}-1} + 1\right)=O(1),
    \]
    where the middle equality follows by the convexity of $x\mapsto \tilde{n}^{r-x}p^{\binom{r}{2}-\binom{x}{2}}$.

    Finally, we conclude that 
    \[
        E' = O(\Delta_1 \Delta_2) = O\left(\tilde{n}^{r} p^{\binom{r}{2} - 1} \cdot \tilde{n}^{-\xi}\right).
    \]
    Recalling the notation from the Kim-Vu inequality, we further have 
    \[
        E \coloneqq \max\{\E[Y], E'\} = O\left(\tilde{n}^{r} p^{\binom{r}{2} - 1}\right),
    \]
    where $Y$ counts the number of copies of $F\in \cF^*$ in $H$.
    Thus, $(E E')^{1/2} = O\left(\tilde{n}^{r} p^{\binom{r}{2} - 1} \cdot \tilde{n}^{- \xi/2}\right)$.
    Now, by applying the Kim-Vu inequality with $k\coloneqq 2\binom{r}{2}-1$ and $\lambda = (\log n)^{2}$, we have
    \begin{align*}
        \Pr\left(|Y - \E[Y]| > a_k (E E')^{1/2} \lambda^k\right) < 18 e^{-\lambda} \tilde{n}^{k-1} = n^{-\omega(1)},
    \end{align*}
    for some constant $a_k$ depending on $k$ only. This concludes the proof of the claim.
    \end{proof}
    As was mentioned, this claim implies Lemma \ref{claim:many good f high frequencies low p}, and thus completes its proof.
\end{proof}

We are now ready to bound the characteristic function of $X_r$ and conclude the proof of Lemma \ref{lemma:characteristic function integrable - low freq and low p}. For this, denote by $\cS$ the event that the assertion of Lemma \ref{claim:many good f high frequencies low p} holds, that is $|S| \ge c \tilde{n}^{r}p^{\binom{r}{2}-1}$.
Now Lemma \ref{claim:many good f high frequencies low p} implies that 
\begin{align*}
    \left|\E_{G_0, G_1}\left[e^{it \Tilde{X} / \sigma_r}\right]\right| &\le n^{-\omega(1)}+\left|\E_{G_0, G_1}\left[ \prod _{f\in A_1\times A_2}e^{it\left( w_f \cdot \mathbf{1}_{f \in G_0}\right)/\sigma_r} \mid \cS\right]  \right|.
\end{align*}
Let $(\Gamma_{0},\Gamma_1)\coloneqq (G_0\setminus (A_1\times A_2),G_1\setminus (A_1\times A_2))$, and note that for every $f\in A_1\times A_2$ the weights $w_f$ are fixed conditionally on $(\Gamma_0,\Gamma_1)$, and so is the set $S$. 
In addition, conditioning on $(\Gamma_{0},\Gamma_{1})$, the random variables $\{\mathbf{1}_{f\in G_0}:f\in A_1\times A_2\}$ remain mutually independent and thus we have
\[
   \left|\E\left[e^{it \Tilde{X} / \sigma_r}\right]\right| \leq n^{-\omega(1)}+\E\left[\prod_{f\in S} \left\lvert \E\left[e^{it\left(w_f \cdot \mathbf{1}_{f \in G_0}\right)/\sigma_r} \ \right] \right\rvert\big\lvert \  \cS  \land (\Gamma_0,\Gamma_1) \right].
\]

By assuming that $(\Gamma_0,\Gamma_1)$ is fixed and that it satisfies $\cS$, Claim \ref{claim:char-bounds} yields  
\begin{align*}
    \left\lvert \E\left[\prod_{f\in S}e^{it\left(w_f \cdot \mathbf{1}_{f \in G_0}\right)/\sigma_r} \mid (\Gamma_0,\Gamma_1)\right] \right\rvert 
    &\le  \prod_{f\in S}\left(1-8p(1-p) \cdot \left\| \frac{t w_f}{2 \pi \sigma_r}\right\|^2\right)  \\&\le \exp\left(-8p(1-p)\cdot \frac{t^2}{(2\pi \sigma_r)^2} \cdot |S|\right),
\end{align*}
where the last inequality uses the definition of $S$ and the assumption that $t \le \pi \sigma_r$, which together imply that for every $f \in S$ we have $\left|\frac{t  w_f}{2\pi \sigma_r}\right| \leq \frac{1}{2}$. Under the event $\cS$ we have $|S| \ge c \tilde{n}^{r}p^{\binom{r}{2}-1}$, which yields the following bound
\begin{equation*}\label{eq:bound-char-func-sparse}
    \left|\E_{G_0, G_1}\left[e^{it \Tilde{X} / \sigma_r}\right]\right| \le n^{-\omega(1)} + \exp\left({-\Omega\left( \tilde{n}^{r}p^{\binom{r}{2}} \cdot \left(\frac{t}{ \sigma_r}\right)^2\right)} \right).
\end{equation*}
This combined with \eqref{align:after decoupling high frequencies low p} concludes the proof of Lemma \ref{lemma:characteristic function integrable - low freq and low p}.

\section*{Acknowledgements}
The authors would like to thank Wojciech Samotij for many helpful discussions during the early stages of the project, as well as for several suggestions that improved the readability of the paper.
This research was supported in part by the Israel Science Foundation grant 2110/22 and 1028/16,
by the ERC Consolidator Grant 101044123 (RandomHypGra), and by the ERC Starting Grant 633509.

\bibliographystyle{plain}
\bibliography{bib.bib}

@article {Ros2011,
    AUTHOR = {Ross, Nathan},
     TITLE = {Fundamentals of {S}tein's method},
   JOURNAL = {Probab. Surv.},
  FJOURNAL = {Probability Surveys},
    VOLUME = {8},
      YEAR = {2011},
     PAGES = {210--293},
      ISSN = {1549-5787},
   MRCLASS = {60F05 (05C80 60C05)},
  MRNUMBER = {2861132},
MRREVIEWER = {Anant\ P.\ Godbole},
       DOI = {10.1214/11-PS182},
       URL = {https://doi.org/10.1214/11-PS182},
}

@article {BarKarRuc1989,
    AUTHOR = {Barbour, A. D. and Karo\'nski, Micha\l{} and Ruci\'nski,
              Andrzej},
     TITLE = {A central limit theorem for decomposable random variables with
              applications to random graphs},
   JOURNAL = {J. Combin. Theory Ser. B},
  FJOURNAL = {Journal of Combinatorial Theory. Series B},
    VOLUME = {47},
      YEAR = {1989},
    NUMBER = {2},
     PAGES = {125--145},
      ISSN = {0095-8956,1096-0902},
   MRCLASS = {60F05 (05C80 60C05)},
  MRNUMBER = {1047781},
MRREVIEWER = {John\ C.\ Wierman},
       DOI = {10.1016/0095-8956(89)90014-2},
       URL = {https://doi.org/10.1016/0095-8956(89)90014-2},
}

@book {Dur2010,
    AUTHOR = {Durrett, Rick},
     TITLE = {Probability: theory and examples},
    SERIES = {Cambridge Series in Statistical and Probabilistic Mathematics},
    VOLUME = {31},
   EDITION = {Fourth},
 PUBLISHER = {Cambridge University Press, Cambridge},
      YEAR = {2010},
     PAGES = {x+428},
      ISBN = {978-0-521-76539-8},
   MRCLASS = {60-01},
  MRNUMBER = {2722836},
       DOI = {10.1017/CBO9780511779398},
       URL = {https://doi.org/10.1017/CBO9780511779398},
}

@article{AraMat2023,
  title={Local central limit theorem for triangle counts in sparse random graphs},
  author={Ara{\'u}jo, Pedro and Mattos, Let{\'\i}cia},
  journal={arXiv preprint arXiv:2307.09446},
  year={2023}
}

@article {SahSaw2022,
    AUTHOR = {Sah, Ashwin and Sawhney, Mehtaab},
     TITLE = {Local limit theorems for subgraph counts},
   JOURNAL = {J. Lond. Math. Soc. (2)},
  FJOURNAL = {Journal of the London Mathematical Society. Second Series},
    VOLUME = {105},
      YEAR = {2022},
    NUMBER = {2},
     PAGES = {950--1011},
      ISSN = {0024-6107,1469-7750},
   MRCLASS = {60F05 (05C80 60C05)},
  MRNUMBER = {4389308},
       DOI = {10.1112/jlms.12523},
       URL = {https://doi.org/10.1112/jlms.12523},
}

@article{Ber2018,
  title={A local limit theorem for cliques in G (n, p)},
  author={Berkowitz, Ross},
  journal={arXiv preprint arXiv:1811.03527},
  year={2018}
}

@article{Ber2016,
  title={A quantitative local limit theorem for triangles in random graphs},
  author={Berkowitz, Ross},
  journal={arXiv preprint arXiv:1610.01281},
  year={2016}
}

@article {GilKop2016,
    AUTHOR = {Gilmer, Justin and Kopparty, Swastik},
     TITLE = {A local central limit theorem for triangles in a random graph},
   JOURNAL = {Random Structures Algorithms},
  FJOURNAL = {Random Structures \& Algorithms},
    VOLUME = {48},
      YEAR = {2016},
    NUMBER = {4},
     PAGES = {732--750},
      ISSN = {1042-9832,1098-2418},
   MRCLASS = {05C80 (05C35 60F05)},
  MRNUMBER = {3508725},
MRREVIEWER = {Anant\ P.\ Godbole},
       DOI = {10.1002/rsa.20604},
       URL = {https://doi.org/10.1002/rsa.20604},
}

@article {ErdRen1960,
    AUTHOR = {Erd\H os, P. and R\'enyi, A.},
     TITLE = {On the evolution of random graphs},
   JOURNAL = {Magyar Tud. Akad. Mat. Kutat\'o{} Int. K\"ozl.},
  FJOURNAL = {A Magyar Tudom\'anyos Akad\'emia. Matematikai Kutat\'o{}
              Int\'ezet\'enek K\"ozlem\'enyei},
    VOLUME = {5},
      YEAR = {1960},
     PAGES = {17--61},
      ISSN = {0541-9514},
   MRCLASS = {05.40},
  MRNUMBER = {125031},
MRREVIEWER = {John\ Riordan},
}

@article {Bol1981,
    AUTHOR = {Bollob\'as, B\'ela},
     TITLE = {Threshold functions for small subgraphs},
   JOURNAL = {Math. Proc. Cambridge Philos. Soc.},
  FJOURNAL = {Mathematical Proceedings of the Cambridge Philosophical
              Society},
    VOLUME = {90},
      YEAR = {1981},
    NUMBER = {2},
     PAGES = {197--206},
      ISSN = {0305-0041,1469-8064},
   MRCLASS = {05C99 (05C35 60C05)},
  MRNUMBER = {620729},
MRREVIEWER = {Micha\l\ Karo\'nski},
       DOI = {10.1017/S0305004100058655},
       URL = {https://doi.org/10.1017/S0305004100058655},
}

@incollection {KarRuc1983,
    AUTHOR = {Karo\'nski, Micha\l{} and Ruci\'nski, Andrzej},
     TITLE = {On the number of strictly balanced subgraphs of a random
              graph},
 BOOKTITLE = {Graph theory (\L ag\'ow, 1981)},
    SERIES = {Lecture Notes in Math.},
    VOLUME = {1018},
     PAGES = {79--83},
 PUBLISHER = {Springer, Berlin},
      YEAR = {1983},
      ISBN = {3-540-12687-2},
   MRCLASS = {05C80},
  MRNUMBER = {730636},
MRREVIEWER = {Charles\ M.\ Grinstead},
       DOI = {10.1007/BFb0071616},
       URL = {https://doi.org/10.1007/BFb0071616},
}

@book {Kar1984,
    AUTHOR = {Karo\'nski, Micha\l},
     TITLE = {Balanced subgraphs of large random graphs},
    SERIES = {Seria Matematyka [Mathematics Series]},
    VOLUME = {7},
      NOTE = {With a Polish summary},
 PUBLISHER = {Uniwersytet im. Adama Mickiewicza w Poznaniu, Pozna\'n},
      YEAR = {1984},
     PAGES = {48},
   MRCLASS = {05C80 (05C05 05C38 60C05)},
  MRNUMBER = {779093},
MRREVIEWER = {Paul\ A.\ Catlin},
}

@article {Ruc1988,
    AUTHOR = {Ruci\'nski, Andrzej},
     TITLE = {When are small subgraphs of a random graph normally
              distributed?},
   JOURNAL = {Probab. Theory Related Fields},
  FJOURNAL = {Probability Theory and Related Fields},
    VOLUME = {78},
      YEAR = {1988},
    NUMBER = {1},
     PAGES = {1--10},
      ISSN = {0178-8051,1432-2064},
   MRCLASS = {60C05 (05C80 60F05)},
  MRNUMBER = {940863},
MRREVIEWER = {G.\ N.\ Bagaev},
       DOI = {10.1007/BF00718031},
       URL = {https://doi.org/10.1007/BF00718031},
}

@inproceedings {NowWie1986,
    AUTHOR = {Nowicki, Krzysztof and Wierman, John C.},
     TITLE = {Subgraph counts in random graphs using incomplete
              {$U$}-statistics methods},
 BOOKTITLE = {Proceedings of the {F}irst {J}apan {C}onference on {G}raph
              {T}heory and {A}pplications ({H}akone, 1986)},
   JOURNAL = {Discrete Math.},
  FJOURNAL = {Discrete Mathematics},
    VOLUME = {72},
      YEAR = {1988},
    NUMBER = {1-3},
     PAGES = {299--310},
      ISSN = {0012-365X,1872-681X},
   MRCLASS = {05C80},
  MRNUMBER = {975550},
MRREVIEWER = {Alan\ M.\ Frieze},
       DOI = {10.1016/0012-365X(88)90220-8},
       URL = {https://doi.org/10.1016/0012-365X(88)90220-8},
}

@article {Jan1990,
    AUTHOR = {Janson, Svante},
     TITLE = {Orthogonal decompositions and functional limit theorems for
              random graph statistics},
   JOURNAL = {Mem. Amer. Math. Soc.},
  FJOURNAL = {Memoirs of the American Mathematical Society},
    VOLUME = {111},
      YEAR = {1994},
    NUMBER = {534},
     PAGES = {vi+78},
      ISSN = {0065-9266,1947-6221},
   MRCLASS = {60F17 (05C80 60C05)},
  MRNUMBER = {1219708},
MRREVIEWER = {Andrew\ D.\ Barbour},
       DOI = {10.1090/memo/0534},
       URL = {https://doi.org/10.1090/memo/0534},
}

@article {FoxKwanSau2021a,
    AUTHOR = {Fox, Jacob and Kwan, Matthew and Sauermann, Lisa},
     TITLE = {Combinatorial anti-concentration inequalities, with
              applications},
   JOURNAL = {Math. Proc. Cambridge Philos. Soc.},
  FJOURNAL = {Mathematical Proceedings of the Cambridge Philosophical
              Society},
    VOLUME = {171},
      YEAR = {2021},
    NUMBER = {2},
     PAGES = {227--248},
      ISSN = {0305-0041,1469-8064},
   MRCLASS = {60C05 (05C80 05D40 60E15)},
  MRNUMBER = {4299587},
       DOI = {10.1017/s0305004120000183},
       URL = {https://doi.org/10.1017/s0305004120000183},
}

@article {FoxKwanSau2021b,
    AUTHOR = {Fox, Jacob and Kwan, Matthew and Sauermann, Lisa},
     TITLE = {Anti-concentration for subgraph counts in random graphs},
   JOURNAL = {Ann. Probab.},
  FJOURNAL = {The Annals of Probability},
    VOLUME = {49},
      YEAR = {2021},
    NUMBER = {3},
     PAGES = {1515--1553},
      ISSN = {0091-1798,2168-894X},
   MRCLASS = {05C80 (05C30 60C05)},
  MRNUMBER = {4255152},
MRREVIEWER = {Lyuben\ R.\ Mutafchiev},
       DOI = {10.1214/20-aop1490},
       URL = {https://doi.org/10.1214/20-aop1490},
}

@article {Har1960,
    AUTHOR = {Harris, T. E.},
     TITLE = {A lower bound for the critical probability in a certain
              percolation process},
   JOURNAL = {Proc. Cambridge Philos. Soc.},
  FJOURNAL = {Proceedings of the Cambridge Philosophical Society},
    VOLUME = {56},
      YEAR = {1960},
     PAGES = {13--20},
      ISSN = {0008-1981},
   MRCLASS = {60.00},
  MRNUMBER = {115221},
MRREVIEWER = {G.\ Newell},
}

@article{SahSawZhu2024,
  title={Local limit theorem for joint subgraph counts},
  author={Sah, Ashwin and Sawhney, Mehtaab and Zhu, Daniel G},
  journal={arXiv preprint arXiv:2412.09535},
  year={2024}
}

@article {RolRos2015,
    AUTHOR = {R\"ollin, Adrian and Ross, Nathan},
     TITLE = {Local limit theorems via {L}andau-{K}olmogorov inequalities},
   JOURNAL = {Bernoulli},
  FJOURNAL = {Bernoulli. Official Journal of the Bernoulli Society for
              Mathematical Statistics and Probability},
    VOLUME = {21},
      YEAR = {2015},
    NUMBER = {2},
     PAGES = {851--880},
      ISSN = {1350-7265,1573-9759},
   MRCLASS = {60E15 (60B10 60C05)},
  MRNUMBER = {3338649},
MRREVIEWER = {Hac\`ene\ Djellout},
       DOI = {10.3150/13-BEJ590},
       URL = {https://doi.org/10.3150/13-BEJ590},
}

@article {BerSahSaw2021,
    AUTHOR = {Berkowitz, Ross and Sah, Ashwin and Sawhney, Mehtaab},
     TITLE = {Number of arithmetic progressions in dense random subsets of
              {$\Bbb{Z}/n\Bbb{Z}$}},
   JOURNAL = {Israel J. Math.},
  FJOURNAL = {Israel Journal of Mathematics},
    VOLUME = {244},
      YEAR = {2021},
    NUMBER = {2},
     PAGES = {589--620},
      ISSN = {0021-2172,1565-8511},
   MRCLASS = {11B25 (60F05)},
  MRNUMBER = {4344037},
MRREVIEWER = {Paul\ Potgieter},
       DOI = {10.1007/s11856-021-2180-7},
       URL = {https://doi.org/10.1007/s11856-021-2180-7},
}

@article{KwaSau2023,
  title={Resolution of the quadratic Littlewood--Offord problem},
  author={Kwan, Matthew and Sauermann, Lisa},
  journal={arXiv preprint arXiv:2312.13826},
  year={2023}
}

@article {MR1774845,
    AUTHOR = {Kim, Jeong Han and Vu, Van H.},
     TITLE = {Concentration of multivariate polynomials and its
              applications},
   JOURNAL = {Combinatorica},
  FJOURNAL = {Combinatorica. An International Journal on Combinatorics and
              the Theory of Computing},
    VOLUME = {20},
      YEAR = {2000},
    NUMBER = {3},
     PAGES = {417--434},
      ISSN = {0209-9683,1439-6912},
   MRCLASS = {05C80 (60C05)},
  MRNUMBER = {1774845},
MRREVIEWER = {Tomasz\ J.\ \L uczak},
       DOI = {10.1007/s004930070014},
       URL = {https://doi.org/10.1007/s004930070014},
}

@book {JanLucRuc2000,
    AUTHOR = {Janson, Svante and \L uczak, Tomasz and Rucinski, Andrzej},
     TITLE = {Random graphs},
    SERIES = {Wiley-Interscience Series in Discrete Mathematics and
              Optimization},
 PUBLISHER = {Wiley-Interscience, New York},
      YEAR = {2000},
     PAGES = {xii+333},
      ISBN = {0-471-17541-2},
   MRCLASS = {05C80 (60C05 82B41)},
  MRNUMBER = {1782847},
MRREVIEWER = {Mark\ R.\ Jerrum},
       DOI = {10.1002/9781118032718},
       URL = {https://doi.org/10.1002/9781118032718},
}

@article {Spe1990,
    AUTHOR = {Spencer, Joel},
     TITLE = {Counting extensions},
   JOURNAL = {J. Combin. Theory Ser. A},
  FJOURNAL = {Journal of Combinatorial Theory. Series A},
    VOLUME = {55},
      YEAR = {1990},
    NUMBER = {2},
     PAGES = {247--255},
      ISSN = {0097-3165,1096-0899},
   MRCLASS = {05C80},
  MRNUMBER = {1075710},
MRREVIEWER = {Zbigniew\ Palka},
       DOI = {10.1016/0097-3165(90)90070-D},
       URL = {https://doi.org/10.1016/0097-3165(90)90070-D},
}

@article {JanOleRuc2004,
    AUTHOR = {Janson, Svante and Oleszkiewicz, Krzysztof and Ruci\'nski,
              Andrzej},
     TITLE = {Upper tails for subgraph counts in random graphs},
   JOURNAL = {Israel J. Math.},
  FJOURNAL = {Israel Journal of Mathematics},
    VOLUME = {142},
      YEAR = {2004},
     PAGES = {61--92},
      ISSN = {0021-2172,1565-8511},
   MRCLASS = {05C80},
  MRNUMBER = {2085711},
MRREVIEWER = {David\ B.\ Penman},
       DOI = {10.1007/BF02771528},
       URL = {https://doi.org/10.1007/BF02771528},
}

\appendix
\section{Proof of Claim \ref{claim:bound on expectations of T_2L}}\label{app:Appendix}

This appendix is devoted to the proof of Claim~\ref{claim:bound on expectations of T_2L}. For the convenience of the reader, we recall that $p\in(0,1/2]$ satisfies \eqref{align: p assumption}, that $k \in [r-3]$, and that $0 < b \le a_k$ are integers that satisfy Properties \ref{item:b condition} and \ref{item:k condition}. In addition, we restate the claim. 

\ClaimAppendix*

The proof of the four items above are similar to each other with the first step being the following. Upon fixing $a_k,b$ and $p$, the left-hand sides of the inequalities in the first three items are maximised at the extreme values of $n_A$ and $n_B$ (as a product of an exponential in $n_A$ and an exponential in $n_B$). 
Therefore, it suffices to prove \ref{n_A and n_B large} for $n_A\in\{iL,(2i-2)L\}$ and $n_B\in\{jL,(2j-(r-k-2))L\}$, \ref{n_A large n_B small} for $n_A\in\{iL,(2i-2)L\}$ and $n_B\in\{j,jL\}$, and \ref{n_A small and n_B large} for $n_A\in\{i,iL\}$ and $n_B\in\{jL,(2j-(r-k-2))L\}$. As we cannot straight away employ this trick to prove \ref{n_A and n_B small}, the proof of this part is slightly different.
The technical and tedious part is then to verify each inequality. 

By property \ref{item:k condition} (applied with $k'=0 < k$), we have $n^k p^{\binom{r}{2} - \binom{r-k}{2}} \ge n^{\delta}$ and therefore, $a_k^{k}p^{\binom{r}{2}-\binom{r-k}{2}}\geq n^{\delta -k\epsilon_k} \geq 1$ provided that $k\epsilon_k<\delta$.
Throughout the rest of the proofs in this appendix, we will frequently use Property \ref{item:b condition} and Property \ref{item:k condition} with $k'=0$ as described above; for the convenience of the reader, we repeat their weaker versions, here:
\begin{align}
    \label{align:a_k inequality to use in hm}
    a_k^{k} \cdot p^{\binom{r}{2} - \binom{r-k}{2}} , b^{r-k-2}\cdot  p^{\binom{r-k}{2} - 1} \ge 1.
\end{align}

\subsection{Proof of \ref{n_A and n_B large}}
    Suppose that $n_A \in \{i L, (2i-2)L\}$ and $n_B \in \{j L, (2j-(r-k-2)) L\}$ maximise 
    \[
        a_k^{n_A} b^{n_B} p^{\frac{(n_A + n_B)(r-1) + (r-k)\left((n_A - i L) + (n_B - j L)\right)}{2}} = \left(a_k p^{\frac{2r - k - 1}{2}}\right)^{n_A} \left(b p^{\frac{2r - k - 1}{2}}\right)^{n_B} p^{- \frac{(r-k)(iL + jL)}{2}}.
    \]
    Note that if $n_B = (2j - (r-k-2))L$ then as $b \le a_k$ we may assume that $n_A = (2i-2)L$.
    We will prove the desired upper bound separately for each of the three possible cases (all cases expect $n_A = iL$ and $n_B = (2j - (r-k-2))L$). For simplicity of presentation, throughout the proof, we write
    \[
        f\coloneqq\frac{(n_A + n_B)(r-1) + (r-k)\left((n_A - i L) + (n_B - j L)\right)}{2}.
    \]
    \noindent\textit{Case 1:} Assume that $n_A = (2i-2)L$ and $n_B = (2j-(r-k-2)) L$.\\
    By the assumption of this case, as well as the assumption that $i + j = r$, we have
        \begin{align*}
            f = \left(2\binom{r}{2} - \binom{r-k}{2}\right)L.
        \end{align*}
        Therefore, it remains to show that
        \[
            a_k^{(2i-2)L} b^{(2j-(r-k-2))L} p^{\left(2\binom{r}{2} - \binom{r-k}{2}\right) L} \le n^{C_r} \left(a_k^{2k+2} b^{r-k-2} p^{2\binom{r}{2} - \binom{r-k}{2}}\right)^{L},
        \]
        which is equivalent to
        \begin{equation}\label{eq:claim1-case1}
             a_k^{2i - 2} b^{2j - (r-k-2)} \le n^{C_r / L} a_k^{2k+2} b^{r-k-2}.
        \end{equation}
        To see this, recall that $b\leq a_k$, and thus among all $(i,j)\in I$ such that $i+j=r$, the pair maximising $a_k^{i}b^{j}$ is the one with the largest value of $i$. This optimal pair is $(k+2,r-k-2)$, and when plugging it into the left-hand side of \eqref{eq:claim1-case1}, the proof follows for any $C_r \ge 0$.

    \noindent\textit{Case 2:} Assume that $n_A = (2i - 2)L$ and $n_B = j L$.\\
    Similarly to the previous case, as $i + j = r$, it is straightforward to check that
        \begin{align*}
           f= \left(\binom{r}{2} + \frac{(i-2)(r-1) + (r-k)(i-2)}{2}\right)L.
        \end{align*}
        It is then enough to show that
        \[
            a_k^{2i - 2} b^{r-i} p^{\binom{r}{2} + \frac{(i-2)(r-1) + (r-k)(i-2)}{2}} \le n^{C_r/L} a_k^{2k+2} b^{r-k-2} p^{2\binom{r}{2} - \binom{r-k}{2}}.
        \]
        Recall that $2 \le i \le k+2$ and thus the left-hand side above is maximised at $i \in \{2, k+2\}$. On the one hand, if $i = k+2$,
        \[
            a_k^{2i - 2} b^{r-i} p^{\binom{r}{2} + \frac{(i-2)(r-1) + (r-k)(i-2)}{2}}  =  a_k^{2k+2} b^{r-k-2} p^{2\binom{r}{2} - \binom{r-k}{2}},
        \]
        as needed. On the other hand, if $i = 2$, since $b \le a_k$ and by \eqref{align:a_k inequality to use in hm},
        \begin{align*}
             a_k^{2i - 2} b^{r-i} p^{\binom{r}{2} + \frac{(i-2)(r-1) + (r-k)(i-2)}{2}} &= a_k^2 b^{r-2} p^{\binom{r}{2}} \le \left(a_k^k p^{\binom{r}{2} - \binom{r-k}{2}}\right) \cdot a_k^{2+k} b^{r-k-2} p^{\binom{r}{2}}\\& = a_k^{2k+2} b^{r-k-2} p^{2\binom{r}{2} - \binom{r-k}{2}},
        \end{align*}
    concluding the proof in this case.

    \noindent\textit{Case 3:} Assume that $n_A = i L$ and $n_B = j L$.\\
    In this case, we have $f=\binom{r}{2}\cdot L$. This implies that it is then enough to show the following: 
        \begin{align}\label{align:ineq4 used in hm}
            a_k^{i} b^{r-i} p^{\binom{r}{2}} \le  a_k^{2k+2} b^{r-k-2} p^{2\binom{r}{2} - \binom{r-k}{2}}.
        \end{align}
        Since $b \le a_k$ and since $2 \le i \le k+2$, the left-hand side above is maximised taking $i = k+2$. Then, by  \eqref{align:a_k inequality to use in hm}, the required inequality is obtained as follows:
        \begin{equation}\label{align:ineq4 used in hm}
             \begin{split}
             a_k^{i} b^{r-i} p^{\binom{r}{2}} &\le a_k^{k+2} b^{r-k-2} p^{\binom{r}{2}} \le \left(a_k^k p^{\binom{r}{2}-\binom{r-k}{2}}\right) \cdot a_k^{k+2} b^{r-k-2} p^{\binom{r}{2}} \\
            &=a_k^{2k+2} b^{r-k-2} p^{2\binom{r}{2} - \binom{r-k}{2}}. \qedhere
            \end{split}
        \end{equation}

\subsection{Proof of \ref{n_A large n_B small}}
    Suppose that $n_A \in \{i L, (2i-2)L\}$ and $n_B \in \{j, j L\}$ maximise 
    \[
        a_k^{n_A} b^{n_B} p^{n_A \cdot j + \frac{n_A(i - 1) + 2(n_A- i L)}{2} + \frac{n_B(j - 1)}{2}}.
    \]
    Throughout the proof, we write
    \[
        f \coloneqq n_A \cdot j + \frac{n_A(i - 1) + 2(n_A- i L)}{2} + \frac{n_B(j - 1)}{2}.
    \]
    We distinguish four cases according to the values of $n_A$ and $n_B$, proving the desired upper bound in each case.
    
    \noindent\textit{Case 1:} Suppose that $n_A = (2i - 2)L$ and $n_B = j L$.\\
    Recalling that $i+j=r$, we get
        \begin{align*}
            f = \left(2\binom{r}{2} - \binom{r-i+2}{2}\right).
        \end{align*}
    Therefore, it suffices to show that
        \[
            a_k^{(2i - 2)L} b^{(r-i) L} p^{\left(2\binom{r}{2} - \binom{r-i+2}{2}\right)L} \le n^{C_r} \left(a_k^{2k+2} b^{r-k-2} p^{2\binom{r}{2} - \binom{r-k}{2}}\right)^{L},
        \]
        which is equivalent to
        \[
            1 \le n^{C_r/L} a_k^{2k-2i+4} b^{i-k-2} p^{\binom{r-i+2}{2} - \binom{r-k}{2}}.
        \]
        Since $2 \le i \le k+2$ and since the right-hand side above is a convex function in $i$, it is minimised at $i \in \{2, k+2\}$. If $i = k+2$, the inequality is trivial as it is equivalent to $1 \le n^{C_r/L}$. If $i = 2$ we have 
        \[
             n^{C_r/L} a_k^{2k} b^{-k} p^{\binom{r}{2} - \binom{r-k}{2}} \ge n^{C_r/L} a_k^{k} p^{\binom{r}{2} - \binom{r-k}{2}} \ge 1,
        \]
        where the first inequality holds as $a_k\ge b$ and the second inequality holds due to \eqref{align:a_k inequality to use in hm}. This concludes the proof in this case.
        
    \noindent\textit{Case 2:} Assume that $n_A = i L$ and $n_B = j L$.\\
     As $i + j = r$, we have $f  = \binom{r}{2}L$, and hence, it remains to show that
        \[
            a_k^{i} b^{r-i} p^{\binom{r}{2}} \le n^{C_r/L} a_k^{2k+2} b^{r-k-2} p^{2\binom{r}{2} - \binom{r-k}{2}}.
        \]
        This was proven in Case 3 of \ref{n_A and n_B large}, see \eqref{align:ineq4 used in hm}.
    
    \noindent\textit{Case 3:} Assume that $n_A = (2i-2)L$ and $n_B = j$.\\
    In this case, again, by using $i+j=r$ we have
        \begin{align*}
            f= \left(\binom{r}{2} - \binom{r-i}{2} + \binom{r}{2} - \binom{r-i+2}{2}\right)L + \binom{j}{2}.
        \end{align*}
        Our goal is to verify that
        \[
            \left(a_k^{2i-2} p^{\binom{r}{2} - \binom{r-i}{2} + \binom{r}{2} - \binom{r-i+2}{2}}\right)^{L} b^{j}  p^{\binom{j}{2}} \le n^{C_r} \left(a_k^{2k+2} b^{r-k-2} p^{2 \binom{r}{2} - \binom{r-k}{2}}\right)^{L}.
        \]
        As $b \le n$, we have $b^{j} p^{\binom{j}{2}} \le n^{ C_r}$ provided that $C_r \ge j$, and hence, it is enough to show that
        \begin{equation}\label{eq:claim2Case3}
            a_k^{2i - 2} p^{- \binom{r-i}{2} - \binom{r-i+2}{2}} \le a_k^{2k+2} b^{r-k-2} p^{- \binom{r-k}{2}}.
        \end{equation}    
        Since $2 \le i \le k+2$ and since the left-hand side above is a convex function in $i$, it is maximised at $i \in \{2, k+2\}$.
        For $i = 2$, \eqref{eq:claim2Case3} follows by using the above and \eqref{align:a_k inequality to use in hm} in the following manner. 
        First, by combining the two inequalities in \eqref{align:a_k inequality to use in hm}, we have
        \begin{align*}
            1 &\le a_k^k p^{\binom{r}{2} - \binom{r-k}{2}} b^{r-k-2} p^{\binom{r-k}{2}-1} = a_k^k b^{r-k-2} p^{\binom{r}{2}-1} \le a_k^k b^{r-k-2} p^{\binom{r-2}{2}}.
        \end{align*}
        This, and the former inequality in \eqref{align:a_k inequality to use in hm} yield the required bound:
        \begin{equation}\label{eq:claim8.5-case3} 
        \begin{aligned}
        a_k^2 p^{-\binom{r-2}{2} - \binom{r}{2}} &\le 
        \left(a_k^k b^{r-k-2} p^{\binom{r-2}{2}}\right) \cdot \left(a_k^k p^{\binom{r}{2} - \binom{r-k}{2}}\right) \cdot a_k^2 p^{-\binom{r-2}{2} - \binom{r}{2}} \\
            &= a_k^{2k+2} b^{r-k-2} p^{-\binom{r-k}{2}}.
        \end{aligned}
        \end{equation}
        For $i = k+2$, by \eqref{align:a_k inequality to use in hm} we have, 
        \begin{align}\label{eq:bound on b^ p^}
            1 \le b^{r-k-2} p^{\binom{r-k}{2}-1} \le b^{r-k-2} p^{\binom{r-k-2}{2}}.
        \end{align}
        This then implies that
        \begin{align*}
            a_k^{2k+2} p^{-\binom{r-k-2}{2} - \binom{r-k}{2}} &\le\left(b^{r-k-2} p^{\binom{r-k-2}{2}}\right) \cdot a_k^{2k+2} p^{-\binom{r-k-2}{2} - \binom{r-k}{2}} \\&= a_k^{2k+2} b^{r-k-2} p^{- \binom{r-k}{2}}.
        \end{align*}       
        This proves \eqref{eq:claim2Case3}, and concludes the proof in this case.

    \noindent\textit{Case 4:} Assume that $n_A = i L$ and $n_B = j$.\\
    In this case, we have
        \begin{align*}
            f= \left(\binom{r}{2} - \binom{r-i}{2}\right)L + \binom{j}{2}.
        \end{align*}
    Hence, as $b \le n$ and as we may take $C_r$ to be sufficiently large, it is enough to show that
        \[
            a_k^{i} p^{\binom{r}{2} - \binom{r-i}{2}} \le a_k^{2k+2} b^{r-k-2} p^{2\binom{r}{2} - \binom{r-k}{2}}.
        \]
        Since $2 \le i \le k+2$ and since the left-hand side above is a convex function in $i$, it is maximised at $i \in \{2, k+2\}$. If $i = 2$ we have \eqref{eq:claim8.5-case3}, and if $i = k+2$, we have
        \begin{align*}
            a_k^{k+2} p^{-\binom{r-k-2}{2}} &\le \left(b^{r-k-2} p^{\binom{r-k-2}{2}}\right) \cdot \left(a_k^k p^{\binom{r}{2} - \binom{r-k}{2}}\right) \cdot a_k^{k+2} p^{-\binom{r-k-2}{2}} \\&= a_k^{2k+2} b^{r-k-2} p^{\binom{r}{2} - \binom{r-k}{2}},
        \end{align*}
        where the inequality holds by the former inequality in \eqref{align:a_k inequality to use in hm} and by \eqref{eq:bound on b^ p^}. This completes the proof of this case and thereby the proof of the claim.  

\subsection{Proof of \ref{n_A small and n_B large}}
    Suppose that $n_A  \in \{i, i L\}$ and $n_B \in \{j L, (2j - (r-k-2))L \}$ maximise 
    \[
        a_k^{n_A} b^{n_B} p^{n_B \cdot i + \frac{n_A(i - 1)}{2} + \frac{n_B(j - 1) + (r-k-2)(n_B- j L)}{2}}.
    \] 
    First, observe that if $a_k,b,p$ and $n_B$ are fixed, then the above is an increasing function of $n_A$ provided that $\eps_k$ is small enough. Indeed, we have
    \[
        a_k p^{\frac{i - 1}{2}} \ge a_k p^{\frac{r-1}{2}} = \left(a_k^{r} p^{\binom{r}{2}}\right)^{1/r} = \left(n^{r-2} p^{\binom{r}{2} - 1} \cdot n^2 p \cdot n^{-r \epsilon_k}\right)^{1/r} \ge \left(n^{2 - r \epsilon_k} p\right)^{1/r} \ge 1,
    \]
    where in the last inequality we assume that $\epsilon_k$ is sufficiently small. Therefore, we may assume that $n_A = i L$. Let us distinguish two cases according to the value of $n_B$.
    
    \noindent\textit{Case 1:} Assume that $n_B = j L$.\\
    As $i + j = r$, we have
        \begin{align*}
            n_B \cdot i + \frac{n_A(i - 1)}{2} + \frac{n_B(j - 1)}{2} + \frac{(r-k-2)(n_B-jL)}{2}
            = \left(i \cdot j + \binom{i}{2} + \binom{j}{2}\right)L 
            = \binom{r}{2}L.
        \end{align*}
    We then have
        \[
            a_k^{n_A} b^{n_B} p^{n_B \cdot i + \frac{n_A(i - 1)}{2} + \frac{n_B(j - 1)}{2} + \frac{(r-k-2)(n_B-jL)}{2}} = \left(a_k^i b^{r-i} p^{\binom{r}{2}}\right)^{L}\le \left(a_k^{2k+2} b^{r-k-2} p^{2\binom{r}{2} - \binom{r-k}{2}}\right)^{L},
        \]
    which follows from Case 3 of the proof of \ref{n_A and n_B large}, see \eqref{align:ineq4 used in hm}.

    \noindent\textit{Case 2:} Assume that $n_B = (2j - (r-k-2))L$.\\
    Similarly to the previous case, since $i+j=r$ we have
        \begin{align*}
            n_B \cdot i + \frac{n_A(i - 1)}{2} + \frac{n_B(j - 1) + (r-k-2)(n_B- j L)}{2} = \left(2 \binom{r}{2} - \binom{i + r - k - 2}{2}\right)L.
        \end{align*}
        Hence, it remains to show that       
        \[
            a_k^{i} b^{2(r-i) - (r-k-2)} p^{2 \binom{r}{2} - \binom{i + r - k - 2}{2}} \le a_k^{2k+2} b^{r-k-2} p^{2\binom{r}{2} - \binom{r-k}{2}}.
        \]
        Since $2 \le i \le k+2$ and since the left-hand side above is a convex function in $i$, it is maximised at $i \in \{2, k+2\}$. If $i = 2$, recall that $b \le a_k$ which readily implies the required inequality:
        \begin{align*}
             a_k^2 b^{r + k - 2} p^{2 \binom{r}{2} - \binom{ r - k}{2}} &\le a_k^{2k + 2} b^{r - k - 2} p^{2 \binom{r}{2} - \binom{ r - k}{2}}.
        \end{align*}                 
        If $i = k+2$, we have
        \[
            a_k^{k+2} b^{r-k-2} p^{\binom{r}{2}} \le \left(a_k^k p^{\binom{r}{2} - \binom{r-k}{2}}\right) \cdot a_k^{k+2} b^{r-k-2} p^{\binom{r}{2}} =a_k^{2k+2} b^{r-k-2} p^{2\binom{r}{2} - \binom{r-k}{2}},
        \]
        where the inequality follows from \eqref{align:a_k inequality to use in hm}, and the proof is concluded.

\subsection{Proof of \ref{n_A and n_B small}}
    Suppose that $i\le n_A \leq i L$ and $j\le n_B \le j L$ maximise 
    \begin{equation}\label{eq:claim4}
        f\coloneqq a_k^{n_A} b^{n_B} p^{\max\{n_A\cdot j,n_B \cdot i\} + \frac{n_A(i - 1)}{2} + \frac{n_B(j - 1)}{2}}.
    \end{equation}
    We first eliminate the maximum above by splitting the analysis into two cases.
    
    \noindent\textit{Case 1:} Assume that $n_A\cdot j\leq n_B \cdot i$.\\
    As in the proof of \ref{n_A small and n_B large}, upon fixing $a_k,b,i,j,p,$ and $n_B$, the function in \eqref{eq:claim4} is increasing in $n_A$. Therefore, we may assume in this case that $n_A \cdot j= n_B\cdot i$, where here, $n_A$ may be non-integer. In particular, as $j\le n_B\le jL$, we have
    \begin{equation}\label{eq:Claim4-1}
        f= \left(a_k^{i} b^{j} p^{ ij + \frac{i(i - 1)}{2} + \frac{j(j - 1)}{2}}\right)^{n_B/j}= \left(a_k^{i} b^{j} p^{ \binom{r}{2}}\right)^{n_B/j}\le \max\left\{a_k^{i} b^{j} p^{ \binom{r}{2}},\left(a_k^{i} b^{j} p^{\binom{r}{2}}\right)^{L}\right\}.
    \end{equation}

    By \eqref{align:a_k inequality to use in hm} and the fact that $a_k^2 p \ge 1$ which holds provided that $\epsilon_k$ is sufficiently small, we have,
    \[
        1 \le \left(a_k^k p^{\binom{r}{2}-\binom{r-k}{2}}\right)^2 \left(b^{r-k-2} \cdot p^{\binom{r-k}{2}-1}\right) \left(a_k^2 p\right) = a_k^{2k+2} b^{r-k-2} p^{2\binom{r}{2} - \binom{r-k}{2}}
    \]
    Then, it is clear that by taking $C_r$ large enough 
    \begin{equation}\label{eq:Claim4-2}
        a_k^{i} b^{j} p^{ ij + \frac{i(i - 1)}{2} + \frac{j(j - 1)}{2}} \leq n^{C_r}\leq n^{C_r} \left(a_k^{2k+2} b^{r-k-2} p^{2\binom{r}{2} - \binom{r-k}{2}}\right)^{L}.
    \end{equation}
    Moreover, as $b\leq a_k$ and as $i\le k+2$ we have
    \begin{equation}\label{eq:Claim4-3}
        a_k^{i} b^{j} p^{ ij + \frac{i(i - 1)}{2} + \frac{j(j - 1)}{2}} =a_k^{i} b^{j} p^{ \binom{r}{2}}\le a_k^{k+2} b^{r-k-2} p^{ \binom{r}{2}}\le a_k^{2k+2} b^{r-k-2} p^{ 2\binom{r}{2}-\binom{r-k}{2}},
    \end{equation}
    where the last inequality holds by the first inequality in \eqref{align:a_k inequality to use in hm}. By combining \eqref{eq:Claim4-1}, \eqref{eq:Claim4-2}, and \eqref{eq:Claim4-3} the proof in this case is complete.

    \noindent\textit{Case 2:} Assume that $n_B\cdot i\leq n_A \cdot j$.\\
        
    In this case, we have
    \[
        f= a_k^{n_A} b^{n_B} p^{n_A\cdot j + \frac{n_A(i - 1)}{2} + \frac{n_B(j - 1)}{2}}.
    \]
    As in the previous cases, the maximum of $f$ is attained at the boundary values of $n_A$ and $n_B$.

    Observe that if the maximum of $f$ occurs at $n_A = i$, we must have $n_B = j$. This yields $n_A \cdot j = n_B \cdot i$, which is already covered by the first case. Consequently, we may assume that $f$ is maximised when $n_A = iL$. Furthermore, if $n_B = jL$, we again obtain $n_A \cdot j = n_B \cdot i$, which is similarly covered by the first case.
    
    Suppose then that $f$ is maximized at $n_A = iL$ and $n_B = j$. Under these conditions, we have
    \[
        f = \left(a_k^{i} p^{ij + \binom{i}{2}}\right)^{L} \cdot b^{j} p^{\binom{j}{2}} = \left(a_k^{i} p^{\binom{r}{2} - \binom{r-i}{2}}\right)^{L} \cdot b^{j} p^{\binom{j}{2}},
    \]
    and this case can be completed in the same manner as Case 4 in the proof of \ref{n_A large n_B small}.

\end{document}